\documentclass[10pt,reqno]{amsart}
\usepackage{amssymb, hyperref}
\usepackage{color}

\def\normo#1{\left\|#1\right\|}
\def\normb#1{\Big\|#1\Big\|}
\def\norm#1{\|#1\|}
\def\bra#1{\langle#1\rangle}
\def\wt#1{\widetilde{#1}}
\def\wh#1{\widehat{#1}}
\def\set#1{\{#1\}}

\newcommand{\T}{{\mathbb T}}
\newcommand{\R}{{\mathbb R}}
\newcommand{\Z}{{\mathbb Z}}
\newcommand{\ft}{{\mathcal{F}}}

\newcommand{\N}{{\mathcal{N}}}
\newcommand{\NT}{\mathcal{NT}}

\newcommand{\Sch}{{\mathcal{S}}}

\newcommand{\noi}{{\noindent}}

\numberwithin{equation}{section}
\newtheorem{theorem}{Theorem}[section]
\newtheorem{proposition}[theorem]{Proposition}
\newtheorem{lemma}[theorem]{Lemma}
\newtheorem{corollary}[theorem]{Corollary}
\newtheorem{remark}[theorem]{Remark}

\newcommand{\px}{\partial_x}
\newcommand{\pt}{\partial_t}

\begin{document}
\title[Fifth-order modified KdV equation]{Low regularity Cauchy problem for the fifth-order modified KdV equations on $\T$}

\author[C. Kwak]{Chulkwang Kwak}

\address{Facultad de Matem\'aticas, Pontificia Universidad Cat\'olica de Chile, Campus San Joaqu\'in. Avda. Vicu\~na Mackenna 4860, Santiago, Chile}
\email{chkwak@mat.uc.cl}

\begin{abstract}

In this paper, we consider the fifth-order modified Korteweg-de Vries (modified KdV) equation under the periodic boundary condition. We prove the local well-posedness in $H^s(\mathbb T)$, $s > 2$, via the energy method. The main tool is the short-time Fourier restriction norm method, which was first introduced in its current form by Ionescu, Kenig and Tataru [Global well-posedness of the KP-I initial-value problem in the energy space, Invent. Math. 173 (2) (2008) 265--304]. Besides, we use the frequency localized modified energy to control the \emph{high-low} interaction component in the energy estimate. We remark that under the periodic setting, the integrable structure is very useful (but not necessary) to remove harmful terms in the nonlinearity and this work is the first low regularity well-posedness result for the fifth-order modified KdV equation.

\end{abstract}

\thanks{}
\thanks{} \subjclass[2010]{35Q53, 37K10} \keywords{the fifth-order modified KdV equation; local well-posedness; complete integrability; short time $X^{s,b}$ space; nonlinear transformation; modified energy.}
%35Q53 : KdV-like equations
%37K10 : Completely integrable systems, integrability tests, bi-Hamiltonian structures, hierarchies (KdV, KP, Toda, etc.)
\maketitle

\tableofcontents

\section{Introduction}\label{sec:intro}

It is known that the spectrum of the Hill operator
\[L(t) = \frac{d^2}{dx^2} - u(t,x)\]
on the interval $[0,4\pi]$, which is of twice the length of the period of $u$, is independent of $t$, so the periodic eigenvalues are conserved quantities when the potential $u(t,x)$ evolves according to the Korteweg-de Vries (KdV) equation
\[u_t + u_{xxx} = 6uu_x.\]
This property was first discovered by Gardner, Greene, Kruskal and Miura \cite{GGKM1967}. By considering this property, Lax \cite{Lax1968} found the family of equations (known by the name KdV hierarchy). This observation opened the theory of complete integrable systems of PDEs, and the KdV equation and its hierarchy hold a central example among other integrable systems. Lax also represented those equations by the form (known as the \emph{Lax pair formulation})
\[\pt L = A_n L - LA_n = [A_n; L].\]
Here $L$ is defined as above, and $A_n$ is the differential operator in the form
\[A_n = 4^n \px^{2n+1} + \sum_{j=1}^n \{a_{nj}\px^{2j-1} + \px^{2j-1} a_{nj}\}, \hspace{1em} n=0,1,\cdots,\]
where $A_0 = \px$ and the coefficient $a_{nj} = a_{nj}(u)$ be chosen such that the operator $[A_n;L]$ has order zero. Thereafter, Magri \cite{Magri1978} realized that this complete integrable system has an additional structure, the so-called bi-Hamiltonian structure, and then further studies on the theory of the complete integrability have been widely progressed by several researchers (see, for instance, \cite{AS1981, KP2003} and references therein). 

On the other hand, the Miura transformation \cite{Miura1968}
\[u = v_x + v^2\]
transforms solutions to all equations in the defocussing modified KdV hierarchy to solutions of equations in the KdV hierarchy. Thanks to this observation, we obtain the modified KdV hierarchy \eqref{eq:hamiltonian} from the KdV hierarchy. Like the equations in the KdV hierarchy, all equations in the modified KdV hierarchy satisfy the property of the complete integrability. Recently, Choudhuri, Talukdar and Das \cite{CTD2009} derived the Lax representation and constructed the bi-Hamiltonian structure of the equations in the modified KdV hierarchy. The followings are a few equations and their associated Hamiltonians with respect to bi-Hamiltonian structures:
\begin{equation}\label{eq:hamiltonian}
\begin{split}
\pt u - \px u = 0 \hspace{0.5em}&\hspace{0.5em} \int \frac12 u^2\\
\pt u - \px^3 u + 6u^2\px u = 0 \hspace{0.5em}&\hspace{0.5em} \int \frac12 u_x^2 + \frac12 u^4\\
\pt u - \px^5 u + 40u\px u \px^2u + 10 u^2\px^3u + 10(\px u)^3 - 30u^4 \px u = 0 \hspace{0.5em}&\hspace{0.5em} \int \frac12 u_{xx}^2 +u^6 + 5u^2u_x^2\\
\vdots \hspace{0.5em}&\hspace{0.5em} \vdots\\
\end{split}
\end{equation} 

In this paper, we consider the integrable fifth-order equation in the modified KdV hierarchy \eqref{eq:hamiltonian}:
\begin{equation}\label{eq:5mkdv}
\begin{cases}
\pt u - \px^5 u + 40u\px u \px^2u + 10 u^2\px^3u + 10(\px u)^3 - 30u^4 \px u = 0, \hspace{1em} (t,x) \in \R \times \T, \\
u(0,x) = u_0(x) \in H^s(\T),
\end{cases}
\end{equation}
where $\T = [0,2\pi]$.

Due to the theory of the complete integrability (or the inverse spectral method), the global solution exists to any equation in \eqref{eq:hamiltonian} for any Schwartz initial data. Kappeler and Topalov \cite{KT2006}, \cite{KT2005} proved the global well-posedness of KdV and modified KdV equations in $H^s(\T)$ for $s \ge -1$ and $s \ge 0$, respectively. It is a fine achievement to solve the low regularity well-posedness problem via only the theory of the complete integrability and the property of the Miura transform. Nevertheless, the analytic theory of nonlinear dispersive equations is still required to study on the low regularity well-posedness problem even for integrable equations. In fact, in a number of previous studies on the low regularity local theory for nonlinear dispersive equations (especially, under the non-periodic setting), the integrable structures were ignored. The purpose of this work, however, is to point out that, under the periodic setting, the integrable structure may be helpful (only two low level conservation laws for the fifth-order equation are useful, but much more high level conservation laws may be needed for higher-order equations in \eqref{eq:hamiltonian}) to study the low regularity well-posedness problem in contrast to the non-periodic problem. 

Generalizing coefficients in the nonlinear terms may break the integrable structure. The following equation generalizes \eqref{eq:5mkdv} to non-integrable case:
\begin{equation}\label{eq:5mkdv5}
\begin{cases}
\pt u - \px^5 u + a_1u\px u \px^2u + a_2 u^2\px^3u + a_3(\px u)^3 + a_4u^4 \px u = 0, \hspace{1em} (t,x) \in \R \times \T, \\
u(0,x) = u_0(x) \in H^s(\T),
\end{cases}
\end{equation}
where $a_i$'s, $i=1,2,3,4$, are real constants. For studying \eqref{eq:5mkdv5}, we can rely no longer on the property of the complete integrability.

Meanwhile, the Fourier coefficients of both \eqref{eq:5mkdv} and \eqref{eq:5mkdv5} (see \eqref{eq:5mkdv1} below) in terms of the spatial variables reveal that the non-trivial resonant terms of the form (with some coefficients)
\begin{equation}\label{eq:linear like}
\norm{u(t)}_{L^2}^2 \px^3 u, \hspace{1em} \norm{u(t)}_{\dot{H}^1}^2\px u, \hspace{1em} \norm{u(t)}_{L^4}^4\px u
\end{equation}
appear in the nonlinear term due to \eqref{eq:resonant case1} and \eqref{eq:resonant case2}. We call them the \emph{linear-like} resonant terms. The \emph{linear-like} resonant terms \eqref{eq:linear like} are unfavorable, by themselves, as perturbations of the linear evolution in the low regularity Sobolev spaces. However, the complete integrability of \eqref{eq:5mkdv}, which, in particular, is the fact that \eqref{eq:5mkdv} enjoys first two conservation laws in \eqref{eq:hamiltonian}, converts most of \eqref{eq:linear like} in \eqref{eq:5mkdv} into
\[c_1 \px^3 u+c_2 \px u,\]
for some constants $c_1,c_2 \ge 0$\footnote{The constants $c_1, c_2$ depend on the first two conservation laws in \eqref{eq:hamiltonian}.} and hence the linear part of \eqref{eq:5mkdv} can be expressed as 
\[[\px^5 + c_1\px^3 +c_2\px] u.\]

From this observation, we point out, as a purpose of this work, that the use of integrable structure (only the use of conservation laws) may be necessary to study the low regularity well-posedness problem under the periodic setting in contrast to the non-periodic problem. Remark that it, in fact, is not necessary to use the integrable structure for this problem thanks to the suitable nonlinear transformation, but it should be necessary for higher-order equations, due to the higher degree nonlinearity with derivatives, see Remark \ref{rem:full nonlinear transform}.

The following is the main result in this paper:
\begin{theorem}\label{thm:main}
Let $s > 2$. For any $u_0 \in H^s(\T)$ satisfying
\begin{equation}\label{eq:level set}
\int_{\T} (u_0(x))^2 \; dx = \gamma_1, \hspace{2em} \int_{\T} (\px u_0(x))^2 + (u_0(x))^4 \; dx = \gamma_2
\end{equation}
for some $\gamma_1, \gamma_2 \ge 0$, there exists $T=T(\norm{u_0}_{H^s})>0$ such that \eqref{eq:5mkdv} has a unique solution on $[-T,T]$ satisfying
\begin{align}
&u(t,x) \in C([-T,T];H^s(\T)),\nonumber\\
&\eta(t)\sum_{n \in \Z} e^{i(nx - 20n\int_0^t \norm{u(s)}_{L^4}^4 \; ds)}\wh{u}(t,n) \in C([-T,T];H^s(\T)) \cap F^s(T), \label{eq:result}
\end{align}
where $\eta$ is any cut-off function in $C^{\infty}(\R)$ with $\mathrm{supp}\eta \subset [-T,T]$ and the space $F^s(T)$\footnote{The space $F^s$ also depends on the initial data $u_0$ with \eqref{eq:level set}.} will be defined later. Moreover, the flow map $S_T : H^s \to C([-T,T];H^s(\T))$ is continuous on the level set in $H^s$ satisfying \eqref{eq:level set}.
\end{theorem}

\begin{remark}
Combining the property of the Miura transform and Theorem \ref{thm:main}, the local well-posedness result for the fifth-order KdV equation \cite{Kwak2016} could be improved for $s>1$. However, the Cauchy problem for the fifth-order KdV equation in $H^2(\T)$ below by using the analytic theory is still open and technical development is required to approach to this result. 
\end{remark}

\begin{remark}
The regularity threshold $s = 2$ couldn't be achieved by our analysis, due to the logarithmic divergence of frequency-summations in the energy estimate (see, in particular, the proof of Proposition \ref{prop:energy1-2}).
\end{remark}

The use of conservation laws is not sufficient to deal with all \emph{linear-like} resonant terms in \eqref{eq:linear like}, so one of \emph{linear-like} resonant components (of the form $\norm{u}_{L^4}^4\px u$) still remains, even if \eqref{eq:5mkdv} enjoys infinitely many conservation laws. However, the introduction of an appropriate nonlinear transformation (see \eqref{eq:modified solution}) facilitates the elimination of the rest of \emph{linear-like} resonant terms (of the form $\norm{u}_{L^4}^4\px u$) in the nonlinearity (see \eqref{eq:result} in Theorem \ref{thm:main}). This nonlinear transformation, in particular defined in \eqref{eq:modified solution}, has a good property that the transformation is invertible and bi-continuous from the ball in $C([-T,T];H^s)$ to itself for $s \ge \frac14$. Staffilani \cite{Staffilani1997} introduced such a \emph{Gauge transformation} for the generalized KdV equations and showed the bi-continuity property of the transform for $s > 1/2$. See also \cite{CKSTT2004} for its application.  

On the other hand, Theorem \ref{thm:main} can be extended to the local well-posedness for fully non-integrable equations by defining the fully nonlinear transformation involving all resonances in the nonlinearity, which can be shown to be bi-continuous for $s \ge 1$. See Remark \ref{rem:full nonlinear transform} in Section \ref{sec:main} for this issue.\footnote{However, higher-order equations in the hierarchy may have much more resonances in the nonlinearity which cannot be controlled by only the nonlinear transform, and hence we guess that the study on higher-order equations in the hierarchy shall depend on the integrable structure.} The following is the extension result of Theorem \ref{thm:main} for the non-integrable equation \eqref{eq:5mkdv5}:
\begin{theorem}\label{thm:nonintegrable}
Let $s > 2$. Then, \eqref{eq:5mkdv5} is locally well-posed in $H^s(\T)$.\footnote{Thanks to the bi-continuity of the nonlinear transform, no additional restriction of initial data like \eqref{eq:level set} is needed.}
\end{theorem}

The proof of Theorem \ref{thm:nonintegrable} is based on the energy method, and hence the existence of a high regularity (or smooth) solution to \eqref{eq:5mkdv5} is required, since the local well-posedness of \eqref{eq:5mkdv5} has not been studied for any regularity. The high regularity well-posedness of \eqref{eq:5mkdv5} can be obtained by following the standard way: the parabolic regularization argument in \cite{BS1978}. Precisely, we first show an \emph{a priori} bound for the solution $u$ to the $\varepsilon$-parabolic equation, and in addition to the bootstrap argument, the approximation method yields that the solution to $\varepsilon$-parabolic equation converges to the solution to the fifth-order modified KdV equation. Finally, the (unconditional) local well-posedness of \eqref{eq:5mkdv5} for $s > 7/2$ can could be obtained by the direct energy estimate (see Appendix \ref{app:smooth sol}). The main difficulty is to obtain the energy estimate for both the parabolic and the fifth-order modified KdV equations due to the strong nonlinearity. Nevertheless, the use of the modified energy \eqref{eq:kwon energy}, which was first introduced in its form by Kwon \cite{Kwon2008-1} for the fifth-order KdV equation on $\R$, in addition to the Kato-Ponce type commutator estimate and the Sobolev embedding, enables us to obtain the energy bound of the solution $u$. We sketch the proof of the well-posedness for high regularity data in Appendix \ref{app:smooth sol} for the convenience of readers. The rest of the proof of Theorem \ref{thm:nonintegrable} (LWP in the low regularity regime) follows the similar argument of the proof of Theorem \ref{thm:main}, since the only difference is to deal with non-trivial resonances in the nonlinearity.\footnote{Nevertheless, the reason why we consider the integrable equation \eqref{eq:5mkdv} instead of \eqref{eq:5mkdv5} is to raise the importance of the integrable structure under the periodic setting even for the low regularity problems.}

The fifth-order modified KdV equation under the non-periodic setting has been studied by Linares \cite{Linares1995}. Linares used the dispersive smoothing effect \cite{KPV1991} in order to prove the local well-posedness of the fifth-order modified KdV equation in $H^2(\R)$ (hence the global well-posedness via the conservation law). Later, this was improved by Kwon \cite{Kwon2008-2}. Kwon used the standard $X^{s,b}$ space (Fourier restrict norm method) and Tao's $[k,Z]$-multiplier norm method \cite{Tao2001} to prove a trilinear estimate, and hence obtained the local well-posedness in $H^s$ for $s \ge \frac34$ via the contraction mapping principle. In contrast with the non-periodic setting, the trilinear estimate in the $X^{s,b}$ space fails under the periodic boundary condition. As a minor result in this paper, we have 
\begin{theorem}\label{thm:trilinear}
For any $s,b \in \R$, the trilinear estimate
\[\norm{uv\px^3w}_{X_{\tau-n^5}^{s,b-1}} \le C\norm{u}_{X_{\tau-n^5}^{s,b}}\norm{v}_{X_{\tau-n^5}^{s,b}}\norm{w}_{X_{\tau-n^5}^{s,b}}\]
fails.
\end{theorem}

The counter-example involves the \emph{high $\times$ low $\times$ low $\Rightarrow$ high} interaction component along the non-resonant phenomenon of the following type:
\[(P_{low}u)\cdot(P_{low}v)\cdot (P_{high}w_{xxx}).\]

The fifth-order modified KdV evolution provides quite strong modulation effect in the nonlinear interaction, but it is not enough to control three derivatives in the high frequency mode. Hence one cannot obtain the trilinear estimate in the standard $X^{s,b}$ space. This observation gives a clue that the flow map of \eqref{eq:5mkdv} (also \eqref{eq:5mkdv5}) seems not uniformly continuous, that is, the Picard iteration method does not work in this problem. It is remarkable that the solution flow of the fifth-order modified KdV equation behaves in the periodic setting strictly worse than in the non-periodic setting due to the lack of dispersive smoothing effect. The detailed example will be given in Section \ref{sec:trilinear}, later. 

So far, we investigated the principal enemies to study the fifth-order modified KdV equation on $\T$: \emph{linear-like} resonant terms and the failure of the trilinear estimate. As mentioned, the integrable structure and the use of the nonlinear transformation \eqref{eq:modified solution} enable us to resolve the first obstacle. The modified $X^{s,b}$ structure in a short time interval ($\approx \mbox{(frequency)}^{-2}$), the second difficulty could be resolved. The short time structure used in this paper was introduced in its current form by Ionescu, Kenig and Tataru \cite{IKT2008} in the context of KP-I equation. Also, we refer to \cite{KT2007, CCT2008} for different formulas of short time structures. See \cite{Guo2011, GPWW2011, Guo2012, GKK2013, KP2015, GO2015} for its applications.

Now we introduce ingredients in this paper. To prove Theorem \ref{thm:main}, we need to obtain the following estimates:
\begin{equation}\label{eq:brief proof}
\begin{cases}
\begin{array}{ll}
\norm{v}_{F^s(T)} \lesssim \norm{v}_{E^s(T)} + \norm{\mathcal{N}(v)}_{N^s(T)} &\mbox{(Linear)}\\
\norm{\mathcal{N}(v)}_{N^s(T)} \lesssim \norm{v}_{F^s(T)}^3 + \norm{v}_{F^s(T)}^5 &\mbox{(Nonlinear)}\\
\norm{v}_{E^s(T)}^2 \lesssim (1+\norm{v_0}_{H^s}^2)\norm{v_0}_{H^s}^2 + (1+\norm{v}_{F^s(T)}^2+\norm{v}_{F^s(T)}^4)\norm{v}_{F^s(T)}^4 &\mbox{(Energy)}
\end{array}
\end{cases}
\end{equation}
where $v$ is transferred solution by the nonlinear transformation \eqref{eq:modified solution}. Then, thanks to the continuity argument, one can obtain \emph{a priori} bound of solutions to \eqref{eq:5mkdv} at first. To complete the local well-posedness argument, one needs to obtain similar estimates as in \eqref{eq:brief proof} for the difference of two solutions as well. However, the energy estimate for the difference of two solutions cannot be obtained in only $F^s$ space due to the lack of the symmetry among two solutions, and hence the Bona-Smith argument (energy estimate in the intersection of the weaker ($F^0$) and the stronger ($F^{2s}$) spaces) is essential to close the energy estimate.

On the other hand, in order to obtain the second estimation in \eqref{eq:brief proof}, the $L^2$-block estimates (Lemma \ref{lem:tri-L2}) is used. In contrast with the non-periodic setting, this-type estimate under the periodic setting is weaker due to the lack of smoothing effect, in mathematical sense that the (counter) measure of frequency support must be bigger than $1$. Hence, no additional smoothing effect from the very \emph{high-low} interaction components in the low modulation case is expected. Nevertheless, $X^{s,b}$ spaces taken in the short time length $(\approx 2^{-2k})$ at the $2^k$-frequency piece prevent the modulation to be low, and hence it gives another smoothing effect (in the sense of an advantage of the low bound of the modulation effect, which is two derivative gains: $|\tau - \mu(n)| \gtrsim 2^{2k}$). This shows that the short time structure well adapted to study the Cauchy problem for \eqref{eq:5mkdv} and this time scale is suitable to recover the lack of the dispersive smoothing effect in the periodic problem.

Moreover, as a trade-off of choosing the short-time scale $(\mbox{frequency})^{-2}$, we cannot obtain the energy estimate in \eqref{eq:brief proof} by direct calculation. Indeed, in view of the following simple calculations
\[\begin{aligned}
\sum_{0\le k_1 ,k_2 \le k-10}\Big|\sum_{n,\overline{\N}_{3,n}}\int_0^{t_k} \chi_k(n)n[\chi_{k_1}(n_1)\wh{v}(n_1)&\chi_{k_2}(n_2)\wh{v}(n_2)n_3^2\wh{v}(n_3)]\chi_k(n)\wh{v}(n) \;dt\Big|\\
&\lesssim \norm{v}_{F^{\frac14+}(T)}^2\sum_{|k-k'| \le 5}2^{2k}\norm{P_{k'}v}_{F_{k'}(T)}^2
\end{aligned}\]
and
\[\begin{aligned}
\sum_{\substack{0\le k_2  \le k-10 \\ 0\le k_2  \le k-10}}\Big|\sum_{n,\overline{\N}_{3,n}}\int_0^{t_k} \chi_k(n)&[(n_1+n_2)\chi_{k_1}(n_1)\wh{v}(n_1)\chi_{k_2}(n_2)\wh{v}(n_2)n_3^2\wh{v}(n_3)]\chi_k(n)\wh{v}(n)\; dt\Big|\\
&\hspace{7em}\lesssim \norm{v}_{F^{\frac14+}(T)}^2\sum_{|k-k'| \le 5}2^{2k}\norm{P_{k'}v}_{F_{k'}(T)}^2,
\end{aligned}\]
the (two) derivative loss arises in the \emph{high $\times$ low $\times$ low $\Rightarrow$ high} interaction component, when three derivatives are taken in the high frequency mode, which is unfavorable in even short time $F^s$-norm. In order to recover this derivative loss, we introduce the frequency localized modified energy \eqref{eq:new energy1-1}, which plays a role of the transformation, that moves two derivatives from the high frequency mode to the low frequency mode, and thus the \emph{high $\times$ low $\times$ low $\Rightarrow$ high} interaction component can be controlled. Another issue in the energy estimate (as a mathematical difficulty) is exact quintic resonant terms arising in the correction terms. However, one can observe that they should vanish thanks to the symmetry among frequencies and functions (see Remarks \ref{rem:resonant1} and \ref{rem:resonant2} in Section \ref{sec:energy}). Note that the modified energy plays consequentially a similar role as an additional weight under the non-periodic condition, while an additional weight does not work under the periodic setting. See \cite{GKK2013} and \cite{KP2015} for a comparison.   

\begin{remark}
The similar perspective and argument in this work are applicable to the study on the Cauchy problem for the fifth-order KdV equation on $\T$. See \cite{Kwak2016}.
\end{remark}

There are several works on the low regularity well-posedness problem of the fifth-order dispersive equations with similar nonlinearity. For instance, the fifth-order KdV equation on $\R$, which has the nonlinearity of the form of $c_1\px u \px^2 u + c_2 u\px^3u$, was first studied by Ponce \cite{Ponce1993} as the low regularity well-posedness problem. Since the strength of the nonlinearity is stronger than the advantage from the dispersive smoothing effect, the energy method is required to prove the local well-posedness. Ponce used the energy method to prove the local well-posedness for Sobolev initial data $u_0 \in H^s$, $s \ge 4$, and afterward, Kwon \cite{Kwon2008-1} improved Ponce's result for $s > \frac52$. Kwon also used the energy method with the modified energy in addition to the refined Strichartz estimate, the maximal function estimate and the local smoothing estimate. Recently, Guo, Kwon and the author \cite{GKK2013}, and Kenig and Pilod \cite{KP2015} further improved the local result independently. The method in both \cite{GKK2013} and \cite{KP2015} is the energy method based on the short time $X^{s,b}$ space, while the key energy estimates were shown by using an additional weight and modified energy, respectively.

The paper is organized as follows: In Section \ref{sec:preliminary}, we summarize some notations and define function spaces. In Section \ref{sec:trilinear}, we prove Theorem \ref{thm:trilinear} by giving a counter example. In Section \ref{sec:L2 block estimate}, we show the $L^2$ block trilinear estimates which are useful to obtain nonlinear and energy estimates. In Sections \ref{sec:nonlinear} and \ref{sec:energy}, we prove the nonlinear estimate and energy estimate, respectively. In Section \ref{sec:main}, we give the proof of Theorem \ref{thm:main}. Finally, we sketch the proof of high regularity result for \eqref{eq:5mkdv5} in Appendix \ref{app:smooth sol} for the convenience of readers.

\textbf{Acknowledgement.} The author would like to thank his advisor Soonsik Kwon for his helpful comments and encouragement through this research problem. Moreover, the author is grateful to Zihua Guo for his helpful advice to understand well the short time $X^{s,b}$ structure under the periodic setting. C. K. is supported by FONDECYT de Postdoctorado 2017 Proyecto No. 3170067.

\section{Preliminaries}\label{sec:preliminary}
For $x,y \in \R_+$, $x \lesssim y$ means that there exists $C>0$ such that $x \le Cy$, and $x \sim y$ means $x \lesssim y$ and $y\lesssim x$. We also use $\lesssim_s$ and $\sim_s$ as similarly, where the implicit constants depend on $s$. Let $a_1,a_2,a_3,a_4 \in \R$. The quantities $a_{max} \ge a_{sub} \ge a_{thd} \ge a_{min}$ can be conveniently defined to be the maximum, sub-maximum, third-maximum and minimum values of $a_1,a_2,a_3,a_4$ respectively.

For $Z = \R$ or $\Z$, let $\Gamma_k(Z)$ denote $(k-1)$-dimensional hyperplane by 
\[\set{\overline{x} = (x_1,x_2,...,x_k) \in Z^k : x_1 +x_2 + \cdots +x_k= 0}.\] 

For $f \in \Sch '(\R \times \T) $ we denote by $\wt{f}$ or $\ft (f)$ the Fourier transform of $f$ with respect to both spatial and time variables,
\[\wt{f}(\tau,n)=\frac{1}{\sqrt{2\pi}}\int_{\R}\int_{0}^{2\pi} e^{-i xn}e^{-it\tau}f(t,x)\; dx dt .\]
Moreover, we use $\ft_x$ (or $\wh{\;}$ ) and $\ft_t$ to denote the Fourier transform with respect to space and time variable respectively.

This paper is based on the following observation. We take the Fourier coefficient in the spatial variable of \eqref{eq:5mkdv} to obtain
\begin{equation}\label{eq:5mkdv1}
\begin{split}
\pt\wh{u}(n) - in^5\wh{u}(n) =&~{} 30i \sum_{n_1+n_2+n_3+n_4+n_5=n} \wh{u}(n_1)\wh{u}(n_2)\wh{u}(n_3)\wh{u}(n_4)n_5\wh{u}(n_5) \\
&+10i \sum_{n_1+n_2+n_3=n} \wh{u}(n_1)\wh{u}(n_2)n_3^3\wh{u}(n_3) \\
&+10i \sum_{n_1+n_2+n_3=n} n_1\wh{u}(n_1)n_2\wh{u}(n_2)n_3\wh{u}(n_3)\\
&+40i \sum_{n_1+n_2+n_3=n} \wh{u}(n_1)n_2\wh{u}(n_2)n_3^2\wh{u}(n_3)
\end{split}
\end{equation}
The resonant relation for the cubic term in the right-hand side of \eqref{eq:5mkdv1} should be
\begin{equation}\label{eq:resonant case1}
\begin{aligned}
H =& H(n_1,n_2,n_3)\\ 
:=& (n_1+n_2+n_3)^5 - n_1^5 - n_2^5 - n_3^5 \\
=& \frac52(n_1+n_2)(n_1+n_2)(n_2+n_3)(n_1^2+n_2^2+n_3^2+n^2).
\end{aligned}
\end{equation}
Then we can observe that the cubic non-trivial resonant phenomenon appears only when $(n_1+n_2)(n_1+n_2)(n_2+n_3)=0$. Moreover, we also observe that $\norm{u}_{L_x^4}^4 n\wh{u}(n)$ portion be contained in the resonant term of quintic term when 
\begin{equation}\label{eq:resonant case2}
n_i + n_j + n_k +n_l = 0 \quad \mbox{for} \quad 1 \le i < j < k < l \le 5.
\end{equation} 
Then, by using the conservation laws in \eqref{eq:hamiltonian} and gathering resonant terms in the right-hand side of \eqref{eq:5mkdv1}, we can rewrite \eqref{eq:5mkdv1} as following:\footnote{By simple calculation \[40u\px u \px^2u + 10 u^2\px^3u + 10(\px u)^3 = 10\px(u^2\px^2 u) + 10\px(u\px u\px u) \hspace{1em} \mbox{and} \hspace{1em} 30u^4\px u = 6 \px (u^5),\] we can change all nonlinear terms into the divergence form. This divergence form is necessary to control the \emph{high $\times$ high $\times$ high $\Rightarrow$ low} and \emph{high $\times$ high $\times$ low $\Rightarrow$ low} interactions in the nonlinear estimate.}
\begin{equation}\label{eq:5mkdv2}
\begin{split}
\pt\wh{u}(n) - i(n^5 + c_1n^3 + c_2n)\wh{u}(n) =&~{}c_3 i \norm{u(t)}_{L^4}^4n\wh{u}(n) -20in^3|\wh{u}(n)|^2\wh{u}(n)\\
&+6in \sum_{\N_{5,n}} \wh{u}(n_1)\wh{u}(n_2)\wh{u}(n_3)\wh{u}(n_4)\wh{u}(n_5) \\
&+10in \sum_{\N_{3,n}} \wh{u}(n_1)\wh{u}(n_2)n_3^2\wh{u}(n_3) \\
&+10in \sum_{\N_{3,n}} \wh{u}(n_1)n_2\wh{u}(n_2)n_3\wh{u}(n_3),
\end{split}
\end{equation}
where $c_1 = 10\norm{u_0}_{L^2}^2$, $c_2 = 10 (\norm{u_0}_{\dot{H^1}}^2 + \norm{u_0}_{L^4}^4)$, $c_3 =20$, 
\begin{equation}\label{eq:3nonres}
\N_{3,n} = \left\{(n_1,n_2,n_3) \in \Z^3 : \begin{array}{ll} &n_1+n_2+n_3=n, \\ 
&(n_1+n_2)(n_1+n_2)(n_2+n_3) \neq 0
\end{array}
\right\}.
\end{equation}
and
\begin{equation}\label{eq:5nonres}
\N_{5,n} = \left\{
(n_1,n_2,n_3,n_4,n_5) \in \Z^5 : \begin{array}{ll} &n_1+n_2+n_3+n_4+n_5=n,\\
 &n_i + n_j + n_k +n_l \neq 0, \; 1 \le i <j<k<l \le 5
\end{array}
\right\}.
\end{equation}

We call the first term of the right-hand side of \eqref{eq:5mkdv2} the \emph{Resonant} term and the others \emph{Non-resonant} terms. Due to the term $c_3\norm{u(t)}_{L^4}^4n\wh{u}(n)$ in the left-hand side of \eqref{eq:5mkdv2}, $\px^5 + c_1\px^3 + c_2\px + c_3\norm{u(t)}_{L^4}^4\px$ does not play a role of the linear operator to \eqref{eq:5mkdv2} yet, even if we use the partial property of completely integrable system. Hence, for our analysis, we define the nonlinear transformation as 
\begin{equation}\label{eq:modified solution}
\NT(u)(t,x) = v(t,x) := \frac{1}{\sqrt{2\pi}}\sum_{n \in \Z} e^{i(nx - 20n \int_0^t \norm{u(s)}_{L^4}^4 \; ds)}\wh{u}(t,n).
\end{equation}

We from now concentrate on the $v$ instead of $u$. For $v$, we rewrite again \eqref{eq:5mkdv2} as
\begin{equation}\label{eq:5mkdv3}
\begin{split}
\pt\wh{v}(n) - i(n^5 + c_1n^3 + c_2n)\wh{v}(n)=&~{}-20in^3|\wh{v}(n)|^2\wh{v}(n)\\
&+6i n\sum_{\N_{5,n}} \wh{v}(n_1)\wh{v}(n_2)\wh{v}(n_3)\wh{v}(n_4)\wh{v}(n_5) \\
&+10in \sum_{\N_{3,n}} \wh{v}(n_1)\wh{v}(n_2)n_3^2\wh{v}(n_3) \\
&+10in \sum_{\N_{3,n}} \wh{v}(n_1)n_2\wh{v}(n_2)n_3\wh{v}(n_3)\\
=:&~{} \wh{N}_1(v) + \wh{N}_2(v) + \wh{N}_3(v) + \wh{N}_4(v)\\
v(0,x) = u(0,x).\hspace{6.5em}&
\end{split}
\end{equation}
We simply generalize $N_i(v)$ as $N_i(u,v,w)$ for the cubic term, $i=1,3,4$, or $N_2(v_1,v_2,v_3,v_4,v_5)$ for the quintic term.

We introduce that $X^{s,b}$-norm associated to \eqref{eq:5mkdv3} which is given by\footnote{Bourgain \cite{Bourgain1993} used this modification of linear operator by using the conservation law in the study of the modified KdV equation.}
\[\norm{u}_{X^{s,b}}=\norm{\bra{ \tau - \mu(n)}^b\bra{n}^s \ft(u)}_{L_{\tau}^2(\R;\ell_n^2(\Z))},\]
where 
\[\mu(n) = n^5 + c_1n^3 + c_2n\]
and $\bra{\cdot} = (1+|\cdot|^2)^{1/2}$. The $X^{s,b}$ space turns out to be very useful in the study of low-regularity theory for the dispersive equations. The Fourier restriction norm method was first implemented in its current form by Bourgain \cite{Bourgain1993} and further developed by Kenig, Ponce and Vega \cite{KPV1996} and Tao \cite{Tao2001}.

Let $\Z_+ = \Z \cap [0,\infty]$. For $k \in \Z_+$, we set
\[I_0 = \set{n \in \Z : |n| \le 2} \hspace{1em} \mbox{ and } \hspace{1em} I_k = \set{n \in \Z : 2^{k-1} \le |n| \le 2^{k+1}}, \hspace{1em} k \ge 1.\]

Let $\eta_0: \R \to [0,1]$ denote a smooth bump function supported in $ [-2,2]$ and equal to $1$ in $[-1,1]$ with the following property of regularities:
\begin{equation}\label{eq:regularity}
\partial_n^{j} \eta_0(n) = O(\eta_0(n)/\bra{n}^j), \hspace{1em} j=0,1,2,
\end{equation}
as $n$ approaches end points of the support of $\eta$.

For $k \in \Z_+ $, let 
\begin{equation}\label{eq:cut-off1}
\chi_0(n) = \eta_0(n), \hspace{1em} \mbox{and} \hspace{1em} \chi_k(n) = \eta_0(n/2^k) - \eta_0(n/2^{k-1}), \hspace{1em} k \ge 1,
\end{equation}
which is supported in $I_k$, and
\[\chi_{[k_1,k_2]}=\sum_{k=k_1}^{k_2} \chi_k \quad \mbox{ for any} \ k_1 \le k_2 \in \Z_+ .\]
$\{ \chi_k \}_{k \in \Z_+}$ is the inhomogeneous decomposition function sequence to the frequency space. For $k\in \Z_+$, let $P_k$ denote the
operators on $L^2(\T)$ defined by $\widehat{P_kv}(n)=\chi_k(n)\wh{v}(n)$. For $l\in \Z_+$, let
\[P_{\le l}=\sum_{k \le l}P_k, \quad P_{\ge l}=\sum_{k \ge l}P_k.\]
For the time-frequency decomposition, we use the cut-off function $\eta_j$, but the same as $\eta_j(\tau-\mu(n)) = \chi_j(\tau-\mu(n))$.

For $k,j \in \Z_+$, let
\[D_{k,j}=\{(\tau,n) \in \R \times \Z : \tau - \mu(n) \in I_j, n \in I_k \}, \hspace{2em} D_{k,\le j}=\cup_{l\le j}D_{k,l}.\]

For $k \in \Z_+$, we define the $X^{s,\frac12,1}$-type space $X_k$ for frequency localized functions,
\begin{eqnarray}\label{eq:Xk}
X_k=\left\{
\begin{array}{l}
f\in L^2(\R \times \Z): f(\tau,n) \mbox{ is supported in } \R \times I_k   \mbox{ and }\nonumber\\
\norm{f}_{X_k}:=\sum_{j=0}^\infty 2^{j/2}\norm{\eta_j(\tau-\mu(n))\cdot f(\tau,n)}_{L_{\tau}^2\ell_n^2}<\infty
\end{array}
\right\}.
\end{eqnarray}

As in \cite{IKT2008}, at frequency $2^k$ we will use the $X^{s,\frac12,1}$ structure given by the $X_k$-norm, uniformly on the $2^{-2k}$ time scale. For $k\in \Z_+$, we define function spaces
\begin{eqnarray*}
&& F_k=\left\{
\begin{array}{l}
f\in L^2(\R \times \T): \widehat{f}(\tau,n) \mbox{ is supported in } \R \times I_k \mbox{ and } \\
\norm{f}_{F_k}=\sup\limits_{t_k\in \R}\norm{\ft[f\cdot\eta_0(2^{2k}(t-t_k))]}_{X_k}<\infty
\end{array}
\right\},
\\
&&N_k=\left\{
\begin{array}{l}
f\in L^2(\R \times \T): \widehat{f}(\tau,n) \mbox{ is supported in } \R \times I_k \mbox{ and }  \\
\norm{f}_{N_k}=\sup\limits_{t_k\in \R}\norm{(\tau-\mu(n)+i2^{2k})^{-1}\ft[f\cdot\eta_0(2^{2k}(t-t_k))]}_{X_k}<\infty
\end{array}
\right\}.
\end{eqnarray*}
Since the spaces $F_k$ and $N_k$ are defined on the whole line in time variable, we define the local-in-time versions of the spaces in standard ways. For $T\in
(0,1]$ we define the normed spaces
\begin{align*}
F_k(T)=&\{f\in C([-T,T]:L^2): \norm{f}_{F_k(T)}=\inf_{\wt{f}=f \mbox{ in } [-T,T] \times \T }\norm{\wt f}_{F_k}\},\\
N_k(T)=&\{f\in C([-T,T]:L^2): \norm{f}_{N_k(T)}=\inf_{\wt{f}=f \mbox{ in } [-T,T] \times \T }\norm{\wt f}_{N_k}\}.
\end{align*}
We assemble these dyadic spaces in a Littlewood-Paley manner. For $s\geq 0$ and $T\in (0,1]$, we define function spaces solutions and
nonlinear terms:
\begin{eqnarray*}
&&F^{s}(T)=\left\{ u: \norm{u}_{F^{s}(T)}^2=\sum_{k=0}^{\infty}2^{2sk}\norm{P_k(u)}_{F_k(T)}^2<\infty \right\},
\\
&&N^{s}(T)=\left\{ u: \norm{u}_{N^{s}(T)}^2=\sum_{k=0}^{\infty}2^{2sk}\norm{P_k(u)}_{N_k(T)}^2<\infty \right\}.
\end{eqnarray*}
We define the dyadic energy space as follows: For $s\geq 0$ and $u\in C([-T,T]:H^\infty)$
\begin{eqnarray*}
\norm{u}_{E^{s}(T)}^2=\norm{P_{0}(u(0))}_{L^2}^2+\sum_{k\geq 1}\sup_{t_k\in [-T,T]}2^{2sk}\norm{P_k(u(t_k))}_{L^2}^2.
\end{eqnarray*}

\begin{lemma}[Properties of $X_k$]\label{lem:prop of Xk}
Let $k, l\in \Z_+$ and $f_k\in X_k$. Then
\begin{equation}\label{eq:prop1}
\begin{split}
&\sum_{j=l+1}^\infty 2^{j/2}\normo{\eta_j(\tau-\mu(n))\int_{\R}|f_k(\tau',n)|2^{-l}(1+2^{-l}|\tau-\tau'|)^{-4}d\tau'}_{L_{\tau}^2\ell_n^2}\\
&+2^{l/2}\normo{\eta_{\leq l}(\tau-\mu(n)) \int_{\R}|f_k(\tau',n)| 2^{-l}(1+2^{-l}|\tau-\tau'|)^{-4}d\tau'}_{L_{\tau}^2\ell_n^2}\lesssim\norm{f_k}_{X_k}.
\end{split}
\end{equation}
In particular, if $t_0\in \R$ and $\gamma\in \Sch(\R)$, then
\begin{eqnarray}\label{eq:prop2}
\norm{\ft[\gamma(2^l(t-t_0))\cdot \ft^{-1}(f_k)]}_{X_k}\lesssim
\norm{f_k}_{X_k}.
\end{eqnarray}
Moreover, from the definition of $X_k$-norm,
\[\normo{\int_{\R}|f_k(\tau',n)|\; d\tau'}_{\ell_n^2} \lesssim \norm{f_k}_{X_k}.\]
\end{lemma}
\begin{proof}
The proof of Lemma \ref{lem:prop of Xk} only depends on the summation over modulations, and there is no difference between the proof in the non-periodic and periodic settings. Hence we omit details and see \cite{GKK2013} for the detailed proof.
\end{proof}

As in \cite{IKT2008}, for any $k\in \Z_+$ we define the set $S_k$ of $k$-\emph{acceptable} time multiplication factors 
\[S_k=\{m_k:\R\rightarrow \R: \norm{m_k}_{S_k}=\sum_{j=0}^{10} 2^{-2jk}\norm{\partial^jm_k}_{L^\infty}< \infty\}.\] 
Direct estimates using the definitions and \eqref{eq:prop2} show that for any $s\geq 0$ and $T\in (0,1]$
\[\begin{cases}
\normb{\sum\limits_{k\in \Z_+} m_k(t)\cdot P_k(u)}_{F^{s}(T)}\lesssim (\sup_{k\in \Z_+}\norm{m_k}_{S_k})\cdot \norm{u}_{F^{s}(T)};\\
\normb{\sum\limits_{k\in \Z_+} m_k(t)\cdot P_k(u)}_{N^{s}(T)}\lesssim (\sup_{k\in \Z_+}\norm{m_k}_{S_k})\cdot \norm{u}_{N^{s}(T)};\\
\normb{\sum\limits_{k\in \Z_+} m_k(t)\cdot P_k(u)}_{E^{s}(T)}\lesssim (\sup_{k\in \Z_+}\norm{m_k}_{S_k})\cdot \norm{u}_{E^{s}(T)}.
\end{cases}\]

\section{Proof of Theorem \ref{thm:trilinear}}\label{sec:trilinear}

In this section, we show the Theorem \ref{thm:trilinear}. We bring the similar argument in \cite{KPV1996} associated to the modified KdV equation in order to construct the counter-examples. As mentioned in the introduction, we observe the \emph{high $\times$ low $\times$ low $\Rightarrow$ high} interaction component in the non-resonance phenomenon, while Kenig, Ponce and Vega focused on the \emph{high $\times$ high $\times$ high $\Rightarrow$ high} interaction component in the resonant term. When we apply our examples to the modified KdV equation, it can be easily controlled in $X^{s,\frac12}$, because the size of maximum modulation is comparable to the square of high frequency size ($\approx N^2$) and hence this factor exactly eliminates the one derivative in the nonlinear term. In contrast to this, \eqref{eq:5mkdv} has two more derivatives in nonlinear terms, and thus one cannot control the this component in $X^{s,b}$-norm, although the modulation effect is better than that of modified KdV equation. Now, we give examples satisfying

\begin{equation}\label{eq:fail}
\norm{uv\px^3w}_{X^{s,b-1}} \nleq C\norm{u}_{X^{s,b}}\norm{v}_{X^{s,b}}\norm{w}_{X^{s,b}}.
\end{equation}
In the case of our examples, the trilinear estimate does not depend on the regularity $s$. So, it suffices to show \eqref{eq:fail} for any $b \in \R$. Fix $N \gg 1$. Let us define the functions
\[f(\tau,n) = a_n \chi_{\frac14}(\tau-n^5), \hspace{1em} g(\tau,n) = b_n \chi_{\frac14}(\tau-n^5), \hspace{1em} h(\tau,n) = d_n \chi_{\frac14}(\tau-n^5),\]
where
\begin{equation*}
a_n = \begin{cases}1, \hspace{0.5em} n=-1\\0, \hspace{0.5em}otherwise\end{cases} \hspace{2em}b_n = \begin{cases}1, \hspace{0.5em}n=2\\0, \hspace{0.5em}otherwise\end{cases} \hspace{2em}d_n = \begin{cases}1, \hspace{0.5em}n=N-1\\0, \hspace{0.5em}otherwise\end{cases}.
\end{equation*}
We focus on the case that $|\tau- n^5|$ is the maximum modulation. We put 
\[\wt{u}(\tau,n) = f(\tau,n) \hspace{2em} \wt{v}(\tau,n) = g(\tau,n) \hspace{2em} \wt{w}(\tau,n) = h(\tau,n).\]
Then we need to calculate $\ft[uv\px^3w](\tau,n)$. Since $\ft[uv\px^3w](\tau,n) = (f \ast g \ast h)(\tau,n)$, performing the summation and integration with respect to $n_1, \tau_1$ variables gives
\begin{align*}
(f \ast g)(\tau_2,n_2) &= \sum_{n_1}a_{n_1}b_{n_2-n_1} \int_{\R}\chi_{\frac14}(\tau_1-n_1^5)\chi_{\frac14}(\tau_2-\tau_1-(n_2-n_1)^5)\; d\tau_1\\
&\cong c \sum_{n_1}a_{n_1}b_{n_2-n_1} \chi_{\frac12}(\tau_2-n_2^5 + 5n_1n_2(n_2-n_1)(n_1^2+n_2^2-n_1n_2))\\
&\cong c \alpha_{n_2}\chi_{\frac12}(\tau_2-n_2^5 -30),
\end{align*}
where 
\[\alpha_n = \begin{cases}1, \hspace{0.5em}n=1\\0, \hspace{0.5em}otherwise\end{cases}.\]
By performing the summation and integration with respect to $n_2,\tau_2$ variables once more, we have
\begin{align*}
(f \ast g \ast h)(\tau,n) &=[(f \ast g) \ast h](\tau,n)\\
&= \sum_{n_2}\alpha_{n_2}d_{n-n_2} \int_{\R}\chi_{\frac12}(\tau_2-n_2^5 -30)\chi_{\frac14}(\tau-\tau_2-(n-n_2)^5)\; d\tau_2\\
&\cong c \sum_{n_2}\alpha_{n_2}d_{n-n_2} \chi_1(\tau-(n-n_2)^5-n_2^5 -30)\\
&\cong c \beta_{n}\chi_{1}(\tau - (n-1)^5-31),
\end{align*}
where
\[\beta_n = \begin{cases}1, \hspace{0.5em}n=N\\0, \hspace{0.5em}otherwise\end{cases}.\]
On the support of $(f \ast g \ast h)(\tau,n)$, since we have $|\tau - n^5| \sim N^4$, we finally obtain
\begin{align*}
\norm{uv\px^3w}_{X^{s,b-1}} &= \norm{\bra{n}^s\bra{\tau-n^5}^{b-1}\ft[uv\px^3w](\tau,n)}_{L_{\tau}^2\ell_n^2}\\
&\sim N^sN^3N^{4(b-1)},
\end{align*}
but
\[\norm{u}_{X^{s,b}}\norm{v}_{X^{s,b}}\norm{w}_{X^{s,b}} \sim N^s.\]
This imposes $b \le \frac14$ to succeed the trilinear estimate and hence, we show \eqref{eq:fail} when $b > \frac14$.

We now construct an example when $b \le \frac14$. To do this, let us focus on the case when $|\tau - n^5|$ is much smaller than the maximum modulation. In this case, we may assume that $|\tau_1-n_1^5|$ is the maximum modulation by symmetry among modulations. Set
\begin{equation*}
a_n = \begin{cases}1, \hspace{0.5em}n=-N\\0, \hspace{0.5em}otherwise\end{cases} \hspace{2em}b_n = \begin{cases}1, \hspace{0.5em}n=2\\0, \hspace{0.5em}otherwise\end{cases} \hspace{2em}d_n = \begin{cases}1, \hspace{0.5em}n=N-1\\0, \hspace{0.5em}otherwise\end{cases}
\end{equation*}
and
\[f(\tau,n) = a_n \chi_{\frac14}(\tau-n^5), \hspace{1em} g(\tau,n) = b_n \chi_{\frac14}(\tau-n^5), \hspace{1em}h(\tau,n) = d_n \chi_{\frac14}(\tau-n^5).\]
From the duality, it suffices to consider
\[\norm{uv\px^3w}_{X_{\tau-n^5}^{-s,-b}} \le C\norm{u}_{X_{\tau-n^5}^{-s,1-b}}\norm{v}_{X_{\tau-n^5}^{s,b}}\norm{w}_{X_{\tau-n^5}^{s,b}},\]
where 
\[\wt{u}(\tau,n) = f(\tau,n) \hspace{2em} \wt{v}(\tau,n) = g(\tau,n) \hspace{2em} \wt{w}(\tau,n) = h(\tau,n).\]
Similarly as before, we need to calculate $\ft[uv\px^3w](\tau,n)$. Since $\ft[uv\px^3w](\tau,n) = (f \ast g \ast h)(\tau,n)$, performing the summation and integration with respect to $n_1, \tau_1$ variables  gives
\begin{align*}
(f \ast g)(\tau_2,n_2) &= \sum_{n_1}a_{n_1}b_{n_2-n_1} \int_{\R}\chi_{\frac14}(\tau_1-n_1^5)\chi_{\frac14}(\tau_2-\tau_1-(n_2-n_1)^5)\; d\tau_1\\
&\cong c \sum_{n_1}a_{n_1}b_{n_2-n_1} \chi_{\frac12}(\tau_2-n_2^5 + 5n_1n_2(n_2-n_1)(n_1^2+n_2^2-n_1n_2))\\
&\cong c \alpha_{n_2}\chi_{\frac12}(\tau_2-n_2^5 + 10N(N-2)[(N-1)^2+3]),
\end{align*}
where 
\[\alpha_n = \begin{cases}1, \hspace{0.5em}n=2-N\\0, \hspace{0.5em}otherwise\end{cases}.\]
By performing the summation and integration with respect to $n_2,\tau_2$ variables once more, we have
\begin{align*}
(f \ast g \ast h)(\tau,n) &=[(f \ast g) \ast h](\tau,n)\\
&= \sum_{n_2}\alpha_{n_2}d_{n-n_2} \int_{\R}\chi_{\frac12}(\tau_2-n_2^5 + 10N(N-2)[(N-1)^2+3])\\
&\hspace{15em}\times\chi_{\frac14}(\tau-\tau_2-(n-n_2)^5)\; d\tau_2\\
&\cong c \sum_{n_2}\alpha_{n_2}d_{n-n_2} \chi_1(\tau-(n-n_2)^5-n_2^5 + 10N(N-2)[(N-1)^2+3])\\
&\cong c \beta_{n}\chi_{1}(\tau - (n+N-2)^5+ (N-2)^5 + 10N(N-2)[(N-1)^2+3]),
\end{align*}
where
\[\beta_n = \begin{cases}1, \hspace{0.5em}n=1\\0, \hspace{0.5em}otherwise\end{cases}.\]
On the support of $(f \ast g \ast h)(\tau,n)$, since we have $|\tau - n^5| \sim N^4$, we finally obtain
\begin{align*}
\norm{uv\px^3w}_{X^{-s,-b}} &= \norm{\bra{n}^{-s}\bra{\tau-n^5}^{-b}\ft[uv\px^3w](\tau,n)}_{L_{\tau}^2\ell_n^2}\\
&\sim N^3N^{-4b},
\end{align*}
but
\[\norm{u}_{X^{-s,1-b}}\norm{v}_{X^{s,b}}\norm{w}_{X^{s,b}} \sim N^{-s}N^s \sim 1.\]
This imposes $b \ge \frac34$ and hence, we show \eqref{eq:fail} when $b \le \frac14$, which complete the proof of Theorem \ref{thm:trilinear}.

\section{$L^2$-block estimates}\label{sec:L2 block estimate}
In this section, we will give $L^2$-block estimates for trilinear estimates. For $n_1,n_2,n_3 \in \Z$, let
\[G(n_1,n_2,n_3) = \mu(n_1 + n_2 + n_3)- \mu(n_1) - \mu(n_2) - \mu(n_3)\]
be the resonance function, which plays an important role in the trilinear $X^{s,b}$-type estimates. 

Let $\zeta_i = \tau_i - \mu(n_i)$. For compactly supported functions $f_i \in L^2(\R \times \Z)$, $i=1,2,3,4$, we define 
\[J(f_1,f_2,f_3,f_4) = \sum_{n_4, \N_{3,n_4}}\int_{\overline{\zeta}\in \Gamma_4(\R)}f_1(\zeta_1,n_1)f_2(\zeta_2,n_2)f_3(\zeta_3,n_3)f_4(\zeta_4 + G(n_1,n_2,n_3),n_4),\]
where  $\overline{\zeta} = (\zeta_1,\zeta_2,\zeta_3,-\zeta_4-G(n_1,n_2,n_3))$. From the identities
\[n_1+n_2+n_3 = n_4\]
and
\[\zeta_1+\zeta_2+\zeta_3 = \zeta_4 + G(n_1,n_2,n_3)\]
on the support of $J(f_1,f_2,f_3,f_4)$, we see that $J(f_1,f_2,f_3,f_4)$ vanishes unless
\begin{equation}\label{eq:support property}
\begin{array}{c}
2^{k_{max}} \sim 2^{k_{sub}}\\
2^{j_{max}} \sim \max(2^{j_{sub}}, |G|),
\end{array}
\end{equation}
where $|n_i| \sim 2^{k_i}$ and $|\zeta_i| \sim 2^{j_i}$, $i=1,2,3,4$. From the simple calculation we know that
\[\int_{\R} (f \ast g) \cdot h = \int_{\R} (f^{\ast} \ast h) \cdot g,\]
where $f^{\ast}(x) = f(-x)$. Then, in addition to the fact that the convolution operator admits the commutative law ($f \ast g = g \ast f$), we have 
\begin{equation}\label{eq:symmetry}
|J(f_1,f_2,f_3,f_4)|=|J(f_2,f_1,f_3,f_4)|=|J(f_3,f_2,f_1,f_4)|=|J(f_1^{\ast},f_2^{\ast},f_4,f_3)|.
\end{equation}
\begin{lemma}\label{lem:tri-L2}
Let $k_i, j_i\in \Z_+$, $i=1,2,3,4$. Let $f_{k_i,j_i} \in L^2(\R \times \Z) $ be nonnegative functions supported in $I_{j_i} \times I_{k_i}$.

\noi(a) For any $k_i,j_i \in \Z_+$, $i=1,2,3,4$, we
have
\begin{eqnarray}\label{eq:tri-block estimate-a1}
J(f_{k_1,j_1},f_{k_2,j_2},f_{k_3,j_3},f_{k_4,j_4}) \lesssim 2^{(j_{min}+j_{thd})/2}2^{(k_{min}+k_{thd})/2}\prod_{i=1}^4 \|f_{k_i,j_i}\|_{L^2}.
\end{eqnarray}

\noi(b) Let $k_{thd} \le k_{max}-10$.

\noi(b-1) If $(k_i,j_i) = (k_{thd},j_{max})$ for $i=1,2,3,4$ and $j_{sub} \le 4k_{max}$, we have
\begin{eqnarray}\label{eq:tri-block estimate-b1.1}
 J(f_{k_1,j_1},f_{k_2,j_2},f_{k_3,j_3},f_{k_4,j_4}) \lesssim 2^{(j_1+j_2+j_3+j_4)/2}2^{-(j_{sub}+j_{max})/2}2^{k_{thd}/2}\prod_{i=1}^4 \|f_{k_i,j_i}\|_{L^2}.
\end{eqnarray}
\noi(b-2) If $(k_i,j_i) = (k_{thd},j_{max})$ for $i=1,2,3,4$ and $j_{sub} > 4k_{max}$, we have
\begin{eqnarray}\label{eq:tri-block estimate-b1.2}
 J(f_{k_1,j_1},f_{k_2,j_2},f_{k_3,j_3},f_{k_4,j_4}) \lesssim 2^{(j_1+j_2+j_3+j_4)/2}2^{-2k_{max}}2^{k_{thd}/2}2^{-j_{max}/2}\prod_{i=1}^4 \|f_{k_i,j_i}\|_{L^2}.
\end{eqnarray}
\noi(b-3) If $(k_i,j_i) \neq (k_{thd},j_{max})$ for $i=1,2,3,4$ and $j_{sub} \le 4k_{max}$, we have
\begin{eqnarray}\label{eq:tri-block estimate-b1.3}
J(f_{k_1,j_1},f_{k_2,j_2},f_{k_3,j_3},f_{k_4,j_4}) \lesssim 2^{(j_1+j_2+j_3+j_4)/2}2^{-(j_{sub}+j_{max})/2}2^{k_{min}/2}\prod_{i=1}^4 \|f_{k_i,j_i}\|_{L^2}.
\end{eqnarray}
\noi(b-4) If $(k_i,j_i) \neq (k_{thd},j_{max})$ for $i=1,2,3,4$ and $j_{sub} > 4k_{max}$, we have
\begin{eqnarray}\label{eq:tri-block estimate-b1.4}
 J(f_{k_1,j_1},f_{k_2,j_2},f_{k_3,j_3},f_{k_4,j_4}) \lesssim 2^{(j_1+j_2+j_3+j_4)/2}2^{-2k_{max}}2^{k_{min}/2}2^{-j_{max}/2}\prod_{i=1}^4 \|f_{k_i,j_i}\|_{L^2}.
\end{eqnarray}

\end{lemma} 

\begin{proof}
For (a), by using the Cauchy-Schwarz inequality with respect to $\zeta$ and $n$ variables to $J(f_1,f_2,f_3,f_4)$, we can easily obtain \eqref{eq:tri-block estimate-a1}.

For (b), we fix $\zeta_i$, $i =1,2,3,4$. We first consider the summation over frequencies. Due to \eqref{eq:symmetry} in addition to the fact of $\norm{f}_{L^2} = \norm{f^{\ast}}_{L^2}$, we may assume that $j_1 \le j_2 \le j_3 \le j_4$. Since constraints of modulations satisfy
\[\zeta_i = \tau_i -\mu(n_i) = O(2^{j_i}), \hspace{1em} i=1,2,3,\]
it is enough to estimate
\begin{equation}\label{eq:block1}
\sum_{\substack{n_4, \N_{3,n_4}\\ \mu(n_1) + \mu(n_2) + \mu(n_3) = \tau_4 + O(2^{j_3})}}f_{k_1,j_1}(n_1)f_{k_2,j_2}(n_2)f_{k_3,j_3}(n_3)f_{k_4,j_4}(n_1+n_2+n_3).
\end{equation}

For the proofs of \eqref{eq:tri-block estimate-b1.1} and \eqref{eq:tri-block estimate-b1.2}, we suppose to hold $k_4 = k_{thd}$. Assume that $|n_3| \ll |n_4| \ll |n_2| \le |n_1|$\footnote{Both $|n_3| \sim |n_4|$ and $|n_4| \sim |n_2|$ can be treated in the proofs of \eqref{eq:tri-block estimate-b1.3} and \eqref{eq:tri-block estimate-b1.4}.}Then \eqref{eq:block1} can be rewritten by
\[\sum_{\substack{n_1,n_3,n_4\\ \mu(n_1) + \mu(n_4-n_1-n_3) + \mu(n_3) = \tau_4 + O(2^{j_3})}}f_{k_1,j_1}(n_1)f_{k_2,j_2}(n_4-n_1-n_2)f_{k_3,j_3}(n_3)f_{k_4,j_4}(n_4).\]
Then, since
\[\partial_{n_3}\left(\mu(n_1) + \mu(n_4-n_1-n_3) + \mu(n_3)\right) = 5n_3^4 -5(n_4-n_1-n_3)^4 +3c_1n_3^2-3c_1(n_4-n_1-n_3)^2,\]
which implies that $n_3$ is contained in two intervals of length $O(2^{-4k_1}2^{j_3})$, i.e.
\[\mbox{the number of } n_3 \lesssim 2^{-4k_1}2^{j_3} + 1,\]
If $2^{-4k_1}2^{j_3} \lesssim 1$, the number of $n_3$ is constant independent of frequencies. By the Cauchy-Schwarz inequality with respect to $n_3,n_1,n_4$ variables in regular order, we have 
\begin{equation}\label{eq:block2}
\begin{aligned}
&\sum_{A}|f_{k_1,j_1}(n_1)f_{k_2,j_2}(n_4-n_1-n_2)f_{k_3,j_3}(n_3)f_{k_4,j_4}(n_4)|\\
&\hspace{6em}\lesssim \norm{f_{k_3,j_3}}_{\ell^2}\sum_{n_4,n_1}|f_{k_1,j_1}(n_1)f_{k_2,j_2}(n_4-n_1-n_2)f_{k_4,j_4}(n_4)|\\
&\hspace{6em}\lesssim \norm{f_{k_3,j_3}}_{\ell^2}\norm{f_{k_1,j_1}}_{\ell^2}\norm{f_{k_2,j_2}}_{\ell^2}\sum_{n_4}|f_{k_4,j_4}(n_4)|\\
&\hspace{6em}\lesssim 2^{k_{thd/2}}\norm{f_{k_1,j_1}}_{\ell^2}\norm{f_{k_2,j_2}}_{\ell^2}\norm{f_{k_3,j_3}}_{\ell^2}\norm{f_{k_4,j_4}}_{\ell^2},
\end{aligned}
\end{equation}
where $A =  \set{(n_4,n_1,n_3) \in \Z^3 :  \mu(n_1) + \mu(n_4-n_1-n_3) + \mu(n_3) = \tau_4 + O(2^{j_3})}$. By performing the Cauchy-Schwarz inequality again in terms of $\zeta$ variables, we have \eqref{eq:tri-block estimate-b1.1}. Otherwise, the similar argument yields \eqref{eq:tri-block estimate-b1.2}, while we have $2^{-2k_1}2^{j_3/2}$ factor at the first inequality in \eqref{eq:block2} from the Cauchy-Schwarz inequality in terms of $n_3$.

We remark that in the case of $k_4 = k_{thd}$, no matter which ordering of frequencies we will fix, the similar change of variables is acceptable, and hence the same result can be obtained thanks to the symmetry of $n_1,n_2$ and $n_3$ variables. Indeed, for instance, if $|n_1| \ll |n_4| \ll |n_3| \le |n_2|$, by replacing the variable $n_3$ by $n_4 - n_1+n_2$ and counting the number of $n_1$ variable from
\[\partial_{n_1}\left(\mu(n_1) + \mu(n_2) + \mu(n_4-n_1-n_2)\right) = 5n_1^4 -5(n_4-n_1-n_2)^4 +3c_1n_1^2-3c_1(n_4-n_1-n_2)^2,\]
we can have the same result. The rest of cases also hold.

For the proof of \eqref{eq:tri-block estimate-b1.3} and \eqref{eq:tri-block estimate-b1.4}, we assume $k_4 \neq k_{thd}$ and we split this case into two cases: $k_4 = k_{max}$ and $k_4 = k_{min}$ \footnote{Due to the support property, the case when $k_4 = k_{sub}$ is exactly same as the case when $k_4 = k_{max}$.}. For the first case ($k_4 = k_{max}$), we may assume that $|n_1| \le |n_2| \le |n_3| \le |n_4|$ due to the symmetry of $n_1,n_2$ and $n_3$ variables.\footnote{This assumption still makes a sense due to the above remark.} By replacing $n_3$ by $n_4 - n_1+n_2$, \eqref{eq:block1} can be rewritten by
\[\sum_{\substack{n_1,n_4,n_2\\ \mu(n_1) + \mu(n_2) + \mu(n_4- n_1-n_2) = \tau_4 + O(2^{j_3})}}f_{k_1,j_1}(n_1)f_{k_2,j_2}(n_2)f_{k_3,j_3}(n_4-n_1-n_2)f_{k_4,j_4}(n_4).\]
Since
\[\partial_{n_2}\left(\mu(n_1) + \mu(n_2) + \mu(n_4- n_1-n_2)\right) = 5n_2^4 -5(n_4-n_1-n_2)^4 +3c_1n_2^2-3c_1(n_4-n_1-n_2)^2,\]
$|n_4| \sim 2^{k_4}$ and $k_1 \le k_2 \le k_3 - 10$, $n_2$ is contained in two intervals of length $O(2^{-4k_4}2^{j_3})$, i.e.
\[\mbox{the number of } n_2 \lesssim 2^{-4k_4}2^{j_3} + 1.\]
If $2^{-4k_4}2^{j_3} \lesssim 1$, the number of $n_1$ and $n_2$ are constants independent of $k_i$. By the Cauchy-Schwarz inequality with respect to $n_2, n_4, n_1$ variables in regular order and $\zeta$ variables, we have \eqref{eq:tri-block estimate-b1.3}. Otherwise, we also apply the Cauchy-Schwarz inequality to obtain \eqref{eq:tri-block estimate-b1.4}.

When $k_4 = k_{min}$, we can similarly prove it. Briefly, from the the identity $n_1+n_2+n_3 = n_4$, we can represent the maximum frequency by the other frequencies. After that, if we count the number of the third largest frequency, we can have the same result by performing the analogous procedure as above. We omit the details and hence it completes the proof.
\end{proof}

As an immediate consequence, we have the following corollary:

\begin{corollary}\label{cor:tri-L2}
Let $k_i, j_i\in \Z_+$, $i=1,2,3,4$. Let $f_{k_i,j_i} \in L^2(\R\times\Z) $ be nonnegative functions supported in $I_{j_i} \times I_{k_i}$, $i=1,2,3$.

\noi(a) For any $k_i, j_i\in \Z_+$, $i=1,2,3,4$, we
have
\begin{eqnarray}\label{eq:tri-block estimate-a2}
\norm{\mathbf{1}_{D_{k_4,j_4}}(n,\tau) (f_{k_1,j_1}\ast f_{k_2,j_2}\ast f_{k_3,j_3})}_{L^2} \lesssim 2^{(j_{min}+j_{thd})/2}2^{(k_{min}+k_{thd})/2}\prod_{i=1}^3 \|f_{k_i,j_i}\|_{L^2}.
\end{eqnarray}

\noi(b) Let $k_{thd} \le k_{max}-10$.

\noi(b-1) If $(k_i,j_i) = (k_{thd},j_{max})$ for $i=1,2,3,4$ and $j_{sub} \le 4k_{max}$, we have
\begin{eqnarray}\label{eq:tri-block estimate-b2.1}
 \norm{\mathbf{1}_{D_{k_4,j_4}}(n,\tau) (f_{k_1,j_1}\ast f_{k_2,j_2}\ast f_{k_3,j_3})}_{L^2} \lesssim 2^{(j_1+j_2+j_3+j_4)/2}2^{-(j_{sub}+j_{max})/2}2^{k_{thd}/2}\prod_{i=1}^3 \|f_{k_i,j_i}\|_{L^2}.
\end{eqnarray}
\noi(b-2) If $(k_i,j_i) = (k_{thd},j_{max})$ for $i=1,2,3,4$ and $j_{sub} > 4k_{max}$, we have
\begin{eqnarray}\label{eq:tri-block estimate-b2.2}
\norm{\mathbf{1}_{D_{k_4,j_4}}(n,\tau) (f_{k_1,j_1}\ast f_{k_2,j_2}\ast f_{k_3,j_3})}_{L^2} \lesssim 2^{(j_1+j_2+j_3+j_4)/2}2^{-2k_{max}}2^{k_{thd}/2}2^{-j_{max}/2}\prod_{i=1}^3 \|f_{k_i,j_i}\|_{L^2}.
\end{eqnarray}
\noi(b-3) If $(k_i,j_i) \neq (k_{thd},j_{max})$ for $i=1,2,3,4$ and $j_{sub} \le 4k_{max}$, we have
\begin{eqnarray}\label{eq:tri-block estimate-b2.3}
\norm{\mathbf{1}_{D_{k_4,j_4}}(n,\tau) (f_{k_1,j_1}\ast f_{k_2,j_2}\ast f_{k_3,j_3})}_{L^2} \lesssim 2^{(j_1+j_2+j_3+j_4)/2}2^{-(j_{sub}+j_{max})/2}2^{k_{min}/2}\prod_{i=1}^3 \|f_{k_i,j_i}\|_{L^2}.
\end{eqnarray}
\noi(b-4) If $(k_i,j_i) \neq (k_{thd},j_{max})$ for $i=1,2,3,4$ and $j_{sub} > 4k_{max}$, we have
\begin{eqnarray}\label{eq:tri-block estimate-b2.4}
 \norm{\mathbf{1}_{D_{k_4,j_4}}(n,\tau) (f_{k_1,j_1}\ast f_{k_2,j_2}\ast f_{k_3,j_3})}_{L^2} \lesssim 2^{(j_1+j_2+j_3+j_4)/2}2^{-2k_{max}}2^{k_{min}/2}2^{-j_{max}/2}\prod_{i=1}^3 \|f_{k_i,j_i}\|_{L^2}.
\end{eqnarray}
\end{corollary}

\section{Nonlinear estimates}\label{sec:nonlinear}
In this section, we show the trilinear and quintilinear estimates. 

\begin{lemma}[Resonance estimate]\label{lem:resonant1}
Let $k \ge 0$. Then, we have
\begin{equation}\label{eq:resonant1-1}
\norm{P_kN_1(u,v,w)}_{N_k} \lesssim 2^k\norm{P_ku}_{F_{k}}\norm{P_kv}_{F_{k}}\norm{P_kw}_{F_{k}}.
\end{equation} 
\end{lemma}

\begin{proof}
From the definitions of $N_1(u,v,w)$ and $N_k$ norm, the left-hand side of \eqref{eq:resonant1-1} is bounded by
\begin{equation}\label{eq:resonant1-2}
\begin{aligned}
\sup_{t_k \in \R} &\Big\|(\tau - \mu(n)+ i2^{2k})^{-1}2^{3k}\mathbf{1}_{I_k}(n)\ft\left[\eta_0\left(2^{2k-2}(t-t_k)\right)P_ku\right] \\
&\hspace{3em}\ast \ft\left[\eta_0\left(2^{2k-2}(t-t_k)\right)P_kv\right] \ast \ft\left[\eta_0\left(2^{2k-2}(t-t_k)\right)P_kw\right]\Big\|_{X_k}
\end{aligned}
\end{equation}
Set $u_k = \ft\left[\eta_0\left(2^{2k-2}(t-t_k)\right)P_ku\right], v_k = \ft\left[\eta_0\left(2^{2k-2}(t-t_k)P_kv\right)\right]$ and $w_k = \ft\left[\eta_0\left(2^{2k-2}(t-t_k)\right)P_kw\right]$. We decompose each of $u_k,v_k$ and $w_k$ into modulation dyadic pieces as $u_{k,j_1}(\tau,n) = u_k(\tau,n)\eta_{j_1}(\tau - \mu(n))$, $v_{k,j_2}(\tau,n) = v_k(\tau,n)\eta_{j_2}(\tau - \mu(n))$ and $w_{k,j_3}(\tau,n) = w_k(\tau,n)\eta_{j_3}(\tau - \mu(n))$, respectively, with usual modification like $f_{\le j}(\tau) = f(\tau)\eta_{\le j}(\tau-\mu(n))$. Then, from the Cauchy-Schwarz inequality, \eqref{eq:resonant1-2} is bounded by
\begin{equation}\label{eq:resonant1-3}
2^{3k}\sum_{j_4 \ge 0} \frac{2^{j_4/2}}{\max(2^{j_4},2^{2k})} \sum_{j_1,j_2,j_3 \ge 2k} 2^{(j_{min}+j_{thd})/2}\norm{u_{k,j_1}}_{L_{\tau}^2\ell_n^2}\norm{v_{k,j_2}}_{L_{\tau}^2\ell_n^2}\norm{w_{k,j_3}}_{L_{\tau}^2\ell_n^2}.
\end{equation}
Since $j_1,j_2,j_3 \ge 2k$, if $j_4 \le 2k$, we use 
\[(\max(2^{j_4},2^{2k}))^{-1}2^{(j_{min}+j_{thd})/2} \lesssim 2^{(j_1+j_2+j_3)/2}2^{-3k},\] 
and otherwise, we use
\[(\max(2^{j_4},2^{2k}))^{-1}2^{(j_{min}+j_{thd})/2} \lesssim 2^{-j_4}2^{(j_1+j_2+j_3)/2}2^{-k}.\] 
By performing all summations over $j_1, j_2, j_3$ and $j_4$, we have
\begin{align*}
\eqref{eq:resonant1-3} &\lesssim 2^k\sum_{j_1,j_2,j_3\ge 2k} 2^{(j_1+j_2+j_3)/2}\norm{u_{k,j_1}}_{L_{\tau}^2\ell_n^2}\norm{v_{k,j_2}}_{L_{\tau}^2\ell_n^2}\norm{w_{k,j_3}}_{L_{\tau}^2\ell_n^2}\\
&\lesssim 2^k \norm{u_k}_{X_k}\norm{v_k}_{X_k}\norm{w_k}_{X_k},
\end{align*}
which implies \eqref{eq:resonant1-1}. 
\end{proof}

Next, we show the non-resonance estimates by dividing into several cases. From the support property \eqref{eq:support property}, we, from now on, may assume that\footnote{Although we have by simple calculation that 
\[\mu(n)-\mu(n_1)-\mu(n_2)-\mu(n_3) = (n_1+n_2)(n_1+n_3)(n_2+n_3)(n_1^2+n_2^2+n_3^2+n^2 + c_1),\]
it makes a sense to assume \eqref{eq:high modulation} due to $c_1 \ge 0$.}
\begin{equation}\label{eq:high modulation}
\max(|\tau-\mu(n)|,|\tau_j-\mu(n_j)|;j=1,2,3) \gtrsim |(n_1+n_2)(n_1+n_3)(n_2+n_3)|(n_1^2+n_2^2+n_3^2+n^2).
\end{equation}

\begin{lemma}[High-high-high $\Rightarrow$ high]\label{lem:nonres1}
Let $k_4 \ge 20$ and $|k_1-k_4|, |k_2-k_4|, |k_3-k_4| \le 5$. Then, we have 
\begin{equation}\label{eq:nonres1-1}
\begin{aligned}
\norm{P_{k_4}N_3(P_{k_1}u,P_{k_2}v,P_{k_3}w)}_{N_{k_4}} &+ \norm{P_{k_4}N_4(P_{k_1}u,P_{k_2}v,P_{k_3}w)}_{N_{k_4}}\\
&\hspace{3em}\lesssim 2^{3k_3/2}\norm{P_{k_1}u}_{F_{k_1}}\norm{P_{k_2}v}_{F_{k_2}}\norm{P_{k_3}w}_{F_{k_3}}.
\end{aligned}
\end{equation} 
\end{lemma}

\begin{proof}
As in the proof of Lemma \ref{lem:resonant1}, both terms of the left-hand side of \eqref{eq:nonres1-1} are bounded by 
\begin{equation}\label{eq:nonres1-2}
\begin{aligned}
\sup_{t_k \in \R} &\Big\|(\tau_4 - \mu(n_4)+ i2^{2k_4})^{-1}2^{3k_4}\mathbf{1}_{I_{k_4}}(n_4)\ft\left[\eta_0\left(2^{2k_4-2}(t-t_k)\right)P_{k_1}u\right] \\
&\hspace{2em}\ast \ft\left[\eta_0\left(2^{2k_4-2}(t-t_k)P_{k_2}v\right)\right] \ast \ft\left[\eta_0\left(2^{2k_4-2}(t-t_k)\right)P_{k_3}w\right]\Big\|_{X_{k_4}}
\end{aligned}
\end{equation}
Set, similarly as in the proof of Lemma \ref{lem:resonant1}, $f_{k_1} = \ft\left[\eta_0\left(2^{2k_4-2}(t-t_k)\right)P_{k_1}u\right]$, $f_{k_2} = \ft\left[\eta_0\left(2^{2k_4-2}(t-t_k)\right)P_{k_2}v\right]$ and $f_{k_3} = \ft\left[\eta_0\left(2^{2k_4-2}(t-t_k)\right)P_{k_3}w\right]$. Also, we decompose $f_{k_i}$ into modulation dyadic pieces as $f_{k_i,j_i}(\tau,n) = f_{k_i}(\tau,n)\eta_{j_i}(\tau -\mu(n))$, $j=1,2,3$, with usual modification $f_{k,\le j}(\tau,n) = f_{k}(\tau,n)\eta_{\le j}(\tau-\mu(n))$. Then, \eqref{eq:nonres1-2} is bounded by
\[2^{3k_3}\sum_{j_4 \ge 0} \frac{2^{j_4/2}}{\max(2^{j_4},2^{2k_4})} \sum_{j_1,j_2,j_3 \ge 2k_4}\norm{\mathbf{1}_{D_{k_4,j_4}}\cdot(f_{k_1,j_1} \ast f_{k_2,j_2} \ast f_{k_3,j_3})}_{L_{\tau_4}^2\ell_{n_4}^2},\]
and by applying \eqref{eq:tri-block estimate-a2} to $\norm{\mathbf{1}_{D_{k_4,j_4}}\cdot(f_{k_1,j_1} \ast f_{k_2,j_2} \ast f_{k_3,j_3})}_{L_{\tau_4}^2\ell_{n_4}^2}$, we have
\[\begin{aligned}
\eqref{eq:nonres1-2} \lesssim 2^{3k_3}\sum_{j_4 \ge 0}& \frac{2^{j_4/2}}{\max(2^{j_4},2^{2k_4})} \\
& \times \sum_{j_1,j_2,j_3 \ge 2k_4}2^{(j_{min}+j_{thd})/2}2^{k_4}\norm{f_{k_1,j_1}}_{L_{\tau}^2\ell_n^2}\norm{f_{k_2,j_2}}_{L_{\tau}^2\ell_n^2}\norm{f_{k_3,j_3}}_{L_{\tau}^2\ell_n^2}.
\end{aligned}\]
From \eqref{eq:high modulation}, by using $j_{max} \ge 3k_3$ and $j_1,j_2,j_3 \ge 2k_3$, if $j_4 \le 2k_4$, we get 
\[(\max(2^{j_4},2^{2k_4}))^{-1}2^{(j_{min}+j_{thd})/2} \lesssim 2^{-2k_4}2^{(j_1+j_2+j_3)/2}2^{-3k_3/2}.\]
Otherwise, if $j_4 = j_{max} \ge 3k_4$, then 
\[(\max(2^{j_4},2^{2k_4}))^{-1}2^{(j_{min}+j_{thd})/2} \lesssim 2^{-j_4}2^{(j_1+j_2+j_3)/2}2^{-k_4},\] 
and if $j_4 \neq j_{max}$, then 
\[(\max(2^{j_4},2^{2k_4}))^{-1}2^{(j_{min}+j_{thd})/2} \lesssim 2^{-j_4}2^{(j_1+j_2+j_3)/2}2^{-3k_4/2}.\] 
Hence by performing all summations over $j_1, j_2, j_3$ and $j_4$, we get
\[\eqref{eq:nonres1-2} \lesssim 2^{3k_3/2}\norm{P_{k_1}u}_{F_{k_1}}\norm{P_{k_2}v}_{F_{k_2}}\norm{P_{k_3}w}_{F_{k_3}},\]
which completes the proof of Lemma \ref{lem:nonres1}.
\end{proof}

\begin{lemma}[High-high-low $\Rightarrow$ high]\label{lem:nonres2}
Let $k_4 \ge 20$, $|k_2-k_4|, |k_3-k_4| \le 5$ and $k_1 \le k_4-10$. Then, we have 
\[\begin{aligned}\norm{P_{k_4}N_3(P_{k_1}u,P_{k_2}v,P_{k_3}w)}_{N_{k_4}} &+ \norm{P_{k_4}N_4(P_{k_1}u,P_{k_2}v,P_{k_3}w)}_{N_{k_4}}\\
&\hspace{3em}\lesssim 2^{k_1/2}\norm{P_{k_1}u}_{F_{k_1}}\norm{P_{k_2}v}_{F_{k_2}}\norm{P_{k_3}w}_{F_{k_3}}.
\end{aligned}\]
\end{lemma}

\begin{proof}
Since we have $j_{max} \ge 5k_4$ from \eqref{eq:high modulation}, once we perform the same procedure as in the proof of Lemma \ref{lem:nonres1}, we can obtain better result than that in Lemma \ref{lem:nonres1}. Furthermore, from \eqref{eq:prop2}, we obtain
\[\norm{f_{k_1}}_{X_{k_1}} = \norm{\ft\left[\eta_0\left((2^{2k_4-2}(t-t_k)\right) \cdot P_{k_1}u\eta_0\left(2^{2k_1}(t-t_k)\right)\right]}_{X_{k_1}} \lesssim \norm{P_{k_1}u}_{F_{k_1}}.\] 
We omit the details. 
\end{proof}

\begin{lemma}[High-high-high $\Rightarrow$ low]\label{lem:nonres3}
Let $k_3 \ge 20$, $|k_1-k_3|, |k_2-k_3| \le 5$ and $k_4 \le k_3-10$. Then, we have 
\begin{equation}\label{eq:nonres3-1}
\begin{aligned}
\norm{P_{k_4}N_3(P_{k_1}u,P_{k_2}v,P_{k_3}w)}_{N_{k_4}} &+ \norm{P_{k_4}N_4(P_{k_1}u,P_{k_2}v,P_{k_3}w)}_{N_{k_4}}\\
&\hspace{-1em}\lesssim k_32^{k_3}2^{-k_4/2}\norm{P_{k_1}u}_{F_{k_1}}\norm{P_{k_2}v}_{F_{k_2}}\norm{P_{k_3}w}_{F_{k_3}}.
\end{aligned}
\end{equation}
\end{lemma}

\begin{proof}
Since $k_4 \le k_3 -10$, one can observe that the $N_{k_4}$-norm is taken on the time intervals of length $2^{-2k_4}$, while each $F_{k_i}$-norm is taken on shorter time intervals of length $2^{-2k_i}$, $i=1,2,3$. Thus, we divide the time interval, which is taken in $N_{k_4}$-norm, into $2^{2k_3-2k_4}$ intervals of length $2^{-2^{2k_3}}$ in order to obtain the right-hand side of \eqref{eq:nonres3-1}. Let $\gamma: \R \to [0,1]$ denote a smooth function supported in $[-1,1]$ with $ \sum_{m\in \Z} \gamma^3(x-m) \equiv 1$. From the definition of $N_{k_4}$-norm, the left-hand side of \eqref{eq:nonres3-1} is dominated by
\begin{equation}\label{eq:nonres3-2}
\begin{split}
\sup_{t_k\in \R}&2^{k_4}2^{2k_3}\Big\|(\tau_4-\mu(n_4) +i 2^{2k_4})^{-1}\mathbf{1}_{I_{k_4}}\\
&\cdot  \sum_{|m| \le C 2^{2k_3-2k_4}} \ft[\eta_0(2^{2k_4}(t-t_k))\gamma (2^{2k_3}(t-t_k)-m)P_{k_1}u]\\
 &\hspace{6em}\ast \ft[\eta_0(2^{2k_4}(t-t_k))\gamma (2^{2k_3}(t-t_k)-m)P_{k_2}v]\\
 &\hspace{6em}\ast \ft[\eta_0(2^{2k_4}(t-t_k))\gamma (2^{2k_3}(t-t_k)-m)P_{k_3}w]\Big\|_{X_{k_4}}.
\end{split}
\end{equation}
As in the proof of Lemma \ref{lem:nonres1}, we have from \eqref{eq:tri-block estimate-a2} that
\[\begin{aligned}\eqref{eq:nonres3-2} \lesssim &2^{4k_3-k_4}\sum_{j_4 \ge 0} \frac{2^{j_4/2}}{\max(2^{j_4},2^{2k_4})}\\
&\qquad\times \sum_{j_1,j_2,j_3 \ge 2k_3} 2^{(j_{min}+j_{med})/2}2^{(k_3+k_4)/2}\norm{f_{k_1,j_1}}_{L_{\tau}^2\ell_n^2}\norm{f_{k_2,j_2}}_{L_{\tau}^2\ell_n^2}\norm{f_{k_3,j_3}}_{L_{\tau}^2\ell_n^2}.
\end{aligned}\]
If $j_4 < 2k_4$, then since $j_4 \le j_1, j_2, j_3$ and $j_{max} \ge 5k_3$, we have
\[(\max(2^{j_4},2^{2k_4}))^{-1}2^{(j_{min}+j_{thd})/2} \lesssim 2^{-2k_4}2^{(j_1+j_2+j_3)/2}2^{-5k_3/2}2^{-k_3}.\]
If $2k_4 \le j_4 < 2k_3 $, then since we still have $j_4 \le j_1, j_2, j_3$ and $j_{max} \ge 5k_3$, and hence we get
\begin{equation}\label{eq:nonres3-3}
(\max(2^{j_4},2^{2k_4}))^{-1}2^{(j_{min}+j_{thd})/2} \lesssim 2^{-j_4}2^{(j_1+j_2+j_3)/2}2^{-5k_3/2}2^{-k_3}.
\end{equation}
Otherwise, we always obtain 
\[(\max(2^{j_4},2^{2k_4}))^{-1}2^{(j_{min}+j_{thd})/2} \lesssim 2^{-j_4/2}2^{(j_1+j_2+j_3)/2}2^{-5k_3/2}2^{-k_3}.\]
In fact, the worst bound comes from the summation of \eqref{eq:nonres3-3} over $j_1, j_2, j_3$ and $j_4$. Indeed,
\begin{align*}
2^{4k_3}&2^{-k_4}\sum_{2k_4 \le j_4 < 2k_3}\frac{2^{j_4/2}}{\max(2^{j_4},2^{2k_4})}\sum_{j_1,j_2,j_3 \ge 2k_3}2^{(j_{min}+j_{med})/2}2^{(k_3+k_4)/2}\prod_{i=1}^{3}\norm{f_{k_i,j_i}}_{L_{\tau}^2\ell_n^2}\\
&\lesssim 2^{4k_3}2^{-k_4}\sum_{2k_4 \le j_4 < 2k_3}2^{j_4/2} \\
& \hspace{5em} \times\sum_{j_1,j_2,j_3 \ge 2k_3}2^{-j_4/2}2^{(j_1+j_2+j_3)/2}2^{-5k_3/2}2^{-k_3}2^{(k_3+k_4)/2}\prod_{i=1}^{3}\norm{f_{k_i,j_i}}_{L_{\tau}^2\ell_n^2}\\
&\lesssim k_32^{k_3}2^{-k_4/2}\norm{P_{k_1}u}_{F_{k_1}}\norm{P_{k_2}v}_{F_{k_2}}\norm{P_{k_3}w}_{F_{k_3}},
\end{align*} 
which complete the proof of Lemma \ref{lem:nonres3}.
\end{proof}

Next, we estimate the \emph{high-low-low $\Rightarrow$ high} non-resonant interaction component in the nonlinear term. As mentioned in Sections \ref{sec:intro} and \ref{sec:trilinear}, the trilinear estimate fails in the standard $X^{s,b}$ space due to the strong nonlinearity and the lack of dispersive smoothing effect. The following lemma shows that the short time length ($\approx (\mbox{frequency})^{-2}$) which is taken in $F_{k}$ or $N_k$ spaces is suitable to estimate the cubic nonlinearity.

\begin{lemma}[High-low-low $\Rightarrow$ high]\label{lem:nonres4}
Let $k_4 \ge 20$, $|k_3-k_4| \le 5$ and $k_1, k_2 \le k_4 -10$. Then, we have 
\[\begin{aligned}\norm{P_{k_4}N_3(P_{k_1}u,P_{k_2}v,P_{k_3}w)}_{N_{k_4}} &+ \norm{P_{k_4}N_4(P_{k_1}u,P_{k_2}v,P_{k_3}w)}_{N_{k_4}}\\
&\hspace{5em}\lesssim 2^{k_{min}/2}\norm{P_{k_1}u}_{F_{k_1}}\norm{P_{k_2}v}_{F_{k_2}}\norm{P_{k_3}w}_{F_{k_3}}.
\end{aligned}\]
\end{lemma}

\begin{proof}
As in the proof of Lemma \ref{lem:nonres1} -- \ref{lem:nonres3}, it is enough to consider 
\begin{equation}\label{eq:nonres4-1}
2^{3k_4}\sum_{j_4 \ge 0} \frac{2^{j_4/2}}{\max(2^{j_4},2^{2k_4})} \sum_{j_1,j_2,j_3 \ge 2k}\norm{\mathbf{1}_{D_{k_4,j_4}}\cdot(f_{k_1,j_1} \ast f_{k_3,j_3} \ast f_{k_3,j_3})}_{L_{\tau_4}^2\ell_{n_4}^2}.
\end{equation}
By the symmetry of $k_1$ and $k_2$ or the definition of $N_4(u,v,w)$, we may assume $k_1 \le k_2$.

\textbf{Case I}. $k_1 \le k_2-10$. In this case, we have from \eqref{eq:high modulation} that
\begin{equation}\label{eq:nonres4-2}
j_{max} \ge 4k_4+k_2.
\end{equation}
We first rewrite the summation over $j_4$ as follows:
\[\begin{aligned}
\sum_{j_4 \ge 0} &= \sum_{j_4 < 2k_4} + \sum_{2k_4 \le j_4 < 4k_4+k_2 -5 }+\sum_{4k_4+k_2-5 \le j_4 < 4k_4+k_2 +5}+ \sum_{4k_4+k_2 +5 \le j_4} \\
&=:\sum_{I}+\sum_{II}+\sum_{III}+\sum_{IV}.
\end{aligned}\]

For the summation over $I$, if $j_2 = j_{max}$, then from \eqref{eq:tri-block estimate-b2.1} or \eqref{eq:tri-block estimate-b2.2}, \eqref{eq:nonres4-1} is bounded by
\[2^{3k_4}\sum_{I} 2^{j_4/2}2^{-2k_4} \sum_{\substack{j_1,j_2,j_3 \ge 2k_4\\j_{sub}\le 4k_4}}2^{(j_1+j_2+j_3+j_4)/2}2^{-(j_{max}+j_{sub})/2}2^{k_2/2}\prod_{i=1}^{3}\norm{f_{k_i,j_i}}_{L_{\tau}^2\ell_n^2}\]
or
\[2^{3k_4}\sum_{I} 2^{j_4/2}2^{-2k_4} \sum_{\substack{j_1,j_2,j_3 \ge 2k_4\\j_{sub}> 4k_4}}2^{(j_1+j_2+j_3+j_4)/2}2^{-j_{max}/2}2^{-2k_4}2^{k_2/2}\prod_{i=1}^{3}\norm{f_{k_i,j_i}}_{L_{\tau}^2\ell_n^2},\]
respectively. In fact, one can easily know that the first bound is worse than the second one, so it suffices to consider the first one. By using \eqref{eq:nonres4-2} and $j_{sub} \ge 2k_4$, and performing the summation over $j_1,j_2,j_3$ and $j_4$, we obtain that
\[\eqref{eq:nonres4-1} \lesssim \norm{P_{k_1}u}_{F_{k_1}}\norm{P_{k_2}v}_{F_{k_2}}\norm{P_{k_3}w}_{F_{k_3}}.\]
If $j_2 \neq j_{max}$, similarly as before, \eqref{eq:nonres4-1} is bounded by
\[2^{3k_4}\sum_{I} 2^{j_4/2}2^{-2k_4} \sum_{\substack{j_1,j_2,j_3 \ge 2k_4\\j_{sub}\le 4k_4}}2^{(j_1+j_2+j_3+j_4)/2}2^{-(j_{max}+j_{sub})/2}2^{k_1/2}\prod_{i=1}^{3}\norm{f_{k_i,j_i}}_{L_{\tau}^2\ell_n^2}\]
or
\[2^{3k_4}\sum_{I} 2^{j_4/2}2^{-2k_4} \sum_{\substack{j_1,j_2,j_3 \ge 2k_4\\j_{sub}> 4k_4}}2^{(j_1+j_2+j_3+j_4)/2}2^{-j_{max}/2}2^{-2k_4}2^{k_1/2}\prod_{i=1}^{3}\norm{f_{k_i,j_i}}_{L_{\tau}^2\ell_n^2},\]
from \eqref{eq:tri-block estimate-b2.3} or \eqref{eq:tri-block estimate-b2.4}, respectively. In this case, it is also enough to consider the first one and then we can obtain by performing the summation over $j_1,j_2,j_3$ and $j_4$ that
\[\eqref{eq:nonres4-1} \lesssim 2^{(k_1-k_2)/2}\norm{P_{k_1}u}_{F_{k_1}}\norm{P_{k_2}v}_{F_{k_2}}\norm{P_{k_3}w}_{F_{k_3}}.\]

For the summation over $II$, we follow the same argument as in the case of summation over $I$. Then we have from \eqref{eq:tri-block estimate-b2.1} that
\begin{align*}  
\eqref{eq:nonres4-1} &\lesssim 2^{3k_4}\sum_{II} 2^{j_4/2}2^{-j_4} \sum_{\substack{j_1,j_2,j_3 \ge 2k_4\\j_{sub}\le 4k_4}}2^{(j_1+j_2+j_3+j_4)/2}2^{-(j_{max}+j_{sub})/2}2^{k_2/2}\prod_{i=1}^{3}\norm{f_{k_i,j_i}}_{L_{\tau}^2\ell_n^2} \\
&\lesssim \norm{P_{k_1}u}_{F_{k_1}}\norm{P_{k_2}v}_{F_{k_2}}\norm{P_{k_3}w}_{F_{k_3}},
\end{align*}
since $j_4 \le j_{sub}$.

For the summation over $III$, we know $j_4 = j_{max}$ and hence, similarly as before, we have from \eqref{eq:tri-block estimate-b2.3} that
\begin{align*}  
\eqref{eq:nonres4-1} &\lesssim 2^{3k_4}\sum_{III} 2^{j_4/2}2^{-j_4} \sum_{\substack{j_1,j_2,j_3 \ge 2k_4\\j_{sub}\le 4k_4}}2^{(j_1+j_2+j_3+j_4)/2}2^{-(j_{max}+j_{sub})/2}2^{k_1/2}\prod_{i=1}^{3}\norm{f_{k_i,j_i}}_{L_{\tau}^2\ell_n^2} \\
&\lesssim 2^{(k_1-k_2)/2}\norm{P_{k_1}u}_{F_{k_1}}\norm{P_{k_2}v}_{F_{k_2}}\norm{P_{k_3}w}_{F_{k_3}}.
\end{align*}

For the last summation, since $j_4, j_{sub} \ge 4k_4+k_2$, we obtain from \eqref{eq:tri-block estimate-b2.4} that
\begin{align*}  
\eqref{eq:nonres4-1} &\lesssim 2^{3k_4}\sum_{IV} 2^{j_4/2}2^{-j_4} \sum_{\substack{j_1,j_2,j_3 \ge 2k_4\\j_{sub} \ge 4k_4+k_2}}2^{(j_1+j_2+j_3+j_4)/2}2^{-j_{max}/2}2^{-2k_4}2^{k_1/2}\prod_{i=1}^{3}\norm{f_{k_i,j_i}}_{L_{\tau}^2\ell_n^2} \\
&\lesssim 2^{-k_4}2^{(k_1-k_2)/2}\norm{P_{k_1}u}_{F_{k_1}}\norm{P_{k_2}v}_{F_{k_2}}\norm{P_{k_3}w}_{F_{k_3}}.
\end{align*}

\textbf{Case II}. $|k_1 - k_2| \le 5$. In this case, we have from \eqref{eq:high modulation} that $j_{max} \ge 4k_4$. 

\noindent As in the proof of \textbf{Case I}, it is enough to consider the case when $(k_i,j_i) \neq (k_{thd},j_{max})$, $2k_4 \le j_4 < 4k_4$ and $j_{sub} \le 4k_4$. Then, by \eqref{eq:tri-block estimate-b2.3}, we have
\[\begin{aligned}
\eqref{eq:nonres4-1} &\lesssim 2^{3k_4}\sum_{2k_4 \le j_4 < 4k_4} 2^{j_4/2}2^{-j_4} \\
&\hspace{5em}\times\sum_{\substack{j_1,j_2,j_3 \ge 2k_4\\j_{sub} \le 4k_4}}2^{(j_1+j_2+j_3+j_4)/2}2^{-(j_{max}+j_{sub})/2}2^{k_{min}/2}\prod_{i=1}^{3}\norm{f_{k_i,j_i}}_{L_{\tau}^2\ell_n^2}\\
&\lesssim 2^{k_{min}/2}\norm{P_{k_1}u}_{F_{k_1}}\norm{P_{k_2}v}_{F_{k_2}}\norm{P_{k_3}w}_{F_{k_3}}.
\end{aligned}\]
Thus, we complete the proof of Lemma \ref{lem:nonres4}.
\end{proof}

Next, we estimate the \emph{high-high-low $\Rightarrow$ low} non-resonant interaction component which has similar frequency interaction phenomenon as in the \emph{high-low-low $\Rightarrow$ high} non-resonant interaction component  in the $L^2$-block estimates. But, as in the proof of Lemma \ref{lem:nonres3}, we lose the regularity gain from longer time interval which is taken in the resulting frequency localized space $N_{k_4}$. 

\begin{lemma}[High-high-low $\Rightarrow$ low]\label{lem:nonres5}
Let $k_3 \ge 20$, $|k_2-k_3| \le 5$ and $k_1,k_4 \le k_3 -10$. Then, we have
\[\begin{aligned}
\norm{P_{k_4}N_3(P_{k_1}u,P_{k_2}v,P_{k_3}w)}_{N_{k_4}} &+ \norm{P_{k_4}N_4(P_{k_1}u,P_{k_2}v,P_{k_3}w)}_{N_{k_4}}\\
&\hspace{1em}\lesssim k_32^{k_3}C(k_1,k_4)\norm{P_{k_1}u}_{F_{k_1}}\norm{P_{k_2}v}_{F_{k_2}}\norm{P_{k_3}w}_{F_{k_3}},
\end{aligned}\]
where
\[C(k_1,k_4)= 
\begin{cases}
2^{-3k_4/2}2^{k_1/2} \hspace{1em}  &, k_1 \le k_4 -10 \\
2^{-k_4} \hspace{1em}  &, k_4 \le k_1 - 10\\
2^{-k_4/2} \hspace{1em}  &,|k_1 -k_4| <10
\end{cases}
.\]
\end{lemma}

\begin{proof}
As in the proof of Lemma \ref{lem:nonres3}, both terms are bounded by
\begin{equation}\label{eq:nonres5-1}
2^{4k_3}2^{-k_4}\sum_{j_4 \ge 0}\frac{2^{j_4/2}}{\max(2^{j_4},2^{2k_4})}\sum_{j_1,j_2,j_3 \ge 2k_3}\norm{\mathbf{1}_{D_{k_4,j_4}}\cdot(f_{k_1,j_1} \ast f_{k_3,j_3} \ast f_{k_3,j_3})}_{L_{\tau_4}^2\ell_{n_4}^2}.
\end{equation}

\textbf{Case I}. $k_1 \le k_4-10$. In this case we have from \eqref{eq:high modulation} that $j_{max} \ge 4k_3 + k_4$.

Similarly as \textbf{Case I} in the proof of Lemma \ref{lem:nonres4}, we can know that the worst case is when 
\begin{equation}\label{eq:nonres5-2}
2k_4 \le j_4 \le 4k_3 \hspace{1em} \mbox{and} \hspace{1em} j_{sub} \le 4k_3,
\end{equation}
and hence, it suffices to perform the summation over \eqref{eq:nonres5-2}. By using \eqref{eq:tri-block estimate-b2.3} and $j_{max} \ge 4k_3+k_4$, and performing the summation over $j_1,j_2,j_3$ and $j_4$ with \eqref{eq:nonres5-2}, we obtain
\begin{align*}
\eqref{eq:nonres5-1} &\lesssim 2^{4k_3}2^{-k_4}\sum_{2k_4 \le j_4 \le 4k_3}2^{-j_4/2}\\
&\hspace{2em}\times\sum_{j_1,j_2,j_3 \ge 2k_3}2^{(j_1+j_2+j_3+j_4)/2}2^{-(j_{max}+j_{sub})/2}2^{k_1/2}\prod_{i=1}^{3}\norm{f_{k_i,j_i}}_{L_{\tau}^2\ell_n^2}\\
&\lesssim k_32^{k_3}2^{-3k_4/2}2^{k_1/2}\norm{P_{k_1}u}_{F_{k_1}}\norm{P_{k_2}v}_{F_{k_2}}\norm{P_{k_3}w}_{F_{k_3}}.
\end{align*}

\textbf{Case II}. $k_4 \le k_1-10$. In this case we have from \eqref{eq:high modulation} that $j_{max} \ge 4k_3 + k_1$.

Similarly as \textbf{Case I}, we can know that the worst case is when 
\begin{equation}\label{eq:nonres5-3}
j_1 = j_{max}, \hspace{1em} 2k_4 \le j_4 \le 4k_3 \hspace{1em} \mbox{and} \hspace{1em} j_{sub} \le 4k_3,
\end{equation}
and hence, it suffices to perform the summation over \eqref{eq:nonres5-3}. By using \eqref{eq:tri-block estimate-b2.1} and $j_{max} \ge 4k_3+k_1$, and performing the summation over $j_1,j_2,j_3$ and $j_4$ with \eqref{eq:nonres5-3}, we obtain
\begin{align*}
\eqref{eq:nonres5-1} &\lesssim 2^{4k_3}2^{-k_4}\sum_{2k_4 \le j_4 \le 4k_3}2^{-j_4/2}\\
&\hspace{5em}\times\sum_{j_1,j_2,j_3 \ge 2k_3}2^{(j_1+j_2+j_3+j_4)/2}2^{-(j_{max}+j_{sub})/2}2^{k_1/2}\prod_{i=1}^{3}\norm{f_{k_i,j_i}}_{L_{\tau}^2\ell_n^2}\\
&\lesssim k_32^{k_3}2^{-k_4}\norm{P_{k_1}u}_{F_{k_1}}\norm{P_{k_2}v}_{F_{k_2}}\norm{P_{k_3}w}_{F_{k_3}}.
\end{align*}

\textbf{Case III}. $|k_1 - k_4| < 10$. In this case we have from \eqref{eq:high modulation} that $j_{max} \ge 4k_3$.

Similarly as \textbf{Case II} in the proof of Lemma \ref{lem:nonres4}, we can know that the worst case is when 
\begin{equation}\label{eq:nonres5-3}
(k_i,j_i) \neq (k_{thd}, j_{max}), \hspace{1em} 2k_4 \le j_4 \le 4k_3 \hspace{1em} \mbox{and} \hspace{1em} j_{sub} \le 4k_3.
\end{equation}
This case is exactly same as the worst case in \textbf{Case I}, while high modulation effect is weaker. By performing similar procedure, we obtain
\begin{align*}
\eqref{eq:nonres5-1} &\lesssim 2^{4k_3}2^{-k_4}\sum_{2k_4 \le j_4 \le 4k_3}2^{-j_4/2}\\
&\hspace{5em}\times\sum_{j_1,j_2,j_3 \ge 2k_3}2^{(j_1+j_2+j_3+j_4)/2}2^{-(j_{max}+j_{sub})/2}2^{k_4/2}\prod_{i=1}^{3}\norm{f_{k_i,j_i}}_{L_{\tau}^2\ell_n^2}\\
&\lesssim k_32^{k_3}2^{-k_4/2}\norm{P_{k_1}u}_{F_{k_1}}\norm{P_{k_2}v}_{F_{k_2}}\norm{P_{k_3}w}_{F_{k_3}},
\end{align*}
which completes the proof of Lemma \ref{lem:nonres5}.
\end{proof}

\begin{lemma}[low-low-low$\Rightarrow$ low]\label{lem:nonres6}
Let $0 \le k_1,k_2,k_3,k_4 \le 200$. Then
\begin{equation}\label{eq:nonres6-1}
\begin{aligned}
\norm{P_{k_4}N_3(P_{k_1}u,P_{k_2}v,P_{k_3}w)}_{N_{k_4}} &+ \norm{P_{k_4}N_4(P_{k_1}u,P_{k_2}v,P_{k_3}w)}_{N_{k_4}}\\
&\hspace{3em}\lesssim \norm{P_{k_1}u}_{F_{k_1}}\norm{P_{k_2}v}_{F_{k_2}}\norm{P_{k_3}w}_{F_{k_3}}.
\end{aligned}
\end{equation}
\end{lemma}
\begin{proof}
Similarly as in the proof of Lemma \ref{lem:nonres1}, we can get \eqref{eq:nonres6-1}.
\end{proof}

Now, we concentrate the quintilinear estimate. By the symmetry of $k_i$'s, $i=1,2,3,4,5$, we may assume that $k_1 \le k_2 \le k_3 \le k_4 \le k_5$. Since we use the short time $X^{s,b}$ space, we have to consider whether the resulting frequency is the highest or not. 

\begin{lemma}\label{lem:quintic1}
Let $k_5 \ge 20$ and $|k_5-k_6| \le 5$. Then, we have
\begin{equation}\label{eq:quintic1-1}
\norm{P_{k_6}N_2(P_{k_1}v_1,P_{k_2}v_2,P_{k_3}v_3,P_{k_4}v_4,P_{k_5}v_5)}_{N_{k_6}} \lesssim 2^{(k_1+k_2+k_3+k_4)/2}2^{-k_6}\prod_{i=1}^{5}\norm{P_{k_i}v_i}_{F_{k_i}}.
\end{equation}
\end{lemma}

\begin{proof}
We follow the similar arguments as in the proof of above lemmas. Then the right-hand side of \eqref{eq:quintic1-1} is bounded by
\begin{equation}\label{eq:quintic1-2}
\begin{aligned}
\sup_{t_k \in \R} &\Big\|(\tau - \mu(n_6)+ i2^{2k_6})^{-1}2^{k_6}\mathbf{1}_{I_{k_6}}(n)\ft\left[\eta_0\left(2^{2k_6-2}(t-t_k)\right)P_{k_1}v_1\right] \\
&\hspace{1em}\ast \ft\left[\eta_0\left(2^{2k_6-2}(t-t_k)\right)P_{k_2}v_2\right] \ast \ft\left[\eta_0\left(2^{2k_6-2}(t-t_k)\right)P_{k_3}v_3\right]\\
&\hspace{1em}\ast \ft\left[\eta_0\left(2^{2k_6-2}(t-t_k)\right)P_{k_4}v_4\right] \ast \ft\left[\eta_0\left(2^{2k_6-2}(t-t_k)\right)P_{k_5}v_5\right]\Big\|_{X_{k_6}}
\end{aligned}
\end{equation}
Set, similarly as in the proof of Lemma \ref{lem:nonres1}, $f_{k_i} = \ft\left[\eta_0\left(2^{2k_i-2}(t-t_k)\right)P_{k_i}v_i\right]$, $i=1,2,3,4,5$, and also, we decompose $f_{k_i}$ into modulation dyadic pieces as $f_{k_i,j_i}(\tau,n) = f_{k_i}(\tau,n)\eta_{j_i}(\tau -\mu(n))$, $j=1,2,3,4,5$, with usual modification $f_{k,\le j}(\tau,n) = f_{k}(\tau,n)\eta_{\le j}(\tau-\mu(n))$. Then, \eqref{eq:quintic1-2} is bounded by
\begin{equation}\label{eq:quintic1-3}
\begin{aligned}
&2^{k_6}\sum_{j_6 \ge 0} \frac{2^{j_6/2}}{\max(2^{j_6},2^{2k_6})} \\
&\hspace{1em}\times\sum_{j_1,j_2j_3,j_4,j_5 \ge 2k_6}\norm{\mathbf{1}_{D_{k_6,j_6}}\cdot(f_{k_1,j_1} \ast f_{k_2,j_2} \ast f_{k_3,j_3} \ast f_{k_4,j_4} \ast f_{k_5,j_5})}_{L_{\tau_6}^2\ell_{n_6}^2}.
\end{aligned}
\end{equation}
We apply the Young's inequality to 
\[\norm{\mathbf{1}_{D_{k_6,j_6}}\cdot(f_{k_1,j_1} \ast f_{k_2,j_2} \ast f_{k_3,j_3} \ast f_{k_4,j_4} \ast f_{k_5,j_5})}_{L_{\tau_6}^2\ell_{n_6}^2}\]
to obtain
\[\begin{aligned}
\eqref{eq:quintic1-3} &\lesssim 2^{k_6}\sum_{j_6 \ge 0} \frac{2^{j_6/2}}{\max(2^{j_6},2^{2k_6})} \sum_{j_1,j_2j_3,j_4,j_5 \ge 2k_6}2^{(j_1+j_2+j_3+j_4+j_5+j_6)/2}\\
&\hspace{10em}\times 2^{-(j_{max}+j_{sub})/2}2^{(k_1+k_2+k_3+k_4)/2}\prod_{i=1}^{5}\norm{f_{k_i,j_i}}_{L_{\tau}^2\ell_n^2}.
\end{aligned}\]
Then by performing the summation over $j_1,j_2,j_3,j_4,j_5$ and $j_6$ with $j_{max}, j_{sub} \ge 2k_6$, we obtain
\[\begin{aligned}
\eqref{eq:quintic1-3} &\lesssim 2^{k_6} \sum_{j_1,j_2j_3,j_4,j_5 \ge 2k_6}2^{(j_1+j_2+j_3+j_4+j_5)/2}2^{-2k_6}2^{(k_1+k_2+k_3+k_4)/2}\prod_{i=1}^{5}\norm{f_{k_i,j_i}}_{L_{\tau}^2\ell_n^2}\\
&\lesssim 2^{-k_6}2^{(k_1+k_2+k_3+k_4)/2}\prod_{i=1}^{5}\norm{P_{k_i}v_i}_{F_{k_i}}.
\end{aligned}\]
Thus, we complete the proof of Lemma \ref{lem:quintic1}

\end{proof}

\begin{lemma}\label{lem:quintic2}
Let $k_5 \ge 20$, $|k_4-k_5| \le 5$ and $k_6 \le k_5 - 10$. Then, we have
\begin{equation}\label{eq:quintic2-1}
\begin{aligned}
&\norm{P_{k_6}N_2(P_{k_1}v_1,P_{k_2}v_2,P_{k_3}v_3,P_{k_4}v_4,P_{k_5}v_5)}_{N_{k_6}}\\
&\hspace{13em} \lesssim k_52^{(k_1+k_2+k_3)/2}2^{-k_6/2}\prod_{i=1}^{5}\norm{P_{k_i}v_i}_{F_{k_i}}.
\end{aligned}
\end{equation}
\end{lemma}

\begin{proof}
Since $k_6 \le k_5-10$, by the same reason as in the proof of Lemma \ref{lem:nonres3}, we further make a partition of interval which is taken in the $N_{k_6}$-norm. Let $\gamma: \R \to [0,1]$ denote a smooth function supported in $[-1,1]$ with $ \sum_{m\in \Z} \gamma^5(x-m) \equiv 1$. Then, from the definition of $N_{k_6}$-norm, the left-hand side of \eqref{eq:quintic2-1} is dominated by
\begin{equation}\label{eq:quintic2-2}
\begin{aligned}
\sup_{t_k \in \R} &\Big\|(\tau - \mu(n_6)+ i2^{2k_6})^{-1}2^{k_6}\mathbf{1}_{I_{k_6}}(n) \\
& \times \sum_{|m| \le C 2^{2k_3-2k_4}} \ft\left[\eta_0\left(2^{2k_6-2}(t-t_k)\right)\gamma (2^{2k_5}(t-t_k)-m)P_{k_1}v_1\right] \\
&\hspace{6em}\ast \ft\left[\eta_0\left(2^{2k_6-2}(t-t_k)\gamma (2^{2k_5}(t-t_k)-m)\right)P_{k_2}v_2\right] \\
&\hspace{6em}\ast \ft\left[\eta_0\left(2^{2k_6-2}(t-t_k)\right)\gamma (2^{2k_5}(t-t_k)-m)P_{k_3}v_3\right]\\
&\hspace{6em}\ast \ft\left[\eta_0\left(2^{2k_6-2}(t-t_k)\right)\gamma (2^{2k_5}(t-t_k)-m)P_{k_4}v_4\right]\\
&\hspace{6em} \ast \ft\left[\eta_0\left(2^{2k_6-2}(t-t_k)\right)\gamma (2^{2k_5}(t-t_k)-m)P_{k_5}v_5\right]\Big\|_{X_{k_6}}
\end{aligned}
\end{equation}
Similarly as before, \eqref{eq:quintic2-2} is bounded by
\begin{equation}\label{eq:quintic2-3}
\begin{aligned}
&2^{k_6}2^{2(k_5-k_6)}\sum_{j_6 \ge 0} \frac{2^{j_6/2}}{\max(2^{j_6},2^{2k_6})} \\
&\hspace{1em}\times \sum_{j_1,j_2j_3,j_4,j_5 \ge 2k_5}\norm{\mathbf{1}_{D_{k_6,j_6}}\cdot(f_{k_1,j_1} \ast f_{k_2,j_2} \ast f_{k_3,j_3} \ast f_{k_4,j_4} \ast f_{k_5,j_5})}_{L_{\tau_6}^2\ell_{n_6}^2}.
\end{aligned}
\end{equation}
We apply the Cauchy-Schwarz inequality to 
\[\norm{\mathbf{1}_{D_{k_6,j_6}}\cdot(f_{k_1,j_1} \ast f_{k_2,j_2} \ast f_{k_3,j_3} \ast f_{k_4,j_4} \ast f_{k_5,j_5})}_{L_{\tau_6}^2\ell_{n_6}^2}\] 
to have
\[\begin{aligned}
\eqref{eq:quintic2-3} &\lesssim 2^{k_6}2^{2(k_5-k_6)}\sum_{j_6 \ge 0} \frac{2^{j_6/2}}{\max(2^{j_6},2^{2k_6})}\sum_{j_1,j_2j_3,j_4,j_5 \ge 2k_5}2^{(j_1+j_2+j_3+j_4+j_5+j_6)/2} \\
&\hspace{9em}\times 2^{-(j_{max}+j_{sub})/2}2^{(k_1+k_2+k_3+k_6)/2}\prod_{i=1}^{5}\norm{f_{k_i,j_i}}_{L_{\tau}^2\ell_n^2}.
\end{aligned}\]
Then by performing the summation over $j_1,j_2,j_3,j_4,j_5$ and $j_6$ with $j_{max}, j_{sub} \ge 2k_5$, we obtain
\[\begin{aligned}
\eqref{eq:quintic2-3} &\lesssim k_52^{k_6}2^{2(k_5-k_6)} \\
&\qquad \times\sum_{j_1,j_2j_3,j_4,j_5 \ge 2k_6}2^{(j_1+j_2+j_3+j_4+j_5)/2}2^{-2k_5}2^{(k_1+k_2+k_3+k_6)/2}\prod_{i=1}^{5}\norm{f_{k_i,j_i}}_{L_{\tau}^2\ell_n^2}\\
&\lesssim k_52^{-k_6/2}2^{(k_1+k_2+k_3)/2}\prod_{i=1}^{5}\norm{P_{k_i}v_i}_{F_{k_i}},
\end{aligned}\]
since the worst term arises in the case when $2k_6 \le j_6 \le 2k_5$. Thus, we complete the proof of Lemma \ref{lem:quintic2}
\end{proof}

\begin{lemma}\label{lem:quintic3}
Let $0 \le k_1, k_2, k_3, k_4, k_5,k_6 \le 200$. Then, we have
\begin{equation}\label{eq:quintic3-1}
\norm{P_{k_6}N_2(P_{k_1}v_1,P_{k_2}v_2,P_{k_3}v_3,P_{k_4}v_4,P_{k_5}v_5)}_{N_{k_6}} \lesssim \prod_{i=1}^{5}\norm{P_{k_i}v_i}_{F_{k_i}}.
\end{equation}
\end{lemma}
\begin{proof}
Similarly as in the proof of Lemma \ref{lem:quintic1}, we can get \eqref{eq:quintic3-1}.
\end{proof}

As a conclusion to this section, we prove the nonlinear estimates for \eqref{eq:5mkdv3} by gathering the block estimates obtained above.

\begin{proposition}\label{prop:nonlinear1}
(a) If $s >1$, $T \in (0,1]$ and $u,v,w, v_i \in F^s(T)$, $i=1,2,3,4,5$, then
\begin{equation}\label{eq:nonlinear1}
\begin{aligned}
\sum_{i=1,3,4} \norm{N_{i}(u,v,w)}_{N^s(T)} &+ \norm{N_2(v_1,v_2,v_3,v_4,v_5)}_{N^s(T)} \\
&\hspace{1em}\lesssim \norm{u}_{F^s(T)}\norm{v}_{F^s(T)}\norm{w}_{F^s(T)} + \prod_{i=1}^{5}\norm{v_i}_{F^s(T)}.
\end{aligned}
\end{equation}

(b) If $T \in (0,1]$, $w, v_5 \in F^0(T)$, $i=1,2,3,4,5$, then
\begin{equation}\label{eq:nonlinear2}
\begin{aligned}
\sum_{i=1,3,4}& \norm{N_{i}(u,v,w)}_{N^0(T)} + \norm{N_2(v_1,v_2,v_3,v_4,v_5)}_{N^0(T)}\\
&\hspace{3em} \lesssim \norm{u}_{F^{1+}(T)}\norm{v}_{F^{1+}(T)}\norm{w}_{F^0(T)} + \prod_{i=1}^{4}\norm{v_i}_{F^{1+}(T)}\norm{v_5}_{F^0(T)}.
\end{aligned}
\end{equation}
\end{proposition} 

\begin{proof}
The proof follows from the dyadic trilinear and quintilinear estimates. See \cite{Guo2012} for a similar proof.
\end{proof}

\section{Energy estimates}\label{sec:energy}
In this section, we will control $\norm{v}_{E^s(T)}$ for \eqref{eq:5mkdv3} by $\norm{v_0}_{H^s}$ and $\norm{v}_{F^s(T)}$. Let us define, for $k \ge 1$, $\psi(n):=n\chi'(n)$ and $\psi_k(n) = \psi(2^{-k}n)$, where $\chi$ is defined in \eqref{eq:cut-off1} and $'$ denote the derivative. Then, we have from the simple observation and the definition of $\chi_k$ that
\[\psi_k(n) = n\chi_k'(n).\]

\begin{remark}
The use of another cut-off function $\psi_k$ is for the second-order Taylor's theorem in the commutator estimates (see Lemma \ref{lem:commutator1}). We, however, do not distinguish between $\psi_k$ and $\chi_k$ in the other estimates, since both multiplier simply play a role of frequency support in the other estimates.
\end{remark}

Recall \eqref{eq:5mkdv3} by slightly modifying from the symmetry of $n_1$ and $n_2$. Then we have
\begin{equation}\label{eq:5mkdv4}
\begin{split}
\pt\wh{v}(n) - i(n^5 + c_1n^3 + c_2n)\wh{v}(n)=&~{}-20in^3|\wh{v}(n)|^2\wh{v}(n)\\
&+6i n\sum_{\N_{5,n}} \wh{v}(n_1)\wh{v}(n_2)\wh{v}(n_3)\wh{v}(n_4)\wh{v}(n_5) \\
&+10in \sum_{\N_{3,n}} \wh{v}(n_1)\wh{v}(n_2)n_3^2\wh{v}(n_3) \\
&+5in \sum_{\N_{3,n}} (n_1+n_2)\wh{v}(n_1)\wh{v}(n_2)n_3\wh{v}(n_3).
\end{split}
\end{equation}
Denote the last three terms in the right-hand side of \eqref{eq:5mkdv4} by $\wh{N}(v)(n)$ only in the proof of Proposition \ref{prop:energy1-2} below. We perform the following procedure for $k \ge 1$,
\[\sum_{n}\chi_k(n)\eqref{eq:5mkdv4} \times \chi_k(-n)\wh{v}(-n) + \overline{\chi_k(n)\eqref{eq:5mkdv4}} \times \chi_k(n)\wh{v}(n),\]
where $\overline{\eqref{eq:5mkdv4}}$ denotes the complex conjugate on \eqref{eq:5mkdv4}. Then we have
\begin{align*}
\pt\norm{P_kv}_{L_x^2}^2 =&~{} - \mbox{Re}\left[12i \sum_{n,\overline{\N}_{5,n}}\chi_k(n)n \wh{v}(n_1)\wh{v}(n_2)\wh{v}(n_3)\wh{v}(n_4)\wh{v}(n_5)\chi_k(n)\wh{v}(n)\right] \\
&-\mbox{Re}\left[20i \sum_{n,\overline{\N}_{3,n}}\chi_k(n)n \wh{v}(n_1)\wh{v}(n_2)n_3^2\wh{v}(n_3)\chi_k(n)\wh{v}(n)\right] \\
&-\mbox{Re}\left[10i \sum_{n,\overline{\N}_{3,n}} \chi_k(n)n(n_1+n_2)\wh{v}(n_1)\wh{v}(n_2)n_3\wh{v}(n_3)\chi_k(n)\wh{v}(n)\right]\\
=:&~{} E_{1} +E_{2} + E_{3}.
\end{align*}

The \emph{high-low-low} interaction component in $\int_0^T E_{2}$ and $\int_0^T E_{3}$ cannot controlled in $F^s(T)$ space directly, due to much more derivatives in the high frequency mode. To overcome this difficulty, it is essential to defined the modified energy as in \cite{Kwon2008-1} and \cite{KP2015}. We, in particular, use the localized version of the modified energy in \cite{KP2015}, by modifying that adapted to the periodic fifth-order mKdV.  

For $k \ge 1$, let us define the new localized energy at $2^k$-frequency piece of $v$ by
\begin{equation}\label{eq:new energy1-1}
\begin{aligned}
E_k(v)(t) =&~{} \norm{P_kv(t)}_{L_x^2}^2 \\
&+ \mbox{Re}\left[\alpha \sum_{n,\overline{\N}_{3,n}}\wh{v}(n_1)\wh{v}(n_2)\psi_k(n_3)\frac{1}{n_3}\wh{v}(n_3)\chi_k(n)\frac1n\wh{v}(n)\right]\\
&+ \mbox{Re}\left[\beta \sum_{n,\overline{\N}_{3,n}}\wh{v}(n_1)\wh{v}(n_2)\chi_k(n_3)\frac{1}{n_3}\wh{v}(n_3)\chi_k(n)\frac1n\wh{v}(n)\right],
\end{aligned}
\end{equation}
where $\alpha$ and $\beta$ are real and will be chosen later. By gathering all localized energy pieces, we define the new modified energy by   
\begin{equation}\label{eq:new energy1-2}
E_{T}^s(v) = \norm{P_0v(0)}_{L_x^2}^2 + \sum_{k \ge 1}2^{2sk} \sup_{t_k \in [-T,T]} E_k(v)(t_k).
\end{equation}

\begin{remark}
As mentioned, the modified energy was first introduced in its current form by Kwon \cite{Kwon2008-1}, and further developed as the localized version by Kenig and Pilod \cite{KP2015}. We also slightly modify them well-adapted to the periodic setting. 
\end{remark}

The following lemma shows that $E_T^s(v)$ and $\norm{v}_{E^s(T)}$ are comparable.
\begin{lemma}\label{lem:comparable energy1-1}
Let $s > \frac12$. Then, there exists $0 < \delta \ll 1$ such that  
\[\frac12\norm{v}_{E^s(T)}^2 \le E_T^s(v) \le \frac32\norm{v}_{E^s(T)}^2,\]
for all $v \in E^s(T) \cap C([-T,T];H^s(\T))$ satisfying $\norm{u}_{L_T^{\infty}H^s(\T)} \le \delta$.
\end{lemma}
\begin{proof}
The proof follows from the Sobolev embedding $H^s(\T) \hookrightarrow L^{\infty}(\T)$, $s > 1/2$. See Lemma 5.1 in \cite{KP2015} for the details.
\end{proof}

We begin with introducing several lemmas which are useful to estimate the modified energy.
\begin{lemma}\label{lem:energy1-1}
Let $T \in (0,1]$, $k_1,k_2,k_3,k_4 \in \Z_+$, and $v_i \in F_{k_i}(T)$, $i=1,2,3,4$. We further assume $k_1 \le k_2 \le k_3 \le k_4$ with $k_4 \ge 10$. Then

(a) For $|k_1 - k_4| \le 5 $, we have
\begin{equation}\label{eq:energy1-1.1}
\left| \sum_{n_4,\overline{\N}_{3,n_4}} \int_0^T  \wh{v}_1(n_1)\wh{v}_2(n_2)\wh{v}_3(n_3)\wh{v}_4(n_4) \; dt\right| \lesssim 2^{k_4/2}\prod_{i=1}^{4}\norm{v_i}_{F_{k_i}(T)}.
\end{equation}

(b) For $|k_2 - k_4| \le 5 $ and $k_1 \le k_4 - 10$, we have
\begin{equation}\label{eq:energy1-1.2}
\left| \sum_{n_4,\overline{\N}_{3,n_4}} \int_0^T \wh{v}_1(n_1)\wh{v}_2(n_2)\wh{v}_3(n_3)\wh{v}_4(n_4) \; dt\right| \lesssim 2^{-k_4}2^{k_1/2}\prod_{i=1}^{4}\norm{v_i}_{F_{k_i}(T)}.
\end{equation}

(c) For $|k_3 - k_4| \le 5$, $k_2 \le k_4 -10$ and $|k_1-k_2| \le 5$, we have
\begin{equation}\label{eq:energy1-1.3}
\left| \sum_{n_4,\overline{\N}_{3,n_4}} \int_0^T \wh{v}_1(n_1)\wh{v}_2(n_2)\wh{v}_3(n_3)\wh{v}_4(n_4) \; dt\right| \lesssim 2^{-k_4}2^{k_1/2}\prod_{i=1}^{4}\norm{v_i}_{F_{k_i}(T)}.
\end{equation}

(d) For $|k_3 - k_4| \le 5$, $k_2 \le k_4 -10$ and $k_1 \le k_2 - 10$, we have
\begin{equation}\label{eq:energy1-1.4}
\left| \sum_{n_4,\overline{\N}_{3,n_4}} \int_0^T\wh{v}_1(n_1)\wh{v}_2(n_2)\wh{v}_3(n_3)\wh{v}_4(n_4) \; dt\right| \lesssim 2^{-k_4}\prod_{i=1}^{4}\norm{v_i}_{F_{k_i}(T)}.
\end{equation}
\end{lemma}

\begin{proof}
We fix extensions $\wt{v}_i \in F_{k_i}$ so that $\norm{\wt{v}_{i}}_{F_{k_i}} \le 2 \norm{v_i}_{F_{k_i}(T)}$, $i=1,2,3,4$. Let $\gamma : \R \to [0,1]$ be a smooth partition of unity function with $\sum_{m \in \Z}\gamma^4(x-m) \equiv 1$, $x \in \R$. Then, we obtain 
\begin{equation}\label{eq:energy1-1.5}
\begin{aligned}
\Big| &\sum_{n_4,\overline{\N}_{3,n_4}} \int_0^T\wh{\wt{v}}_1(n_1)\wh{\wt{v}}_2(n_2)\wh{\wt{v}}_3(n_3)\wh{\wt{v}}_4(n_4) \;dt \Big|\\
&\lesssim \sum_{|m| \lesssim 2^{2k_4}} \Big| \sum_{n_4,\overline{\N}_{3,n_4}} \int_{\R}\left(\gamma(2^{2k_4}t-m)\mathbf{1}_{[0,T]}(t)\wh{\wt{v}}_1(n_1)\right)  \\
&\cdot \left(\gamma(2^{2k_4}t-m)\wh{\wt{v}}_2(n_2)\right)\cdot \left(\gamma(2^{2k_4}t-m)\wh{\wt{v}}_3(n_3)\right) \cdot \left(\gamma(2^{2k_4}t-m)\wh{\wt{v}}_4(n_4)\right) \;dt \Big|
\end{aligned}
\end{equation}
Set
\[A = \set{m:\gamma(2^{2k_4}t-m)\mathbf{1}_{[0,T]}(t) \mbox{ non-zero and } \neq \gamma(2^{2k_4}t-m) }.\]
Then, the summation over $m \lesssim 2^{2k_4}$ in the right-hand side of \eqref{eq:energy1-1.5} is divided into $A$ and $A^c$. Since $|A| \le 4$, we can easily handle (see \cite{Guo2012} for the details) the right-hand side of \eqref{eq:energy1-1.5} on $A$ by showing
\[\sup_{j \in \Z_+}2^{j/2}\norm{\eta_j(\tau-\mu(n)) \cdot \ft[\mathbf{1}_{[0,1]}(t)\gamma(2^{2k_4}t-m)\wt{v}_1]}_{L_{\tau}^2\ell_n^2} \lesssim \norm{\gamma(2^{2k_4}t-m)\wt{v}_1}_{X_{k_1}}.\]
Hence, we only handle the summation on $A^c$, that is, 
\[\gamma(2^{2k_4}t-m)\mathbf{1}_{[0,T]}(t)\wh{\wt{v}}_1(n_1) = \gamma(2^{2k_4}t-m)\wh{\wt{v}}_1(n_1).\] 
Let $f_{k_i} = \ft[\gamma(2^{2k_4}t-m)\wh{\wt{v}}_i(n_i)]$ and $f_{k_i,j_i} = \eta_{j_i}(\tau - \mu(n))f_{k_i}$, $i=1,2,3,4$. By parseval's identity and \eqref{eq:prop1}, the right-hand side of \eqref{eq:energy1-1.5} is dominated by
\begin{equation}\label{eq:energy1-1.6}
\sup_{m \in A^c} 2^{2k_4} \sum_{j_1,j_2,j_3,j_4 \ge 2k_4} |J(f_{k_1,j_1},f_{k_2,j_2},f_{k_3,j_3},f_{k_4,j_4})|.
\end{equation}

(a) By the support property \eqref{eq:support property}, we know $j_{max} \ge 3k_4$. Then, we use \eqref{eq:tri-block estimate-a1} to obtain that
\begin{align*}
\eqref{eq:energy1-1.6} &\lesssim 2^{2k_4} \sum_{j_1,j_2,j_3,j_4 \ge 2k_4}2^{(j_{min}+j_{thd})/2}2^{k_4}\prod_{i=1}^{4} \norm{f_{k_i,j_i}}_{L_{\tau}^2\ell_n^2}\\
&\lesssim 2^{k_4/2}\norm{v_1}_{F_{k_1}(T)}\norm{v_2}_{F_{k_2}(T)}\norm{v_3}_{F_{k_3}(T)}\norm{v_4}_{F_{k_4}(T)}.
\end{align*}

(b) We use the same block-estimate \eqref{eq:tri-block estimate-a1} and argument in (a) with $j_{max} \ge 5k_4$ to have 
\begin{align*}
\eqref{eq:energy1-1.6} &\lesssim 2^{2k_4} \sum_{j_1,j_2,j_3,j_4 \ge 2k_4}2^{(j_{min}+j_{thd})/2}2^{k_1/2}2^{k_4/2}\prod_{i=1}^{4} \norm{f_{k_i,j_i}}_{L_{\tau}^2\ell_n^2}\\
&\lesssim 2^{-k_4}2^{k_1/2}\norm{v_1}_{F_{k_1}(T)}\norm{v_2}_{F_{k_2}(T)}\norm{v_3}_{F_{k_3}(T)}\norm{v_4}_{F_{k_4}(T)}.
\end{align*}

(c) Since $|k_1 - k_2| \le 5$, the case $(k_i,j_i) = (k_{thd},j_{max})$ never happens. Moreover, it suffices to consider only the case when $j_{sub} \le 4k_4$ the case, since $j_{sub} \le 4k_4$ gives the worst bound in the block estimates among other cases. Then by using \eqref{eq:tri-block estimate-b1.3} and $j_{max} \ge 4k_4$, we obtain
\begin{align*}
\eqref{eq:energy1-1.6} &\lesssim 2^{2k_4} \sum_{\substack{j_1,j_2,j_3,j_4 \ge 2k_4\\j_{sub} \le 4k_4}}2^{(j_1+j_2+j_3+j_4)/2}2^{-(j_{max}+j_{sub})/2}2^{k_1/2}\prod_{i=1}^{4} \norm{f_{k_i,j_i}}_{L_{\tau}^2\ell_n^2}\\
&\lesssim 2^{-k_4}2^{k_1/2}\norm{v_1}_{F_{k_1}(T)}\norm{v_2}_{F_{k_2}(T)}\norm{v_3}_{F_{k_3}(T)}\norm{v_4}_{F_{k_4}(T)}.
\end{align*}

(d) In this case, we observe that $j_{max} \ge 4k_4 + k_2$. Similarly, the worst bound of $|J(f_{k_1,j_1},f_{k_2,j_2},f_{k_3,j_3},f_{k_4,j_4})|$ should appear when $j_2 = j_{max}$ and $j_{sub} \le 4k_4$ hold. Hence, by \eqref{eq:tri-block estimate-b1.1}, we have 
\begin{align*}
\eqref{eq:energy1-1.6} &\lesssim 2^{2k_4} \sum_{\substack{j_1,j_2,j_3,j_4 \ge 2k_4\\ j_2 = j_{max}\\j_{sub} \le 4k_4}}2^{(j_1+j_2+j_3+j_4)/2}2^{-(j_{max}+j_{sub})/2}2^{k_2/2}\prod_{i=1}^{4} \norm{f_{k_i,j_i}}_{L_{\tau}^2\ell_n^2}\\
&\lesssim 2^{-k_4}\norm{v_1}_{F_{k_1}(T)}\norm{v_2}_{F_{k_2}(T)}\norm{v_3}_{F_{k_3}(T)}\norm{v_4}_{F_{k_4}(T)}.
\end{align*}
Therefore, we finish the proof of Lemma \ref{lem:energy1-1}.
\end{proof}

The next lemma is a kind of commutator estimate which will be helpful to handle bad terms $\int_0^T E_{2}$ and $\int_0^T E_{3}$ in the original energy.

\begin{lemma}\label{lem:commutator1}
Let $T \in (0,1]$, $k,k_1,k_2 \in \Z_+$ satisfying $k_1,k_2 \le k -10$, $u_i \in F_{k_i}(T)$, $i=1,2$, and $v \in F^0(T)$. Then, we have
\begin{equation}\label{eq:commutator1-1}
\begin{aligned}
\Big|\sum_{n,\overline{\N}_{3,n}}&\int_0^T \chi_k(n)n[\chi_{k_1}(n_1)\wh{u}_1(n_1)\chi_{k_2}(n_2)\wh{u}_2(n_2)n_3^2\wh{v}(n_3)]\chi_k(n)\wh{v}(n) \;dt\\
&\hspace{-2em}+ \frac12\sum_{n,\overline{\N}_{3,n}}\int_0^T (n_1+n_2)\chi_{k_1}(n_1)\wh{u}_1(n_1)\chi_{k_2}(n_2)\wh{u}_2(n_2)\chi_k(n_3)n_3\wh{v}(n_3)\chi_k(n)n\wh{v}(n) \;dt\\
&\hspace{-2em}- \sum_{n,\overline{\N}_{3,n}}\int_0^T (n_1+n_2)\chi_{k_1}(n_1)\wh{u}_1(n_1)\chi_{k_2}(n_2)\wh{u}_2(n_2)\psi_k(n_3)n_3\wh{v}(n_3)\chi_k(n)n\wh{v}(n) \;dt \Big|\\
\lesssim&~{} 2^{2k_2} \norm{P_{k_1}u_1}_{F_{k_1}(T)}\norm{P_{k_2}u_2}_{F_{k_2}(T)}\sum_{|k-k'|\le 5} \norm{P_{k'}v}_{F_{k'}(T)}^2,
\end{aligned}
\end{equation}
and
\begin{equation}\label{eq:commutator1-2}
\begin{aligned}
\Big|&\sum_{n,\overline{\N}_{3,n}}\int_0^T \chi_k(n)[(n_1+n_2)\chi_{k_1}(n_1)\wh{u}_1(n_1)\chi_{k_2}(n_2)\wh{u}_2(n_2)n_3^2\wh{v}(n_3)]\chi_k(n)\wh{v}(n)\; dt\\
&+ \sum_{n,\overline{\N}_{3,n}}\int_0^T (n_1+n_2)\chi_{k_1}(n_1)\wh{u}_1(n_1)\chi_{k_2}(n_2)\wh{u}_2(n_2)\chi_k(n_3)n_3\wh{v}(n_3)\chi_k(n)n\wh{v}(n)\;dt \Big|\\
&\hspace{9em}\lesssim 2^{2k_2} \norm{P_{k_1}u_1}_{F_{k_1}(T)}\norm{P_{k_2}u_2}_{F_{k_2}(T)}\sum_{|k-k'|\le 5} \norm{P_{k'}v}_{F_{k'}(T)}^2,
\end{aligned}
\end{equation}
\end{lemma}

\begin{proof}
We first consider \eqref{eq:commutator1-1}. From $n_1+n_2+n_3+n_4 = 0$ and the symmetry of $n_3,n$, we have
\begin{align*}
\mbox{LHS of }\eqref{eq:commutator1-1} &= \Big|\sum_{n,\overline{\N}_{3,n}}\int_0^T[\chi_k(n)n_3^2 - \chi_k(n_3)n_3^2 - (n_1+n_2)n_3\psi_k(n_3)]\\
&\hspace{7em}\times\chi_{k_1}(n_1)\wh{u}_1(n_1)\chi_{k_2}(n_2)\wh{u}_2(n_2)\wh{v}(n_3)\chi_k(n)n\wh{v}(n)\;dt \Big|\\
&= \Big|\sum_{n,\overline{\N}_{3,n}}\int_0^T\left[\frac{\chi_k(n) - \chi_k(n_3) - (n_1+n_2)\chi_k'(n_3)}{(n_1+n_2)^2}\cdot n_3^2\right]\\
&\hspace{3em}\times(n_1+n_2)^2\chi_{k_1}(n_1)\wh{u}_1(n_1)\chi_{k_2}(n_2)\wh{u}_2(n_2)\wh{v}(n_3)\chi_k(n)n\wh{v}(n)\;dt \Big|. 
\end{align*}
Since both $\chi_k$ and $\chi_k'$ are even functions, $-n_3 = n + (n_1+n_2)$, $|n|\sim|n_3|$ and $\chi_k''(n) = O(\chi_k(n)/n^2)$ due to \eqref{eq:regularity}, we know from the Taylor's theorem that 
\[\left|\frac{\chi_k(n) - \chi_k(n_3) - (n_1+n_2)\chi_k'(n_3)}{(n_1+n_2)^2}\cdot n_3^2\right| \lesssim 1.\]
Hence by the same way as in the proof of Lemma \ref{lem:energy1-1} (c) and (d), we have from $2^{j_{max}} \ge 2^{4k_4}|n_1+n_2|$ that\footnote{In the proof of Lemma \ref{lem:energy1-1} (c), we, in fact, obtain the additional gain $|n_1+n_2|^{-1/2}$. Hence, it covers $|n_1+n_2|^{1/2}$ and we can obtain \eqref{eq:commutator1-1.1.1}.}
\begin{equation}\label{eq:commutator1-1.1.1}
\mbox{LHS of }\eqref{eq:commutator1-1} \lesssim 2^{2k_2}\norm{P_{k_1}u_1}_{F_{k_1}(T)}\norm{P_{k_2}u_2}_{F_{k_2}(T)}\sum_{|k-k'|\le 5} \norm{P_{k'}v}_{F_{k'}(T)}^2,
\end{equation}
for $|k_1-k_2| \le 5$, and 
\[\mbox{LHS of }\eqref{eq:commutator1-1} \lesssim 2^{2k_2}\norm{P_{k_1}u_1}_{F_{k_1}(T)}\norm{P_{k_2}u_2}_{F_{k_2}(T)}\sum_{|k-k'|\le 5} \norm{P_{k'}v}_{F_{k'}(T)}^2,\]
for $k_1 \le k_2 - 10$.

Next, we consider \eqref{eq:commutator1-2}. Since $n = -n_3 -(n_1+n_2)$, we have 
\[\begin{aligned}
&\sum_{n,\overline{\N}_{3,n}}\int_0^T (n_1+n_2)\chi_{k_1}(n_1)\wh{u}_1(n_1)\chi_{k_2}(n_2)\wh{u}_2(n_2)\chi_k(n_3)n_3\wh{v}(n_3)\chi_k(n)n\wh{v}(n)\;dt\\
=&-\sum_{n,\overline{\N}_{3,n}}\int_0^T (n_1+n_2)^2\chi_{k_1}(n_1)\wh{u}_1(n_1)\chi_{k_2}(n_2)\wh{u}_2(n_2)\chi_k(n_3)n_3\wh{v}(n_3)\chi_k(n)\wh{v}(n)\;dt\\
&-\sum_{n,\overline{\N}_{3,n}}\int_0^T (n_1+n_2)\chi_{k_1}(n_1)\wh{u}_1(n_1)\chi_{k_2}(n_2)\wh{u}_2(n_2)\chi_k(n_3)n_3^2\wh{v}(n_3)\chi_k(n)\wh{v}(n)\;dt,
\end{aligned}\]
and similarly as before, we have 
\begin{align*}
\sum_{n,\overline{\N}_{3,n}}&\int_0^T \chi_k(n)[(n_1+n_2)\chi_{k_1}(n_1)\wh{u}_1(n_1)\chi_{k_2}(n_2)\wh{u}_2(n_2)n_3^2\wh{v}(n_3)]\chi_k(n)\wh{v}(n)\\
&- \sum_{n,\overline{\N}_{3,n}}\int_0^T (n_1+n_2)\chi_{k_1}(n_1)\wh{u}_1(n_1)\chi_{k_2}(n_2)\wh{u}_2(n_2)\chi_k(n_3)n_3^2\wh{v}(n_3)]\chi_k(n)\wh{v}(n)\\
=&~{} \sum_{n,\overline{\N}_{3,n}}\int_0^T\left[\frac{\chi_k(n) - \chi_k(n_3)}{(n_1+n_2)}\cdot n_3\right]\\
&\hspace{5em}\times(n_1+n_2)^2\chi_{k_1}(n_1)\wh{u}_1(n_1)\chi_{k_2}(n_2)\wh{u}_2(n_2)n_3\wh{v}(n_3)\chi_k(n)\wh{v}(n)\;dt,
\end{align*}
with
\[\left|\frac{\chi_k(n) - \chi_k(n_3)}{(n_1+n_2)}\cdot n_3\right| \lesssim 1.\]
Again we use \eqref{eq:energy1-1.3} and \eqref{eq:energy1-1.4} so that
\[\mbox{LHS of }\eqref{eq:commutator1-2} \lesssim 2^{2k_2}\norm{P_{k_1}u_1}_{F_{k_1}(T)}\norm{P_{k_2}u_2}_{F_{k_2}(T)}\sum_{|k-k'|\le 5} \norm{P_{k'}v}_{F_{k'}(T)}^2,\]
for both $|k_1-k_2| \le 5$ and $k_1 \le k_2 -10$ cases.
\end{proof}

\begin{remark}
By using the same way, we also have
\begin{equation}\label{eq:commutator1-3}
\begin{aligned}
\Big|\sum_{n,\overline{\N}_{3,n}}&\int_0^T \chi_k(n)[(n_1+n_2)\chi_{k_1}(n_1)\wh{u}_1(n_1)\chi_{k_2}(n_2)\wh{u}_2(n_2)n_3\wh{v}(n_3)]\chi_k(n)n\wh{v}(n)\; dt\\
&\hspace{-2em}- \sum_{n,\overline{\N}_{3,n}}\int_0^T (n_1+n_2)\chi_{k_1}(n_1)\wh{u}_1(n_1)\chi_{k_2}(n_2)\wh{u}_2(n_2)\chi_k(n_3)n_3\wh{v}(n_3)\chi_k(n)n\wh{v}(n)\;dt \Big|\\
&\hspace{7em}\lesssim 2^{2k_2} \norm{P_{k_1}u_1}_{F_{k_1}(T)}\norm{P_{k_2}u_2}_{F_{k_2}(T)}\sum_{|k-k'|\le 5} \norm{P_{k'}v}_{F_{k'}(T)}^2.
\end{aligned}
\end{equation}
This commutator estimate will be used in the proof of Proposition \ref{prop:energy1-3}. 
\end{remark}

We, now, ready to show the energy estimate. 
\begin{proposition}\label{prop:energy1-2}
Let $s > 2$ and $T \in (0,1]$. Then, for the solution $v \in C([-T,T];H^{\infty}(\T))$ to \eqref{eq:5mkdv4}, we have
\[E_T^s(v) \lesssim (1+ \norm{v_0}_{H^s}^2)\norm{v_0}_{H^s}^2 + (1 + \norm{v}_{F^{\frac12+}(T)}^2 + \norm{v}_{F^{\frac12+}(T)}^4)\norm{v}_{F^{2+}(T)}^2\norm{v}_{F^s(T)}^2.\]
\end{proposition}

\begin{proof}
For any $k\in \Z_+$ and $t \in [-T,T]$, recall the localized modified energy \eqref{eq:new energy1-1}
\[\begin{aligned}
E_k(v)(t) =&~{} \norm{P_kv(t)}_{L_x^2}^2 + \mbox{Re}\left[\alpha \sum_{n,\overline{\N}_{3,n}}\wh{v}(n_1)\wh{v}(n_2)\psi_k(n_3)\frac{1}{n_3}\wh{v}(n_3)\chi_k(n)\frac1n\wh{v}(n)\right]\\
&+ \mbox{Re}\left[\beta \sum_{n,\overline{\N}_{3,n}}\wh{v}(n_1)\wh{v}(n_2)\chi_k(n_3)\frac{1}{n_3}\wh{v}(n_3)\chi_k(n)\frac1n\wh{v}(n)\right]\\
=:&~{}I(t) + II(t) + III(t)
\end{aligned}\]
and 
\begin{equation*}
\begin{split}
\pt\norm{P_kv}_{L_x^2}^2 =&~{} - \mbox{Re}\left[12i \sum_{n,\overline{\N}_{5,n}}\chi_k(n)n \wh{v}(n_1)\wh{v}(n_2)\wh{v}(n_3)\wh{v}(n_4)\wh{v}(n_5)\chi_k(n)\wh{v}(n)\right] \\
&-\mbox{Re}\left[20i \sum_{n,\overline{\N}_{3,n}}\chi_k(n)n \wh{v}(n_1)\wh{v}(n_2)n_3^2\wh{v}(n_3)\chi_k(n)\wh{v}(n)\right] \\
&-\mbox{Re}\left[10i \sum_{n,\overline{\N}_{3,n}} \chi_k(n)n(n_1+n_2)\wh{v}(n_1)\wh{v}(n_2)n_3\wh{v}(n_3)\chi_k(n)\wh{v}(n)\right]\\
=:&~{}E_{1}
\end{split}
\end{equation*}

By using the symmetry of $n_1,n_2,n_3$ and $n_1+n_2+n_3+n=0$, the last term in $E_{1}$ can be rewritten as
\begin{equation*}
\begin{split}
-\mbox{Re}&\left[10i \sum_{n,\overline{\N}_{3,n}} \chi_k(n)n(n_1+n_2)\wh{v}(n_1)\wh{v}(n_2)n_3\wh{v}(n_3)\chi_k(n)\wh{v}(n)\right]\\
&\hspace{3em}=\mbox{Re}\left[20i \sum_{n,\overline{\N}_{3,n}} \chi_k(n)n_1\wh{v}(n_1)n_2\wh{v}(n_2)n_3\wh{v}(n_3)\chi_k(n)\wh{v}(n)\right]\\
&\hspace{4em}+\mbox{Re}\left[20i \sum_{n,\overline{\N}_{3,n}} \chi_k(n)(n_1+n_2)\wh{v}(n_1)\wh{v}(n_2)n_3^2\wh{v}(n_3)\chi_k(n)\wh{v}(n)\right].
\end{split}
\end{equation*}

We differentiate $II(t)$ with respect to $t$, respectively. Then, we have from \eqref{eq:5mkdv4} that
\[\begin{aligned}
\frac{d}{dt} II(t) =&~{} \mbox{Re}\Big[\alpha i \sum_{n,\overline{\N}_{3,n}}(\mu(n_1)+\mu(n_2)+\mu(n_3)+\mu(n))\\
&\hspace{7em}\times \wh{v}(n_1)\wh{v}(n_2)\psi_k(n_3)\frac{1}{n_3}\wh{v}(n_3)\chi_k(n)\frac1n\wh{v}(n)\Big]\\
&+ \mbox{Re}\Big[\alpha \sum_{n,\overline{\N}_{3,n}}\wh{N}(v)(n_1)\wh{v}(n_2)\psi_k(n_3)\frac{1}{n_3}\wh{v}(n_3)\chi_k(n)\frac1n\wh{v}(n)\\
&\hspace{7em}+\wh{v}(n_1)\wh{N}(v)(n_2)\psi_k(n_3)\frac{1}{n_3}\wh{v}(n_3)\chi_k(n)\frac1n\wh{v}(n)\\
&\hspace{7em}+\wh{v}(n_1)\wh{v}(n_2)\psi_k(n_3)\frac{1}{n_3}\wh{N}(v)(n_3)\chi_k(n)\frac1n\wh{v}(n)\\
&\hspace{7em}+\wh{v}(n_1)\wh{v}(n_2)\psi_k(n_3)\frac{1}{n_3}\wh{v}(n_3)\chi_k(n)\frac1n\wh{N}(v)(n)\Big]\\
&+ \mbox{Re}\Big[-20\alpha i \sum_{n,\overline{\N}_{3,n}}n_1^3|\wh{v}(n_1)|^2\wh{v}(n_1)\wh{v}(n_2)\psi_k(n_3)\frac{1}{n_3}\wh{v}(n_3)\chi_k(n)\frac1n\wh{v}(n)\\
&\hspace{7em}+\wh{v}(n_1)n_2^3|\wh{v}(n_2)|^2\wh{v}(n_2)\psi_k(n_3)\frac{1}{n_3}\wh{v}(n_3)\chi_k(n)\frac1n\wh{v}(n)\\
&\hspace{7em}+\wh{v}(n_1)\wh{v}(n_2)\psi_k(n_3)n_3^2|\wh{v}(n_3)|^2\wh{v}(n_3)\chi_k(n)\frac1n\wh{v}(n)\\
&\hspace{7em}+\wh{v}(n_1)\wh{v}(n_2)\psi_k(n_3)\frac{1}{n_3}\wh{v}(n_3)\chi_k(n)n^2|\wh{v}(n)|^2\wh{v}(n)\Big].
\end{aligned}\]

We note the following algebraic laws
\[\begin{aligned}
(a+b+c)^5 =&~{} 5(a^4b+a^4c + ab^4+b^4c+ac^4+bc^4)\\
&+ 10(a^3b^2 + a^3c^2 + a^2b^3 + b^3c^2 + a^2c^3+b^2c^3)\\ 
&+ 20(a^3bc + ab^3c +abc^3) + 30(a^2b^2c + a^2bc^2 +ab^2c^2) +a^5 + b^5 + c^5\\
(a+b+c)^3 =&~{} a^3+b^3+c^3 +3(a^2b +a^2c + ab^2+ b^2c + ac^2 + bc^2) + 6abc.
\end{aligned}\]
The symmetry of $n_1$ and $n_2$ yields 
\[\frac{d}{dt} II(t) =E_{2,1} + E_{2,2} + E_{2,3} + E_{2,4} =: E_{2},\]
where
\[\begin{aligned}
E_{2,1} &= \mbox{Re}\Big[\alpha i \sum_{n,\overline{\N}_{3,n}}\big\{10n_1n_2^3(n_3+n) + 5n_1^2n_2^2(n_3+n) + 30n_1n_2^2n_3n \\
&\hspace{3em}+ 10 n_2^3n_3n - 5(n_1+n_2)n_3^2n^2\big\} \wh{v}(n_1)\wh{v}(n_2)\psi_k(n_3)\frac{1}{n_3}\wh{v}(n_3)\chi_k(n)\frac1n\wh{v}(n)\Big],
\end{aligned}\]

\[\begin{aligned}
E_{2,2} = \mbox{Re}\Big[c_1\alpha i \sum_{n,\overline{\N}_{3,n}}\big\{3n_1n_2(n_3+n) + 6n_2n_3n\big\} \wh{v}(n_1)\wh{v}(n_2)\psi_k(n_3)\frac{1}{n_3}\wh{v}(n_3)\chi_k(n)\frac1n\wh{v}(n)\Big],
\end{aligned}\]

\[\begin{aligned}
E_{2,3} &=\mbox{Re}\Big[\alpha \sum_{n,\overline{\N}_{3,n}}\wh{N}(v)(n_1)\wh{v}(n_2)\psi_k(n_3)\frac{1}{n_3}\wh{v}(n_3)\chi_k(n)\frac1n\wh{v}(n)\\
&\hspace{10em}+\wh{v}(n_1)\wh{N}(v)(n_2)\psi_k(n_3)\frac{1}{n_3}\wh{v}(n_3)\chi_k(n)\frac1n\wh{v}(n)\\
&\hspace{10em}+\wh{v}(n_1)\wh{v}(n_2)\psi_k(n_3)\frac{1}{n_3}\wh{N}(v)(n_3)\chi_k(n)\frac1n\wh{v}(n)\\
&\hspace{10em}+\wh{v}(n_1)\wh{v}(n_2)\psi_k(n_3)\frac{1}{n_3}\wh{v}(n_3)\chi_k(n)\frac1n\wh{N}(v)(n)\Big]
\end{aligned}\]
and
\[\begin{aligned}
E_{2,4} &= \mbox{Re}\Big[-20\alpha i \sum_{n,\overline{\N}_{3,n}}n_1^3|\wh{v}(n_1)|^2\wh{v}(n_1)\wh{v}(n_2)\psi_k(n_3)\frac{1}{n_3}\wh{v}(n_3)\chi_k(n)\frac1n\wh{v}(n)\\
&\hspace{10em}+\wh{v}(n_1)n_2^3|\wh{v}(n_2)|^2\wh{v}(n_2)\psi_k(n_3)\frac{1}{n_3}\wh{v}(n_3)\chi_k(n)\frac1n\wh{v}(n)\\
&\hspace{10em}+\wh{v}(n_1)\wh{v}(n_2)\psi_k(n_3)n_3^2|\wh{v}(n_3)|^2\wh{v}(n_3)\chi_k(n)\frac1n\wh{v}(n)\\
&\hspace{10em}+\wh{v}(n_1)\wh{v}(n_2)\psi_k(n_3)\frac{1}{n_3}\wh{v}(n_3)\chi_k(n)n^2|\wh{v}(n)|^2\wh{v}(n)\Big].
\end{aligned}\]

Similarly, we also have with the symmetry of $n_3$ and $n$ that
\[\frac{d}{dt} III(t) =E_{3,1} + E_{3,2} + E_{3,3} +E_{3,4}=: E_{3},\]
where
\[\begin{aligned}
E_{3,1} &= \mbox{Re}\Big[\beta i \sum_{n,\overline{\N}_{3,n}}\big\{20n_1n_2^3n_3 + 10n_1^2n_2^2n_3 + 30n_1n_2^2n_3n \\
&\hspace{3em}+ 10 n_2^3n_3n - 5(n_1+n_2)n_3^2n^2\big\} \wh{v}(n_1)\wh{v}(n_2)\chi_k(n_3)\frac{1}{n_3}\wh{v}(n_3)\chi_k(n)\frac1n\wh{v}(n)\Big],
\end{aligned}\]

\[\begin{aligned}
E_{3,2} = \mbox{Re}\Big[c_1\beta i \sum_{n,\overline{\N}_{3,n}}\big\{6n_1n_2n_3 + 6n_2n_3n\big\} \wh{v}(n_1)\wh{v}(n_2)\psi_k(n_3)\frac{1}{n_3}\wh{v}(n_3)\chi_k(n)\frac1n\wh{v}(n)\Big],
\end{aligned}\]

\[\begin{aligned}
E_{3,3} &=\mbox{Re}\Big[\beta \sum_{n,\overline{\N}_{3,n}}\wh{N}(v)(n_1)\wh{v}(n_2)\chi_k(n_3)\frac{1}{n_3}\wh{v}(n_3)\chi_k(n)\frac1n\wh{v}(n)\\
&\hspace{10em}+\wh{v}(n_1)\wh{N}(v)(n_2)\chi_k(n_3)\frac{1}{n_3}\wh{v}(n_3)\chi_k(n)\frac1n\wh{v}(n)\\
&\hspace{10em}+\wh{v}(n_1)\wh{v}(n_2)\chi_k(n_3)\frac{1}{n_3}\wh{N}(v)(n_3)\chi_k(n)\frac1n\wh{v}(n)\\
&\hspace{10em}+\wh{v}(n_1)\wh{v}(n_2)\chi_k(n_3)\frac{1}{n_3}\wh{v}(n_3)\chi_k(n)\frac1n\wh{N}(v)(n)\Big]
\end{aligned}\]
and
\[\begin{aligned}
E_{3,4} &= \mbox{Re}\Big[-20\beta i \sum_{n,\overline{\N}_{3,n}}n_1^3|\wh{v}(n_1)|^2\wh{v}(n_1)\wh{v}(n_2)\chi_k(n_3)\frac{1}{n_3}\wh{v}(n_3)\chi_k(n)\frac1n\wh{v}(n)\\
&\hspace{10em}+\wh{v}(n_1)n_2^3|\wh{v}(n_2)|^2\wh{v}(n_2)\chi_k(n_3)\frac{1}{n_3}\wh{v}(n_3)\chi_k(n)\frac1n\wh{v}(n)\\
&\hspace{10em}+\wh{v}(n_1)\wh{v}(n_2)\chi_k(n_3)n_3^2|\wh{v}(n_3)|^2\wh{v}(n_3)\chi_k(n)\frac1n\wh{v}(n)\\
&\hspace{10em}+\wh{v}(n_1)\wh{v}(n_2)\chi_k(n_3)\frac{1}{n_3}\wh{v}(n_3)\chi_k(n)n^2|\wh{v}(n)|^2\wh{v}(n)\Big].
\end{aligned}\]

We fix $t_k \in [0,T]$. By integrating $\pt E_k(v)(t)$ with respect to $t$ from $0$ to $t_k$, we have
\begin{equation}\label{eq:energy1-2.3}
E_k(v)(t_k) - E_{k}(v)(0) \le \left|\int_0^{t_k} E_{1} + E_{2} + E_{3} \; dt \right|.
\end{equation}

We estimate the right-hand side of \eqref{eq:energy1-2.3} by dividing it into several cases. First, we choose $\alpha = -4$ and $\beta = -2$ to use Lemma \ref{lem:commutator1}. Then for each $k\ge 1$, we have 
\[\left|\int_0^{t_k} E_{1} + E_{2,1} + E_{3,1} \; dt \right| \lesssim \sum_{i=1}^{8}A_i(k),\]
where
\[\begin{aligned}
A_1(k) =&~{} \sum_{0\le k_1 ,k_2 \le k-10}\Big|\sum_{n,\overline{\N}_{3,n}}\int_0^{t_k} \chi_k(n)n[\chi_{k_1}(n_1)\wh{v}(n_1)\chi_{k_2}(n_2)\wh{v}(n_2)\\
&\hspace{17em}\times n_3^2\wh{v}(n_3)]\chi_k(n)\wh{v}(n) \;dt\\
&+ \frac12\sum_{n,\overline{\N}_{3,n}}\int_0^{t_k} (n_1+n_2)\chi_{k_1}(n_1)\wh{v}(n_1)\chi_{k_2}(n_2)\wh{v}(n_2) \\
&\hspace{17em}\times \chi_k(n_3)n_3\wh{v}(n_3)\chi_k(n)n\wh{v}(n) \;dt\\
&- \sum_{n,\overline{\N}_{3,n}}\int_0^{t_k} (n_1+n_2)\chi_{k_1}(n_1)\wh{v}(n_1)\chi_{k_2}(n_2)\wh{v}(n_2)\\
&\hspace{17em}\times \psi_k(n_3)n_3\wh{v}(n_3)\chi_k(n)n\wh{v}(n) \;dt \Big|,
\end{aligned}\]
\[\begin{aligned}
A_2(k) =&~{}\sum_{0\le  k_1 ,k_2  \le k-10}\Big|\sum_{n,\overline{\N}_{3,n}}\int_0^{t_k} \chi_k(n)[(n_1+n_2)\chi_{k_1}(n_1)\wh{v}(n_1)\chi_{k_2}(n_2)\wh{v}(n_2)\\
&\hspace{21em}\times n_3^2\wh{v}(n_3)]\chi_k(n)\wh{v}(n)\; dt\\
&+ \sum_{n,\overline{\N}_{3,n}}\int_0^{t_k} (n_1+n_2)\chi_{k_1}(n_1)\wh{v}(n_1)\chi_{k_2}(n_2)\wh{v}(n_2)\\
&\hspace{19em}\times\chi_k(n_3)n_3\wh{v}(n_3)\chi_k(n)n\wh{v}(n)\;dt \Big|,
\end{aligned}\]
\[\begin{aligned}
A_3(k) &=\sum_{\substack{\max(k_1,k_2) \ge k -9\\k_3 \ge 0}}\Big|\sum_{n,\overline{\N}_{3,n}}\int_0^{t_k} \chi_{k_1}(n_1)\wh{v}(n_1)\chi_{k_2}(n_2)\wh{v}(n_2) \\
&\hspace{15em}\times\chi_{k_3}(n_3)n_3^2\wh{v}(n_3)\chi_k^2(n)n\wh{v}(n)\; dt\Big|,
\end{aligned}\]
\[\begin{aligned}
A_4(k) &=\sum_{\substack{\max(k_1,k_2) \ge k -9\\k_3 \ge 0}}\Big|\sum_{n,\overline{\N}_{3,n}}\int_0^{t_k} (n_1+n_2)\chi_{k_1}(n_1)\wh{v}(n_1)\chi_{k_2}(n_2)\wh{v}(n_2)\\
&\hspace{15em}\times\chi_{k_3}(n_3)n_3^2\wh{v}(n_3)\chi_k^2(n)\wh{v}(n)\; dt\Big|,
\end{aligned}\]
\[\begin{aligned}
A_5(k) &=\sum_{k_1,k_2,k_3 \ge 0}\Big|\sum_{n,\overline{\N}_{3,n}}\int_0^{t_k} \chi_{k_1}(n_1)n_1\wh{v}(n_1)\chi_{k_2}(n_2)n_2\wh{v}(n_2) \\
&\hspace{9em}\times\chi_{k_3}(n_3)n_3\wh{v}(n_3)\chi_k^2(n)\wh{v}(n)\; dt\Big|,
\end{aligned}\]
\begin{equation}\label{eq:energy1-2.7}
\begin{aligned}
A_6(k) &=\sum_{\max(k_1,k_2) \ge k -9}\Big|\sum_{n,\overline{\N}_{3,n}}\int_0^{t_k} (n_1+n_2)\chi_{k_1}(n_1)\wh{v}(n_1)\chi_{k_2}(n_2)\wh{v}(n_2)\\
&\hspace{7em}\times(\chi_k(n_3)n_3\wh{v}(n_3) + \psi_k(n_3)n_3\wh{v}(n_3))\chi_k(n)n\wh{v}(n) \;dt\Big|,
\end{aligned}
\end{equation}
\[\begin{aligned}
A_7(k) &=\sum_{k_1,k_2 \ge 0}\Big|\sum_{n,\overline{\N}_{3,n}}\int_0^{t_k} \big\{20n_1n_2^3n_3 + 10n_1^2n_2^2n_3 + 30n_1n_2^2n_3n+ 10 n_2^3n_3n\big\} \\
\times& \chi_{k_1}(n_1)\wh{v}(n_1)\chi_{k_2}(n_2)\wh{v}(n_2)\big(\chi_k(n_3)\frac{1}{n_3}\wh{v}(n_3)+\psi_k(n_3)\frac{1}{n_3}\wh{v}(n_3)\big)\chi_k(n)\frac1n\wh{v}(n) \;dt\Big|.
\end{aligned}\]
and
\[\begin{aligned}
A_8(k) &=\sum_{k_1,k_2,k_3,k_4,k_5 \ge 0}\Big|\sum_{n,\overline{\N}_{5,n}}\int_0^{t_k} \chi_{k_1}(n_1)\wh{v}(n_1)\\
&\hspace{3em}\times\chi_{k_2}(n_2)\wh{v}(n_2)\chi_{k_3}(n_3)\wh{v}(n_3)\chi_{k_4}(n_4)\wh{v}(n_4)\chi_{k_5}(n_5)\wh{v}(n_5)\chi_k^2(n)n\wh{v}(n) \;dt\Big|.
\end{aligned}\]
By using Lemma \ref{lem:commutator1} and the Cauchy-Schwarz inequality, we have
\[\begin{aligned}
A_1(k) + A_2(k) &\lesssim \sum_{0\le  k_1 ,k_2  \le k-10} 2^{2k_2} \norm{P_{k_1}v}_{F_{k_1}(T)}\norm{P_{k_2}v}_{F_{k_2}(T)}\sum_{|k-k'|\le 3} \norm{P_{k'}v}_{F_{k'}(T)}^2\\
&\lesssim \norm{v}_{F^{0}(T)}\norm{v}_{F^{2+}(T)}\sum_{|k-k'|\le 5} \norm{P_{k'}v}_{F_{k'}(T)}^2.
\end{aligned}\]
For $A_3(k)$ and $A_4(k)$, we divide the summation over $\max(k_1,k_2) \ge k -9,k_3 \ge 0$ into
\[\sum_{\substack{k_1,k_3 \le k- 10 \\ |k_2 - k| \le 5}}+\sum_{\substack{k_1 \le k- 10 \\ k_2, k_3 \ge k - 9}}+\sum_{\substack{k_3 \le k- 10 \\ k_1, k_2 \ge k - 9}}+\sum_{k_1,k_2,k_3 \ge k - 9},\]
assuming without loss of generality $k_1 \le k_2$. We restrict $A_3(k)$ and $A_4(k)$ to the first summation to obtain by \eqref{eq:energy1-1.3} and \eqref{eq:energy1-1.4} that
\[\begin{aligned}
\sum_{k_1 \le k_3-10} 2^{2k_3} &\norm{P_{k_1}v}_{F_{k_1}(T)}\norm{P_{k_3}v}_{F_{k_3}(T)}\sum_{|k-k'|\le 5} \norm{P_{k'}v}_{F_{k'}(T)}^2\\
&\lesssim \norm{v}_{F^0(T)}\norm{v}_{F^{2+}(T)}\sum_{|k-k'|\le 5} \norm{P_{k'}v}_{F_{k'}(T)}^2.
\end{aligned}\]

For the restriction to the second summation, by using \eqref{eq:energy1-1.2} and \eqref{eq:energy1-1.4}, we have 
\[\begin{aligned}
&\sum_{k_1 \le k- 10} 2^{k_1/2} \norm{P_{k_1}v}_{F_{k_1}(T)}\sum_{|k-k'|\le 5}2^{2k} \norm{P_{k'}v}_{F_{k'}(T)}^3\\
&\hspace{5em} + \sum_{k_1 \le k - 10}\norm{P_{k_1}v}_{F_{k_1}(T)}\norm{P_{k}v}_{F_k(T)}\sum_{\substack{k_2 \ge k + 9 \\ |k_2-k_3| \le 5}}2^{2k_3}\norm{P_{k_2}v}_{F_{k_2}(T)}\norm{P_{k_3}v}_{F_{k_3}(T)}\\
&\lesssim \norm{v}_{F^{\frac12+}(T)}\norm{v}_{F^{2}(T)}\sum_{|k-k'|\le 5} \norm{P_{k'}v}_{F_{k'}(T)}^2 \\
&\hspace{5em}+ \norm{v}_{F^0(T)}\norm{v}_{F^{2+}(T)}\norm{v}_{F^s}2^{-sk-\varepsilon k}\norm{P_{k}v}_{F_k(T)},
\end{aligned}\]
for $s \ge 0 $ and $0 < \varepsilon \ll 1$. 

For the third summation, we can get better or same bounds compared with the second summation due to two derivatives in the low frequency mode. 

For the last restriction, by using \eqref{eq:energy1-1.1}, \eqref{eq:energy1-1.2}, \eqref{eq:energy1-1.3} and \eqref{eq:energy1-1.4}, we have
\[\begin{aligned}
&\sum_{|k-k'| \le 5}2^{7k/2}\norm{P_{k'}v}_{F_{k'}(T)}^4 + \sum_{\substack{k_3 \ge k+9\\|k_3-k'| \le 5}}2^{2k_3}2^{k/2}\norm{P_{k'}v}_{F_{k'}(T)}^3\norm{P_{k}v}_{F_k(T)} \\
&\hspace{3em}+ \sum_{|k-k'|\le 5}2^{k/2}\norm{P_{k'}v}_{F_{k'}(T)}^2\sum_{\substack{k_3 \ge k+9 \\ |k_2-k_3|\le 5}}2^{2k_3}\norm{P_{k_2}v}_{F_{k_2}(T)}\norm{P_{k_3}v}_{F_{k_3}(T)}\\
&\hspace{3em}+ \norm{P_kv}_{F_k(T)}\sum_{k_1 \ge k+9}\norm{P_{k_1}v}_{F_{k_1}(T)}\sum_{\substack{k_3 \ge k_1+9 \\ |k_2-k_3|\le 5}}2^{2k_3}\norm{P_{k_2}v}_{F_{k_2}(T)}\norm{P_{k_3}v}_{F_{k_3}(T)}\\
&\lesssim \norm{v}_{F^{\frac74}(T)}^2\sum_{|k-k'| \le 5}\norm{P_{k'}v}_{F_{k'}(T)}^2 + \norm{v}_{F^s(T)}\norm{v}_{F^{\frac54 +}(T)}^22^{-sk-\varepsilon k}\norm{P_kv}_{F_k(T)}\\
&\hspace{3em}+ \norm{v}_{F^{\frac54 +}(T)}^2\sum_{|k-k'|\le 5}\norm{P_{k'}v}_{F_{k'}(T)}^2 + \norm{v}_{F^{s}(T)}\norm{v}_{F^{1+}(T)}^22^{-sk-\varepsilon k}\norm{P_{k}v}_{F_k(T)},
\end{aligned}\]
for $s \ge 0$ and $0<\varepsilon \ll 1$. Hence, we obtain
\[\begin{aligned}
A_3(k) + A_4(k) &\lesssim \norm{v}_{F^{\frac12}(T)}\norm{v}_{F^{2+}(T)}\\
&\qquad \times \left(\sum_{|k-k'|\le 5}\norm{P_{k'}v}_{F_{k'}(T)}^2  +\norm{v}_{F^{s}(T)}2^{-sk-\varepsilon k}\norm{P_{k}v}_{F_k(T)}\right).
\end{aligned}\]

For $A_5(k)$, we may assume that $k_1 \le k_2 \le k_3$ by the symmetry. We estimate $A_5(k)$ for $k = \max(k_1,k_2,k_3,k)$ or not, separately. When $k = \max(k_1,k_2,k_3,k)$, we divide the summation over $k_1,k_2,k_3$ into 
\[\sum_{|k-k_1| \le 5} + \sum_{\substack{|k-k_2| \le 5\\k_1 \le k-10}} +\sum_{\substack{|k-k_3| \le 5\\k_2 \le k-10 \\ |k_1 - k_2| \le 5}} +\sum_{\substack{|k-k_3| \le 5\\k_2 \le k-10 \\ k_1 \le k_2 -10}}.\]

We use \eqref{eq:energy1-1.1}, \eqref{eq:energy1-1.2}, \eqref{eq:energy1-1.3} and \eqref{eq:energy1-1.4} to estimate $A_5(k)$ under above summations, respectively, to obtain
\[\begin{aligned}
A_5(k) \lesssim&~{} 2^{7k/2}\sum_{|k'-k| \le 5}\norm{P_{k'}v}_{F_{k'}(T)}^4 + \sum_{\substack{|k-k'| \le 5\\k_1 \le k-10}}2^{3k_1/2}2^{k}\norm{P_{k_1}v}_{F_{k_1}(T)}\norm{P_{k'}v}_{F_{k'}(T)}^3\\
&+\sum_{\substack{|k-k'| \le 5\\k_2 \le k-10 \\ |k_1 - k_2| \le 5}}2^{5k_2/2}\norm{P_{k_1}v}_{F_{k_1}(T)}^2\norm{P_{k'}v}_{F_{k'}(T)}^2\\
&+\sum_{\substack{|k-k'| \le 5\\k_2 \le k-10 \\ k_1 \le k_2 -10}}2^{k_1}2^{k_2}\norm{P_{k_1}v}_{F_{k_1}(T)}\norm{P_{k_2}v}_{F_{k_2}(T)}\norm{P_{k'}v}_{F_{k'}(T)}^2\\
\lesssim&~{} \norm{v}_{F^{\frac74}(T)}^2\sum_{|k-k'|\le 5}\norm{P_{k'}v}_{F_{k'}(T)}^2.
\end{aligned}\] 

When $k \le k_3 -10$, by the support property \eqref{eq:support property}, we know $|k_2-k_3| \le 5$. Then, also we divide the summation over $k_1,k_2,k_3$ into
\[\sum_{\substack{|k_1-k_3| \le 5\\k \le k_1-10}} +\sum_{\substack{|k_2-k_3| \le 5\\k_1 \le k_2-10 \\ |k_1 - k| \le 5}} +\sum_{\substack{|k_2-k_3| \le 5\\k_1 \le k_2-10 \\ k \le k_1 -10}} +\sum_{\substack{|k_2-k_3| \le 5\\k \le k_2-10 \\ k_1 \le k- 10}}.\]

Similarly we use \eqref{eq:energy1-1.2}, \eqref{eq:energy1-1.3} and \eqref{eq:energy1-1.4} to estimate $A_5(k)$ under above summations to obtain
\[\begin{aligned}
A_5(k) \lesssim&~{} \sum_{\substack{|k_3-k'| \le 5\\k_3 \ge k+9}}2^{2k_3}2^{k/2}\norm{P_{k'}v}_{F_{k'}(T)}^3\norm{P_{k}v}_{F_k(T)}\\
&+\sum_{k_3 \ge k+9}2^{k_3}\norm{P_{k_3}v}_{F_{k_3}(T)}^2\sum_{|k-k'|\le 5}2^{3k/2}\norm{P_{k'}v}_{F_{k'}(T)}^2\\
&+\sum_{k_3 \ge k_1+9}2^{k_3}\norm{P_{k_3}v}_{F_{k_3}(T)}^2\sum_{k_1 \ge k+9}2^{k_1}\norm{P_{k_1}v}_{F_{k_1}(T)}2^{k/2}\norm{P_kv}_{F_k(T)}\\
&+\sum_{k_3 \ge k+9}2^{k_3}\norm{P_{k_3}v}_{F_{k_3}(T)}^2\sum_{k_1 \le k-10}2^{3k_1/2}\norm{P_{k_1}v}_{F_{k_1}(T)}\norm{P_kv}_{F_k(T)}\\
\lesssim&~{} \left(\norm{v}_{F^{\frac54+}(T)}^2+\norm{v}_{F^{1+}(T)}\norm{v}_{F^{\frac32}(T)} \right)\norm{v}_{F^s(T)}2^{-sk-\varepsilon k}\norm{P_kv}_{F_k(T)}\\
&+ \norm{v}_{F^{\frac54}(T)}^2\sum_{|k-k'|\le 5}\norm{P_{k'}v}_{F_{k'}(T)}^2,
\end{aligned}\] 
for $s \ge 0$ and $0 < \varepsilon \ll 1$. Hence, by gathering all bounds, we conclude that
\[\begin{aligned}
A_5(k) &\lesssim \norm{v}_{F^{1+}(T)}\norm{v}_{F^{\frac32}(T)}\norm{v}_{F^s(T)}2^{-sk-\varepsilon k}\norm{P_kv}_{F_k(T)}\\
&\hspace{9em} + \norm{v}_{F^{\frac74}(T)}^2\sum_{|k-k'|\le 5}\norm{P_{k'}v}_{F_{k'}(T)}^2.
\end{aligned}\]

For $A_6(k)$, we may assume $k_1 \le k_2$ without loss of generality. By the support property \eqref{eq:support property}, the summation over $\max(k_1,k_2) \ge k-9$ is divided into
\[\sum_{\substack{k_2 \ge k + 10\\|k_1 - k_2|\le 5}} + \sum_{\substack{|k_2 - k| \le 5\\|k_1-k_2| \le 5}} + \sum_{\substack{|k_2 - k| \le 5\\k_1 \le k - 10}}.\]
Hence, we use \eqref{eq:energy1-1.3}, \eqref{eq:energy1-1.1} and \eqref{eq:energy1-1.2} to estimate $A_6(k)$ restricted to above summations, respectively, to have 
\[\begin{aligned}
A_6(k) &\lesssim \sum_{k_2 \ge k+10}\norm{P_{k_2}v}_{F_{k_2}(T)}^22^{5k/2}\norm{P_kv}_{F_k(T)}^2+ 2^{7k/2}\sum_{|k-k'|\le 5}\norm{P_{k'}v}_{F_{k'}(T)}^4 \nonumber \\ 
&+ \sum_{k_1 \le k-10}2^{k_1/2}\norm{P_{k_1}v}_{F_{k_1}(T)}2^{2k}\sum_{|k-k'|\le5}\norm{P_{k'}v}_{F_{k'}(T)}^3 \nonumber \\
&\lesssim \left(\norm{v}_{F^{\frac74}(T)}^2 + \norm{v}_{F^{\frac12+}(T)}\norm{v}_{F^{2}(T)}\right)\sum_{|k-k'|\le 5}\norm{P_{k'}v}_{F_{k'}(T)}^2. 
\end{aligned}\]

For $A_7(k)$, since there are less derivatives at the $2^k$-frequency mode in $A_7(k)$ than in the other $A_i$, we can also obtain better or same bounds as
\[A_7(k) \lesssim \left(\norm{v}_{F^{\frac74}(T)}^2 + \norm{v}_{F^{\frac12+}(T)}\norm{v}_{F^{2}(T)}\right)\sum_{|k-k'|\le 5}\norm{P_{k'}v}_{F_{k'}(T)}^2.\]

For $A_8(k)$, we may assume $k_1 \le k_2 \le k_3 \le k_4 \le k_5$ due to the symmetry among frequencies $k_1,k_2,k_3,k_4,$ and $k_5$. We further assume $|k-k_5| \le 5$, otherwise the derivative loss at $2^k$ frequency mode is weaker. We use only the Cauchy-Schwarz inequality to estimate the quintic term except for the special term (see \eqref{eq:energy1-2.14} below). Indeed, by similar argument as in the proof of Lemma \ref{lem:energy1-1}, we have
\begin{equation}\label{eq:energy1-2.11}
\begin{aligned}
&\left|\int_{\T \times [0,T]}v_1v_2v_3v_4v_5v_6 \; dxdt \right| \\
&\lesssim 2^{2k_6}\sum_{j_i \ge 2k_6} \left|\sum_{\overline{n} \in \Gamma_6(\Z)}\int_{\overline{\tau}\in\Gamma_6(\R)}\prod_{i=1}^{6}\ft[\gamma(2^{2k_6}t-m)v_i](\tau_i,n_i) \right|\\
&\lesssim 2^{2k_6}\prod_{l=1}^{4}2^{k_l/2}\sum_{j_i \ge 2k_6}2^{-(j_{max}+j_{sub})/2}\prod_{i=1}^{6}2^{j_i/2}\norm{\eta_{j_i}(\tau_i-\mu(n_i))\ft[\gamma(2^{2k_6}t-m)v_i]}_{L_{\tau_i}^2\ell_{n_i}^2}\\
&\lesssim 2^{(k_1+k_2+k_3+k_4)/2}\prod_{i=1}^{6}\norm{v_i}_{F_{k_i}}(T),
\end{aligned}
\end{equation} 
where $v_i = P_{k_i}v \in F_{k_i}(T)$, $i=1,2,3,4,5,6$ and assuming that $k_1 \le k_2 \le k_3 \le k_4 \le k_5 \le k_6$. 

Moreover, we also have
\begin{equation}\label{eq:energy1-2.15}
\left|\int_{\T \times [0,T]}v_1v_2v_3v_4v_5v_6v_7v_8 \; dxdt \right|\lesssim 2^{(k_1+k_2+k_3+k_4+k_5+k_6)/2}\prod_{i=1}^{8}\norm{v_i}_{F_{k_i}(T)},
\end{equation}
where $v_i = P_{k_i}v \in F_{k_i}(T)$, $i=1,2,3,4,5,6,7,8$ and assuming that $k_1 \le k_2 \le k_3 \le k_4 \le k_5 \le k_6 \le k_7 \le k_8$. We will use \eqref{eq:energy1-2.15} for the septic term later.

If $|k_4 - k_5| \le 5$, by using \eqref{eq:energy1-2.11}, we have
\[\begin{aligned}
\sum_{\substack{|k-k_5| \le 5 \\ |k-k_4| \le 5 \\0 \le k_1 \le k_2 \le k_3 \le k_4}} 2^{(k_1+k_2+k_3)/2}2^{\frac32k_4}&\prod_{j=1}^{5}\norm{P_{k_j}v}_{F_{1,k_j}(T)}\norm{P_{k}v}_{F_k(T)}\\
&\lesssim \norm{v}_{F^{\frac12}(T)}^3\norm{v}_{F^{\frac32+}(T)}\sum_{|k-k'|\le5}\norm{P_{k'}v}_{F_{k'}(T)}^2.
\end{aligned}\]

Otherwise, we need to observe the frequency relation carefully. In other words, under the frequency relation with $k_4 \le k_5 - 10$ condition, the one of following cases should happen (see Section 8 in \cite{Bourgain1993}):
\[|n_4| \ll |n|^{4/5},\]
\[|n_4| \gtrsim |n|^{4/5} \mbox{ and } |n_3| \sim |n_4|\]
and
\[|n_4| \gtrsim |n|^{4/5} \mbox{ and } |n_3| \ll |n_4|.\]

For the first case, since 
\[|\mu(n_1) + \mu(n_2) + \mu(n_3) + \mu(n_4) + \mu(n_5) + \mu(n)| \gtrsim |n|^4,\]
we use $2^{-j_{max}/2} \lesssim 2^{-2k}$ instead of $2^{-j_{max}/2} \lesssim 2^{-k}$ in \eqref{eq:energy1-2.11} to obtain
\[\norm{v}_{F^{\frac12+}(T)}^4\sum_{|k-k'|\le 5}\norm{P_{k'}v}_{F_{k'}(T)}^2.\]

For the second case, since $|n| \lesssim |n_3|^{\frac58}|n_4|^{\frac58}$, we use \eqref{eq:energy1-2.11} so that
\[\norm{v}_{F^{\frac12+}(T)}^2\norm{v}_{F^{\frac98}(T)}^2\sum_{|k-k'|\le 5}\norm{P_{k'}v}_{F_{k'}(T)}^2.\] 

For the last case, since $n_1+n_2+n_3+ n_4 + n_5 +n=0$, we have $|n_5 + n| \sim |n|^{4/5}$, which implies
\[|\mu(n_1) + \mu(n_2) + \mu(n_3) + \mu(n_4) + \mu(n_5) + \mu(n)| \gtrsim |n|^{4+\frac45}.\]
Similarly as the first case, we obtain
\[A_8(k) \lesssim\norm{v}_{F^{\frac12}(T)}^4\sum_{|k-k'|\le 5}\norm{P_{k'}v}_{F_{k'}(T)}^2.\]

Together with all bounds of $A_i(k)$, we obtain
\begin{equation}\label{eq:cubic bound1}
\sum_{k \ge 1}2^{2sk} \sup_{t_k \in [0,T]}\left|\int_0^{t_k} E_{1} + E_{2,1} + E_{3,1} \; dt \right| \lesssim \left(\norm{v}_{F^{2+}(T)}^2 + \norm{v}_{F^{\frac32+}(T)}^4\right)\norm{v}_{F^s(T)}^2.
\end{equation}

Next, we estimate
\[\left|\int_0^{t_k} E_{2,2} + E_{3,2} \; dt \right|.\]
Since $E_{2,2}$ and $E_{3,2}$ are weaker than $E_{1},E_{2,1}$ and $E_{3,1}$ in the sense of the number of derivatives, we obtain better bounds as
\begin{equation}\label{eq:energy1-2.9}
\left|\int_0^{t_k} E_{2,2} + E_{3,2} \; dt \right| \lesssim \norm{v}_{F^{\frac34}(T)}^2\sum_{|k-k'|\le5}\norm{P_{k'}v}_{F_{k'}(T)}^2,
\end{equation}
which implies
\begin{equation}\label{eq:cubic bound2}
\sum_{k \ge 1}2^{2sk} \sup_{t_k \in [0,T]}\left|\int_0^{t_k} E_{2,2} + E_{3,2} \; dt \right| \lesssim \norm{v}_{F^{\frac34}(T)}^2\norm{v}_{F^s(T)}^2.
\end{equation}

For
\[\left|\int_0^{t_k} E_{2,4} + E_{3,4} \; dt \right|,\]
by the symmetries of $n_1,n_2$ and $n_3,n$, respectively, it suffices to estimate
\begin{equation}\label{eq:E4.1}
\sum_{k_1,k_2 \ge 0}\left|\int_0^{t_k}\sum_{n,\overline{\N}_{3,n}}\chi_{k_1}(n_1)\wh{v}(n_1)\chi_{k_2}(n_2)n_2^3|\wh{v}(n_2)|^2\wh{v}(n_2)\psi_k(n_3)\frac{1}{n_3}\wh{v}(n_3)\chi_k(n)\frac1n\wh{v}(n)\; dt \right|
\end{equation}
and
\begin{equation}\label{eq:E4.2}
\sum_{k_1,k_2 \ge 0}\left|\sum_{n,\overline{\N}_{3,n}}\chi_{k_1}(n_1)\wh{v}(n_1)\chi_{k_2}(n_2)\wh{v}(n_2)\psi_k(n_3)\frac{1}{n_3}\wh{v}(n_3)\chi_k(n)n^2|\wh{v}(n)|^2\wh{v}(n)\; dt \right|.
\end{equation}

For \eqref{eq:E4.1}, we may assume that $k_1 \le k_2$ due to the three derivatives taken at $P_{k_2}v$. Then, by using Lemma \ref{lem:energy1-1}, we obtain that
\begin{equation}\label{eq:E4-1}
\begin{aligned}
\eqref{eq:E4.1} \lesssim&~{} \norm{v}_{F^{\frac38}(T)}^2\sum_{|k-k'|\le 5}2^{3k/4}\norm{P_{k'}v}_{F_{k'}(T)}^4\\
&+\norm{v}_{F^{\frac12}(T)}^2\sum_{\substack{k_1 \ge k+10 \\ |k_1-k_2|\le 5}}2^{k_2}\norm{P_{k_1}v}_{F_{k_1}(T)}^2 \sum_{|k-k'|\le 5}2^{-3k/2}\norm{P_{k'}v}_{F_{k'}(T)}^2\\
&+\norm{v}_{F^{\frac16}(T)}^2\sum_{k_1 \le k - 10 }2^{k_1/2}\norm{P_{k_1}v}_{F_{k_1}(T)} \sum_{|k-k'|\le 5}\norm{P_{k'}v}_{F_{k'}(T)}^3\\
&+\norm{v}_{F^{\frac18}(T)}^2\sum_{\substack{k_2 \le k-10 \\ |k_1-k_2|\le 5}}2^{13k_2/4}\norm{P_{k_1}v}_{F_{k_1}(T)}^2 \sum_{|k-k'|\le 5}2^{-3k}\norm{P_{k'}v}_{F_{k'}(T)}^2\\
&+\norm{v}_{F^{0}(T)}^2\sum_{\substack{k_2 \le k-10 \\ k_1 \le k_2 -10}}2^{3k_2}\norm{P_{k_1}v}_{F_{k_1}(T)}\norm{P_{k_2}v}_{F_{k_2}(T)} \sum_{|k-k'|\le 5}2^{-3k}\norm{P_{k'}v}_{F_{k'}(T)}^2\\
\lesssim&~{} \norm{v}_{F^{\frac12}(T)}^4\sum_{|k-k'|\le 5}\norm{P_{k'}v}_{F_{k'}(T)}^2.
\end{aligned}
\end{equation}
For \eqref{eq:E4.2}, we may also assume that $k_1 \le k_2$ by the symmetry of $n_1,n_2$ variables, and then, similarly, we obtain that
\begin{equation}\label{eq:E4-2}
\begin{aligned}
\eqref{eq:E4.2} \lesssim&~{} \norm{v}_{F^{\frac38}(T)}^2\sum_{|k-k'|\le 5}2^{3k/4}\norm{P_{k'}v}_{F_{k'}(T)}^4\\
&+\norm{v}_{F^{\frac18}(T)}^2\sum_{\substack{k_1 \ge k+10 \\ |k_1-k_2|\le 5}}2^{-k_2}\norm{P_{k_1}v}_{F_{k_1}(T)}^2 \sum_{|k-k'|\le 5}2^{5k/4}\norm{P_{k'}v}_{F_{k'}(T)}^2\\
&+\norm{v}_{F^{\frac18}(T)}^2\sum_{k_1 \le k - 10 }2^{k_1/2}\norm{P_{k_1}v}_{F_{k_1}(T)} \sum_{|k-k'|\le 5}2^{-k/4}\norm{P_{k'}v}_{F_{k'}(T)}^3\\
&+\norm{v}_{F^{\frac18}(T)}^2\sum_{\substack{k_2 \le k-10 \\ |k_1-k_2|\le 5}}2^{k_2/2}\norm{P_{k_1}v}_{F_{k_1}(T)}^2 \sum_{|k-k'|\le 5}2^{-k/4}\norm{P_{k'}v}_{F_{k'}(T)}^2\\
&+\norm{v}_{F^{0}(T)}^2\sum_{\substack{k_2 \le k-10 \\ k_1 \le k_2 -10}}\norm{P_{k_1}v}_{F_{k_1}(T)}\norm{P_{k_2}v}_{F_{k_2}(T)} \sum_{|k-k'|\le 5}\norm{P_{k'}v}_{F_{k'}(T)}^2\\
\lesssim&~{} \norm{v}_{F^{\frac38}(T)}^4\sum_{|k-k'|\le 5}\norm{P_{k'}v}_{F_{k'}(T)}^2.
\end{aligned}
\end{equation}
Hence, we conclude that
\begin{equation}\label{eq:quintic bound}
\sum_{k \ge 1}2^{2sk} \sup_{t_k \in [0,T]}\left|\int_0^{t_k} E_{1,2,4} + E_{1,3,4} \; dt \right| \lesssim \norm{v}_{F^{\frac12}(T)}^4\norm{v}_{F^s(T)}^2.
\end{equation}

Lastly, we estimate quintic and septic terms given by
\begin{equation}\label{eq:energy1-2.9}
\left|\int_0^{t_k} E_{2,3} + E_{3,3} \; dt \right|.
\end{equation}

\begin{remark}\label{rem:resonant1}
It is necessary to carefully check the quintic resonance in $E_{2,3}$ and $E_{3,3}$. In fact, the worst terms are of the form of 
\[\begin{aligned}
&\mbox{Re}\Big[\alpha \sum_{n,\overline{\N}_{3,n}}\wh{v}(n_1)\wh{v}(n_2)\psi_k(n_3)\frac{1}{n_3}\left\{10in_3\sum_{\N_{3,n_3}}\wh{v}(n_{3,1})\wh{v}(n_{3,2})n_{3,3}^2\wh{v}(n_{3,3})\right\}\chi_k(n)\frac1n\wh{v}(n)\Big]\\
=&\mbox{Re}\Big[10\alpha i \sum_{n,\overline{\N}_{3,n},\N_{3,n_3}}\wh{v}(n_1)\wh{v}(n_2)\psi_k(n_3)\wh{v}(n_{3,1})\wh{v}(n_{3,2})n_{3,3}^2\wh{v}(n_{3,3})\chi_k(n)\frac1n\wh{v}(n)\Big],
\end{aligned}\]
and\[\mbox{Re}\Big[10\beta i \sum_{n,\overline{\N}_{3,n},\N_{3,n_3}}\wh{v}(n_1)\wh{v}(n_2)\chi_k(n_3)\wh{v}(n_{3,1})\wh{v}(n_{3,2})n_{3,3}^2\wh{v}(n_{3,3})\chi_k(n)\frac1n\wh{v}(n)\Big],\]
where $\N_{3,n_3}$ is the non-resonant set of $n_{3,1}, n_{3,2}$ and $n_{3,3}$ variables defined similarly as the set $\N_{3,n}$. In particular, if $n_{3,3} = -n$ (exact quintic resonant case), one derivative at $2^k$ frequency mode cannot be eliminated by the standard way. Thanks to the properties of $\psi_k$ and $\chi_k$ (real-valued and even functions) and the symmetry on the $n_1+n_2+n_{3,1}+n_{3,2}=0$, we know that
\begin{equation}\label{eq:quintic resonant}
\begin{aligned}
&\sum_{\overline{n} \in \Gamma_4(\Z)}\wh{v}(n_1)\wh{v}(n_2)\psi_k(n-n_{3,1}-n_{3,2})\wh{v}(n_{3,1})\wh{v}(n_{3,2})\chi_k(n)n|\wh{v}(n)|^2\\
=&\overline{\sum_{\overline{n} \in \Gamma_4(\Z)}\wh{v}(-n_1)\wh{v}(-n_2)\psi_k(n-n_{3,1}-n_{3,2})\wh{v}(-n_{3,1})\wh{v}(-n_{3,2})\chi_k(n)n|\wh{v}(n)|^2}\\
=&\overline{\sum_{\overline{n} \in \Gamma_4(\Z)}\wh{v}(n_1)\wh{v}(n_2)\psi_k(n+n_{3,1}+n_{3,2})\wh{v}(n_{3,1})\wh{v}(n_{3,2})\chi_k(n)n|\wh{v}(n)|^2}\\
=&\overline{\sum_{\overline{n} \in \Gamma_4(\Z)}\wh{v}(n_1)\wh{v}(n_2)\psi_k(n-n_1-n_2)\wh{v}(n_{3,1})\wh{v}(n_{3,2})\chi_k(n)n|\wh{v}(n)|^2}\\
=&\overline{\sum_{\overline{n} \in \Gamma_4(\Z)}\wh{v}(n_1)\wh{v}(n_2)\psi_k(n-n_{3,1}-n_{3,2})\wh{v}(n_{3,1})\wh{v}(n_{3,2})\chi_k(n)n|\wh{v}(n)|^2},
\end{aligned}
\end{equation}
where $\overline{n} = (n_1,n_2,n_{3,1},n_{3,2})$. This observation reveals 
\[10\alpha i\sum_{n, \overline{n} \in \Gamma_4(\Z)}\wh{v}(n_1)\wh{v}(n_2)\psi_k(n-n_{3,1}-n_{3,2})\wh{v}(n_{3,1})\wh{v}(n_{3,2})\chi_k(n)n|\wh{v}(n)|^2\]
is a purely imaginary number and hence the exact quintic resonant interaction component vanishes. Similarly, the quintic resonant term in $E_{3}$ also vanishes.
\end{remark}

We first consider the quintic terms in \eqref{eq:energy1-2.9}. For
\[\sum_{n,\overline{\N}_{3,n}}\Big(\wh{N}(v)(n_1)\wh{v}(n_2)+\wh{v}(n_1)\wh{N}(v)(n_2)\Big)\psi_k(n_3)\frac{1}{n_3}\wh{v}(n_3)\chi_k(n)\frac1n\wh{v}(n),\]
if the frequency support of $n$ ($\sim 2^k$) is the widest among other frequency supports, it suffices to control the following one:
\begin{equation}\label{eq:energy1-2.12}
\sum_{0 \le k_1 \le k_2 \le k_3 \le k_4 \le k}2^{k_4}\left|\sum_{\overline{n} \in \Gamma_6(\Z)}\int_0^{t_k}\prod_{i=1}^{4}\chi_{k_i}(n_i)\wh{v}(n_i)\chi_{k}(n_5)\wh{v}(n_5)\chi_{k}(n)\wh{v}(n)\right|.
\end{equation}

From \eqref{eq:energy1-2.11}, we have
\[\eqref{eq:energy1-2.12} \lesssim \norm{v}_{F^{\frac12}(T)}^3\norm{v}_{F^{\frac32+}(T)}\sum_{|k-k'|\le 5}\norm{P_{k'}v}_{F_{k'}(T)}^2.\]

Otherwise, it suffices to control
\begin{equation}\label{eq:energy1-2.13}
\sum_{\substack{0 \le k_1 \le k_2 \le k_3 \le k_4\\ |k_3-k_4| \le 5 \\ k \le k_3 - 10}}2^{3k_4}2^{-2k}\left|\sum_{\overline{n} \in \Gamma_6(\Z)}\int_0^{t_k}\prod_{i=1}^{4}\chi_{k_i}(n_i)\wh{v}(n_i)\chi_{k}(n_5)\wh{v}(n_5)\chi_{k}(n)\wh{v}(n)\right|,
\end{equation}
but we have similarly as before that
\[\eqref{eq:energy1-2.13} \lesssim \norm{v}_{F^{\frac12+}(T)}^2\norm{v}_{F^{\frac32}(T)}^2\sum_{|k-k'|\le 5}\norm{P_{k'}v}_{F_{k'}(T)}^2.\]

For the rest of quintic terms in \eqref{eq:energy1-2.9}, it is enough to consider\footnote{It is not necessary to distinguish $\psi_k$ and $\chi_k$.}
\[\sum_{n,\overline{\N}_{3,n}}\wh{v}(n_1)\wh{v}(n_2)\chi_k(n_3)\frac{1}{n_3}\wh{N}(v)(n_3)\chi_k(n)\frac1n\wh{v}(n).\]
Since $1/n_3$-factor removes one total derivative in $\wh{N}(v)$, the following term is the worst case:
\begin{equation}\label{eq:energy1-2.14}
\begin{aligned}
&\sum_{\substack{0 \le k_1\le k_2 \le k_{3,1} \le k_{3,2} \le k_{3,3} }} \Big|\sum_{n,\overline{\N}_{3,n}\N_{3,n_3}}\int_0^{t_k}\chi_{k_1}(n_1)\wh{v}(n_1)\chi_{k_2}(n_2)\wh{v}(n_2)\\
&\hspace{1em}\times\chi_{k_{3,1}}(n_{3,1})\wh{v}(n_{3,1})\chi_{k_{3,2}}(n_{3,2})\wh{v}(n_{3,2})\chi_{k_{3,3}}(n_{3,3})n_{3,3}^2\wh{v}(n_{3,3})\chi_k(n)\frac1n\wh{v}(n)\Big|.
\end{aligned}
\end{equation} 

We first focus on the case when $|k-k_{3,3}| \le 5$. If $|k_{3,2}-k|\le 5$, since one derivative in the frequency $n_{3,3}$ can be moved to $n_{3,2}$ frequency, we, similarly as before, have
\[\eqref{eq:energy1-2.14} \lesssim \norm{v}_{F^{\frac12+}(T)}^3\norm{v}_{F^{\frac32}(T)}\sum_{|k-k'|\le 5}\norm{P_{k'}v}_{F_{k'}(T)}^2.\]

Otherwise, we use the same argument as in the estimation of $A_8(k)$ and then the one of following cases should happen (except for the case mentioned in Remark \ref{rem:resonant1}):
\[|n_{3,2}| \ll |n|^{4/5},\]
\[|n_{3,2}| \gtrsim |n|^{4/5} \mbox{ and } |n_{3,1}| \sim |n_{3,2}|\]
and
\[|n_{3,2}| \gtrsim |n|^{4/5} \mbox{ and } |n_{3,1}| \ll |n_{3,2}|.\]

For the first case, since 
\[|\mu(n_1) + \mu(n_2) + \mu(n_{3,1}) + \mu(n_{3,2}) + \mu(n_{3,3}) + \mu(n)| \gtrsim |n|^4,\]
we use $2^{-j_{max}/2} \lesssim 2^{-2k}$ instead of $2^{-j_{max}/2} \lesssim 2^{-k}$ in \eqref{eq:energy1-2.11} to obtain
\[\eqref{eq:energy1-2.14} \lesssim \norm{v}_{F^{\frac12+}(T)}^4\sum_{|k-k'|\le 5}\norm{P_{k'}v}_{F_{k'}(T)}^2.\]

For the second case, since $|n_{3,3}| \lesssim |n_{3,1}|^{\frac58}|n_{3,2}|^{\frac58}$, we use \eqref{eq:energy1-2.11} so that
\[\eqref{eq:energy1-2.14} \lesssim \norm{v}_{F^{\frac12+}(T)}^2\norm{v}_{F^{\frac98}(T)}^2\sum_{|k-k'|\le 5}\norm{P_{k'}v}_{F_{k'}(T)}^2.\] 

For the last case, since $n_1+n_2+n_{3,1}+n_{3,2}+n_{3,3}+n=0$, we have $|n_{3,3} + n| \sim |n|^{4/5}$, which implies
\[|\mu(n_1) + \mu(n_2) + \mu(n_{3,1}) + \mu(n_{3,2}) + \mu(n_{3,3}) + \mu(n)| \gtrsim |n|^{4+\frac45}.\]
Similarly as the first case, we obtain
\[\eqref{eq:energy1-2.14} \lesssim \norm{v}_{F^{\frac12}(T)}^4\sum_{|k-k'|\le 5}\norm{P_{k'}v}_{F_{k'}(T)}^2.\]

Now, we focus on the case when $k \le k_{3,3} - 10$. From the support property \eqref{eq:support property}, we know $|k_{3,2}-k_{3,3}| \le 5$, and hence
\[\begin{aligned}\eqref{eq:energy1-2.14} \lesssim&~{} \norm{v}_{F^{\frac12+}(T)}^3\norm{v}_{F^2(T)}\norm{v}_{F^s(T)}2^{-sk-k/2}\norm{P_{k}v}_{F_k(T)} \\
&+ \norm{v}_{F^{\frac12+}(T)}^2\norm{v}_{F^{1}(T)}^2\sum_{|k-k'|\le 5}\norm{P_{k'}v}_{F_{k'}(T)}^2,
\end{aligned}\]
for $s \ge 0$.

Next, we consider the septic term in \eqref{eq:energy1-2.9}. For the septic term in
\[\sum_{n,\overline{\N}_{3,n}}\Big(\wh{N}(v)(n_1)\wh{v}(n_2)+\wh{v}(n_1)\wh{N}(v)(n_2)\Big)\psi_k(n_3)\frac{1}{n_3}\wh{v}(n_3)\chi_k(n)\frac1n\wh{v}(n),\]
since the quintic term in $\wh{N}(v)$ also has one total derivative, by the symmetry of frequencies, it is enough to control
\begin{equation}\label{eq:energy1-2.16}
\begin{aligned}
\sum_{\substack{0 \le k_1\le k_2 \le k_3 \le k_4 \le k_5 \le k_6 }} 2^{k_6}2^{-2k}\Big|\sum_{\overline{n}\in\Gamma_8(\Z)}\int_0^{t_k}\prod_{i=1}^{6}\chi_{k_i}(n_i)\wh{v}(n_i)\psi_k(n_7)\wh{v}(n_7)\chi_k(n)\wh{v}(n)\Big|.
\end{aligned}
\end{equation} 

We apply \eqref{eq:energy1-2.15} to \eqref{eq:energy1-2.16} to obtain
\[\eqref{eq:energy1-2.16} \lesssim \norm{v}_{F^{\frac12+}(T)}^6\sum_{|k-k'|\le 5}\norm{P_{k'}v}_{F_{k'}(T)}^2.\]

Moreover, for the septic term in
\[\sum_{n,\overline{\N}_{3,n}}\wh{v}(n_1)\wh{v}(n_2)\psi_k(n_3)\frac{1}{n_3}\wh{N}_1(v)(n_3)\chi_k(n)\frac1n\wh{v}(n),\]
since the total derivative of the quintic term in $\wh{N}(v)$ is canceled out by $1/n_3$ factor and there is no difference between $\psi_k$ and $\chi_k$ in the septic estimation, it suffices to control
\begin{equation}\label{eq:energy1-2.17}
\begin{aligned}
\sum_{\substack{0 \le k_1\le k_2 \le k_3 \le k_4 \le k_5 \le k_6 \le 7 }} 2^{-k}\Big|\sum_{\overline{n}\in\Gamma_8(\Z)}\int_0^{t_k}\prod_{i=1}^{7}\chi_{k_i}(n_i)\wh{v}(n_i)\chi_k(n)\wh{v}(n)\Big|.
\end{aligned}
\end{equation} 

By using \eqref{eq:energy1-2.15}, we obtain
\[\begin{aligned}
\eqref{eq:energy1-2.17} \lesssim& \norm{v}_{F^{\frac12+}(T)}^5\norm{v}_{F^0(T)}\norm{v}_{F^s(T)}2^{-sk-k/2}\norm{P_{k}v}_{F_k(T)} \\
&+ \norm{v}_{F^{\frac12+}(T)}^4\norm{v}_{F^{0+}(T)}^2\sum_{|k-k'|\le 5}\norm{P_{k'}v}_{F_{k'}(T)}^2.
\end{aligned}\]

Together with all bounds of quintic and septic terms, we conclude that
\begin{equation}\label{eq:quintic,septic bound}
\sum_{k \ge 1}2^{2sk} \sup_{t_k \in [0,T]}\left|\int_0^{t_k} E_{2,3} + E_{3,3} \; dt \right| \lesssim (\norm{v}_{F^{\frac12+}(T)}^2\norm{v}_{F^{\frac32+}(T)}^2 + \norm{v}_{F^{\frac12+}(T)}^6)\norm{v}_{F^s(T)}^2,
\end{equation}
and hence, we complete the proof of Proposition \ref{prop:energy1-2} by recalling the definition of the modified energy \eqref{eq:new energy1-2} and gathering \eqref{eq:cubic bound1}, \eqref{eq:cubic bound2}, \eqref{eq:quintic bound} and \eqref{eq:quintic,septic bound}. 
\end{proof}

As a corollary to Lemma \ref{lem:comparable energy1-1} and Proposition \ref{prop:energy1-2}, we obtain an \emph{a priori} bound of $\norm{v}_{E^s(T)}$ for a smooth solution $v$ to the equation \eqref{eq:5mkdv4}.
\begin{corollary}\label{cor:energy1-2}
Let $s > 2$ and $T \in (0,1]$. Then, there exists $0 < \delta \ll 1$ such that
\begin{equation}\label{eq:energy1-2.1.1}
\norm{v}_{E^s(T)}^2 \lesssim (1+ \norm{v_0}_{H^s}^2)\norm{v_0}_{H^s}^2 + (1 + \norm{v}_{F^{\frac12+}(T)}^2 + \norm{v}_{F^{\frac12+}(T)}^4)\norm{v}_{F^{2+}(T)}^2\norm{v}_{F^s(T)}^2,
\end{equation} 
for the solution $v \in C([-T,T];H^{\infty}(\T))$ to \eqref{eq:5mkdv4} with $\norm{v}_{L_T^{\infty}H_x^{\frac12+}} \le \delta$.
\end{corollary}

In the following, we consider the energy estimate for the difference of two solutions $v_1$ and $v_2$ to the equation in \eqref{eq:5mkdv4}. Let $w = v_1 - v_2$. Then $w$ satisfies
\begin{equation}\label{eq:5mkdv8}
\pt\wh{w}(n) - i\mu(n)\wh{w}(n)=\wh{N}_1(v_1,v_2,w)+\wh{N}_2(v_1,v_2,w)+\wh{N}_3(v_1,v_2,w)+\wh{N}_4(v_1,v_2,w),
\end{equation}
with $w(0,x) = w_0(x) = v_{1,0}(x) - v_{2,0}(x)$ and where
\[\wh{N}_1(v_1,v_2,w) = -20in^3(|\wh{v}_1(n)|^2\wh{w}(n) + \wh{v}_1(n)\wh{v}_2(n)\wh{w}(-n) + |\wh{v}_2(n)|^2\wh{w}(n),\]

\[\begin{aligned}
\wh{N}_2(v_1,v_2,w) =&~{} 6i n\sum_{\N_{5,n}} \wh{w}(n_1)\wh{v}_1(n_2)\wh{v}_1(n_3)\wh{v}_1(n_4)\wh{v}_1(n_5) \\
&+ 6i n\sum_{\N_{5,n}}\wh{v}_2(n_1)\wh{w}(n_2)\wh{v}_1(n_3)\wh{v}_1(n_4)\wh{v}_1(n_5)\\
&+ 6i n\sum_{\N_{5,n}}\wh{v}_2(n_1)\wh{v}_2(n_2)\wh{w}(n_3)\wh{v}_1(n_4)\wh{v}_1(n_5) \\
&+ 6i n\sum_{\N_{5,n}}\wh{v}_2(n_1)\wh{v}_2(n_2)\wh{v}_2(n_3)\wh{w}(n_4)\wh{v}_1(n_5)\\
&+ 6i n\sum_{\N_{5,n}}\wh{v}_2(n_1)\wh{v}_2(n_2)\wh{v}_2(n_3)\wh{v}_2(n_4)\wh{w}(n_5),
\end{aligned}\]

\[\begin{aligned}
\wh{N}_3(v_1,v_2,w) =&~{} 10in \sum_{\N_{3,n}} \wh{w}(n_1)\wh{v}_1(n_2)n_3^2\wh{v}_1(n_3) + 10in \sum_{\N_{3,n}} \wh{v}_2(n_1)\wh{w}(n_2)n_3^2\wh{v}_1(n_3)\\
&+ 10in \sum_{\N_{3,n}} \wh{v}_2(n_1)\wh{v}_2(n_2)n_3^2\wh{w}(n_3)
\end{aligned}\]
and
\[\begin{aligned}
\wh{N}_4(v_1,v_2,w) =&~{} 5in \sum_{\N_{3,n}} (n_1+n_2)\wh{w}(n_1)\wh{v}_1(n_2)n_3\wh{v}_1(n_3) \\
&+ 5in \sum_{\N_{3,n}} (n_1+n_2)\wh{v}_2(n_1)\wh{w}(n_2)n_3\wh{v}_1(n_3)\\
&+ 5in \sum_{\N_{3,n}} (n_1+n_2)\wh{v}_2(n_1)\wh{v}_2(n_2)n_3\wh{w}(n_3).
\end{aligned}\]
We denote $\wh{N}_1(v_1,v_2,w)+\wh{N}_2(v_1,v_2,w)+\wh{N}_3(v_1,v_2,w)+\wh{N}_4(v_1,v_2,w)$ by $\wh{N}(v_1,v_2,w)$ only in the proof of Proposition \ref{prop:energy1-3} below. Similarly as before, for $k \ge 1$, we define the localized modified energy for the difference of two solutions like
\begin{equation}\label{eq:new energy1-3}
\begin{aligned}
\wt{E}_{k}(w)(t) =&~{} \norm{P_kw(t)}_{L_x^2}^2 \\
&+ \mbox{Re}\left[\wt{\alpha} \sum_{n,\overline{\N}_{3,n}}\wh{v}_2(n_1)\wh{v}_2(n_2)\psi_k(n_3)\frac{1}{n_3}\wh{w}(n_3)\chi_k(n)\frac1n\wh{w}(n)\right]\\
&+ \mbox{Re}\left[\wt{\beta} \sum_{n,\overline{\N}_{3,n}}\wh{v}_2(n_1)\wh{v}_2(n_2)\chi_k(n_3)\frac{1}{n_3}\wh{w}(n_3)\chi_k(n)\frac1n\wh{w}(n)\right]
\end{aligned}
\end{equation}
and
\[\wt{E}_{T}^s(w) = \norm{P_0w(0)}_{L_x^2}^2 + \sum_{k \ge 1}2^{2sk} \sup_{t_k \in [-T,T]} \wt{E}_{1,k}(w)(t_k),\]
where $\wt{\alpha}_1$ and $\wt{\beta}_1$ are real and will be chosen later.

\begin{remark}\label{rem:modified energy}
The modified energy \eqref{eq:new energy1-3} is not suitable to control $\wh{N}_4(v_1,v_2,w)$ due to the following term
\[5in \sum_{\N_{3,n}} n_1\wh{w}(n_1)\wh{v}_1(n_2)n_3\wh{v}_1(n_3)+ 5in \sum_{\N_{3,n}} n_2\wh{v}_2(n_1)\wh{w}(n_2)n_3\wh{v}_1(n_3),\]
for the case when $|n_2|, |n_3| \ll |n_1| \sim |n|$ or $|n_1|, |n_3| \ll |n_2| \sim |n|$. Moreover, the cancellation of the quintic resonant case as in Remark \ref{rem:resonant1} cannot be expected. Hence, it is necessary to define the modified energy for the difference of two solutions as
\begin{equation}\label{eq:new energy1-5}
\begin{aligned}
\wt{E}_{k}(w)(t) &=\norm{P_kw(t)}_{L_x^2}^2 \\
&+\sum_{1 \le i \le j \le 2} \mbox{Re}\left[\wt{\alpha}_{i,j} \sum_{n,\overline{\N}_{3,n}}\wh{v}_i(n_1)\wh{v}_j(n_2)\psi_k(n_3)\frac{1}{n_3}\wh{w}(n_3)\chi_k(n)\frac1n\wh{w}(n)\right]\\
&+\sum_{1 \le i \le j \le 2} \mbox{Re}\left[\wt{\beta}_{i,j} \sum_{n,\overline{\N}_{3,n}}\wh{v}_i(n_1)\wh{v}_j(n_2)\chi_k(n_3)\frac{1}{n_3}\wh{w}(n_3)\chi_k(n)\frac1n\wh{w}(n)\right],
\end{aligned}
\end{equation}
where $\wt{\alpha}_{i,j}$ and $\wt{\beta}_{i,j}$, $1 \le i \le j \le 2$, are real and will be chosen later. 

On the other hands, in view of \eqref{eq:5mkdv4},
\[10in\sum_{\N_{3,n}}\wh{v}(n_1)\wh{v}(n_2)n_3^2\wh{v}(n_3)\]
and
\[5in\sum_{\N_{3,n}}(n_1+n_2)n_3\wh{v}(n_1)\wh{v}(n_2)\wh{v}(n_3)\]
can be rewritten as
\[\frac{10}{3}in\sum_{\N_{3,n}}\set{n_1^2 + n_2^2+ n_3^2}\wh{v}(n_1)\wh{v}(n_2)\wh{v}(n_3)\]
and
\[\frac53in\sum_{\N_{3,n}}\set{(n_1+n_2)n_3+(n_1+n_3)n_2+(n_2+n_3)n_1}\wh{v}(n_1)\wh{v}(n_2)\wh{v}(n_3),\]
respectively. Then, by the simple calculation, $\wh{N}_3(v_1,v_2,w)$ and $\wh{N}_4(v_1,v_2,w)$ can be replaced by
\begin{equation}\label{eq:energy-nonlinear3-1}
\begin{aligned}
&\frac{10}{3}in \sum_{\N_{3,n}} \left(\wh{v}_1(n_1)\wh{v}_1(n_2)+\wh{v}_1(n_1)\wh{v}_2(n_2)+\wh{v}_2(n_1)\wh{v}_2(n_2)\right)n_3^2\wh{w}(n_3)\\ 
&+ \frac{20}{3}in \sum_{\N_{3,n}} n_2^2\left(\wh{v}_1(n_1)\wh{v}_1(n_2)+\wh{v}_1(n_1)\wh{v}_2(n_2)+\wh{v}_2(n_1)\wh{v}_1(n_2)+\wh{v}_2(n_1)\wh{v}_2(n_2)\right)\wh{w}(n_3)
\end{aligned}
\end{equation}
and
\begin{equation}\label{eq:energy-nonlinear4-1}
\begin{aligned}
\sum_{1 \le j \le k \le 2}&\frac{10}{3}in \sum_{\N_{3,n}} (n_1+n_2)\wh{v}_j(n_1)\wh{v}_k(n_2)n_3\wh{w}(n_3) \\
&+\sum_{1 \le j \le k \le 2} 5in\sum_{\N_{3,n}} n_1\wh{v}_j(n_1)n_2\wh{v}_k(n_2)\wh{w}(n_3),
\end{aligned}
\end{equation}
respectively. Then, the use of the modified energy \eqref{eq:new energy1-3} enables to control the first term in \eqref{eq:energy-nonlinear3-1} and the first three terms in \eqref{eq:energy-nonlinear4-1}. The rest of  \eqref{eq:energy-nonlinear3-1} and \eqref{eq:energy-nonlinear4-1} does not make troubles, since less derivatives are taken at $\wh{w}(n_3)$ mode. Moreover, thanks to the formula of the first term in \eqref{eq:energy-nonlinear3-1} in addition to \eqref{eq:new energy1-5}, we can get rid of quintic resonant interaction components similarly as in Remark \ref{rem:resonant1} (See Remark \ref{rem:resonant2} below for more details).

Thus, except in the case mentioned above, we will only consider \eqref{eq:new energy1-3} as the modified energy, and
\begin{equation}\label{eq:energy-nonlinear3-2}
\wh{N}_3(v_2,w) = \frac{10}{3}in \sum_{\N_{3,n}} \wh{v}_2(n_1)\wh{v}_2(n_2)n_3^2\wh{w}(n_3)+\frac{20}{3}in \sum_{\N_{3,n}} n_2^2\wh{v}_2(n_1)\wh{v}_2(n_2)\wh{w}(n_3)
\end{equation}
and
\begin{equation}\label{eq:energy-nonlinear4-2}
\wh{N}_4(v_2,w) = \frac{10}{3}in \sum_{\N_{3,n}} (n_1+n_2)\wh{v}_2(n_1)\wh{v}_2(n_2)n_3\wh{w}(n_3)+ 5in\sum_{\N_{3,n}} n_1\wh{v}_2(n_1)n_2\wh{v}_2(n_2)\wh{w}(n_3)
\end{equation}
as the cubic nonlinear term for the difference of two solutions for the sake of convenience of calculation.  For the same reason, we also use
\begin{equation}\label{eq:energy-nonlinear1-2}
\wh{N}_1(v_2,w) = -20in|\wh{v}_2(n)|^2\wh{w}(n).
\end{equation}
\end{remark}

Similarly as in Lemma \ref{lem:comparable energy1-1}, we can show that $\wt{E}_{T}^s(w)$ and $\norm{w}_{E^s(T)}$ are comparable.
\begin{lemma}\label{lem:comparable energy1-2}
Let $s > \frac12$. Then, there exists $0 < \delta \ll 1$ such that  
\[\frac12\norm{w}_{E^s(T)}^2 \le \wt{E}_{T}^s(w) \le \frac32\norm{w}_{E^s(T)}^2,\]
for all $w \in E^s(T) \cap C([-T,T];H^s(\T))$ as soon as $\norm{v_2}_{L_T^{\infty}H^s(\T)} \le \delta$.\footnote{In view of Remark \ref{rem:modified energy}, Lemma \ref{lem:comparable energy1-2} holds true for all $w \in E^s(T) \cap C([-T,T];H^s(\T))$ as soon as $\norm{v_1}_{L_T^{\infty}H^s(\T)} \le \delta$ and $\norm{v_2}_{L_T^{\infty}H^s(\T)} \le \delta$.}
\end{lemma}

\begin{proposition}\label{prop:energy1-3}
Let $s > 2$ and $T \in (0,1]$. Then, for solutions $w \in C([-T,T];H^{\infty}(\T))$ to \eqref{eq:5mkdv8} and $v_1, v_2 \in C([-T,T];H^{\infty}(\T))$ to \eqref{eq:5mkdv4}, we have
\begin{equation}\label{eq:energy1-3.1}
\begin{aligned}
\wt{E}_{T}^0(w) \lesssim&~{} (1+ \norm{v_{1,0}}_{H^{\frac12+}}^2+\norm{v_{1,0}}_{H^{\frac12+}}\norm{v_{2,0}}_{H^{\frac12+}}+\norm{v_{2,0}}_{H^{\frac12+}}^2)\norm{w_0}_{L_x^2}^2\\
&+\left(\norm{v_1}_{F^{2+}(T)}^2+\norm{v_1}_{F^{2+}(T)}\norm{v_2}_{F^{2+}(T)}+\norm{v_2}_{F^{2+}(T)}^2\right)\norm{w}_{F^0(T)}^2\\
&+\Big(\sum_{j=0}^{4}\norm{v_1}_{F^{\frac32+}(T)}^{4-j}\norm{v_2}_{F^{\frac32+}(T)}^j\Big)\norm{w}_{F^0(T)}^2\\
&+\Big(\sum_{j=0}^{6}\norm{v_1}_{F^{\frac12+}(T)}^{6-j}\norm{v_2}_{F^{\frac12+}(T)}^j\Big)\norm{w}_{F^0(T)}^2
%+\Big(\sum_{\substack{i,j,k,l,m,n=1,2\\i\le j\le k\le l \le m \le n}}\norm{v_i}_{F^{\frac12+}(T)}\norm{v_j}_{F^{\frac12+}(T)}\norm{v_k}_{F^{\frac12+}(T)}\norm{v_l}_{F^{\frac12+}(T)}\norm{v_m}_{F^{\frac12+}(T)}\norm{v_n}_{F^{\frac12+}(T)}\Big)\norm{w}_{F^0(T)}^2
\end{aligned} 
\end{equation}
and
\begin{equation}\label{eq:energy1-3.2}
\begin{aligned}
\wt{E}_{T}^s(w) \lesssim&~{} (1+ \norm{v_{1,0}}_{H^{\frac12+}}^2+\norm{v_{1,0}}_{H^{\frac12+}}\norm{v_{2,0}}_{H^{\frac12+}}+\norm{v_{2,0}}_{H^{\frac12+}}^2)\norm{w_0}_{H^s}^2\\
&+\left(\norm{v_1}_{F^s(T)}^2+\norm{v_1}_{F^s(T)}\norm{v_2}_{F^s(T)}+\norm{v_2}_{F^s(T)}^2\right)\norm{w}_{F^s(T)}^2\\
&+\left(\sum_{i,j=1,2}\norm{v_i}_{F^{\frac12}(T)}\norm{v_j}_{F^{2s}(T)}\right)\norm{w}_{F^0(T)}\norm{w}_{F^s(T)}\\
&+\Big(\sum_{j=0}^{4}\norm{v_1}_{F^{s}(T)}^{4-j}\norm{v_2}_{F^{s}(T)}^j\Big)\norm{w}_{F^s(T)}^2\\
&+\Big(\sum_{j=0}^{3}\norm{v_1}_{F^{s}(T)}^{3-j}\norm{v_2}_{F^{s}(T)}^j\Big)(\norm{v_1}_{F^{2s}(T)}+\norm{v_2}_{F^{2s}(T)})\norm{w}_{F^0(T)}\norm{w}_{F^s(T)}\\
&+\Big(\sum_{j=0}^{6}\norm{v_1}_{F^{s}(T)}^{6-j}\norm{v_2}_{F^{s}(T)}^j\Big)\norm{w}_{F^s(T)}^2.
\end{aligned}
\end{equation} 
\end{proposition}

\begin{proof}
We use similar argument as in the proof of Proposition \ref{prop:energy1-2}. For any $k \ge 1$ and $t \in [-T,T]$, we differentiate $\wt{E}_{k}(w)$ with respect to $t$ and deduce that 
\[\frac{d}{dt}\wt{E}_{k}(w) = \frac{d}{dt}\wt{I}(t) + \frac{d}{dt}\wt{II}(t) + \frac{d}{dt}\wt{III}(t),\]
where 
\[\begin{aligned}
\frac{d}{dt}\wt{I}(t) &= \frac{d}{dt}\norm{P_kw}_{L_x^2}^2\\
&= 20 i \sum_{n} \chi_k^2(n) n^3 \wh{v}_1(-n)\wh{v}_2(-n)\wh{w}(n)\wh{w}(n)\\
&+ 2\mbox{Re}\left[\sum_{n}\chi_k(n)\left(\overline{\wh{N}}_{1,2}(v_1,v_2,w)+\overline{\wh{N}}_{1,3}(v_1,v_2,w)+\overline{\wh{N}}_{1,4}(v_1,v_2,w)\right)\chi_k(n)\wt{w}(n)\right]\\
&=: \wt{E}_{1},
\end{aligned}\]
\[\frac{d}{dt}\wt{II}(t)= \wt{E}_{2,1} + \wt{E}_{2,2} + \wt{E}_{2,3} =: \wt{E}_{2},\]
for
\[\begin{aligned}
\wt{E}_{2,1} =&~{}\mbox{Re}\Big[\wt{\alpha}i \sum_{n,\overline{\N}_{3,n}}\big\{10n_1n_2^3(n_3+n) + 5n_1^2n_2^2(n_3+n) + 30n_1n_2^2n_3n \\
&+ 10 n_2^3n_3n - 5(n_1+n_2)n_3^2n^2\big\} \wh{v}_2(n_1)\wh{v}_2(n_2)\psi_k(n_3)\frac{1}{n_3}\wh{w}(n_3)\chi_k(n)\frac1n\wh{w}(n)\Big],
\end{aligned}\]

\[\begin{aligned}
\wt{E}_{2,2} &=\mbox{Re}\Big[c_1\wt{\alpha}i \sum_{n,\overline{\N}_{3,n}}\big\{3n_1n_2(n_3+n) + 6n_2n_3n\big\} \\
&\hspace{9em}\times \wh{v}_2(n_1)\wh{v}_2(n_2)\psi_k(n_3)\frac{1}{n_3}\wh{w}(n_3)\chi_k(n)\frac1n\wh{w}(n)\Big]
\end{aligned}\]
and
\[\begin{aligned}
\wt{E}_{2,3} &=\mbox{Re}\Big[\wt{\alpha} \sum_{n,\overline{\N}_{3,n}}\wh{N}(v_2)(n_1)\wh{v}_2(n_2)\psi_k(n_3)\frac{1}{n_3}\wh{w}(n_3)\chi_k(n)\frac1n\wh{w}(n)\\
&\hspace{5em}+\wh{v}_2(n_1)\wh{N}(v_2)(n_2)\psi_k(n_3)\frac{1}{n_3}\wh{w}(n_3)\chi_k(n)\frac1n\wh{w}(n)\\
&\hspace{5em}+\wh{v}_2(n_1)\wh{v}_2(n_2)\psi_k(n_3)\frac{1}{n_3}\wh{N}(v_1,v_2,w)(n_3)\chi_k(n)\frac1n\wh{w}(n)\\
&\hspace{5em}+\wh{v}_2(n_1)\wh{v}_2(n_2)\psi_k(n_3)\frac{1}{n_3}\wh{v}_2(n_3)\chi_k(n)\frac1n\wh{N}(v_1,v_2,w)(n)\Big],
\end{aligned}\]
and
\[\frac{d}{dt}\wt{III}(t) = \wt{E}_{3,1} + \wt{E}_{3,2} + \wt{E}_{3,3} =: \wt{E}_{3},\] 
for
\[\begin{aligned}
\wt{E}_{3,1} =&~{} \mbox{Re}\Big[\wt{\beta}i \sum_{n,\overline{\N}_{3,n}}\big\{10n_1n_2^3(n_3+n) + 5n_1^2n_2^2(n_3+n) + 30n_1n_2^2n_3n \\
&+ 10 n_2^3n_3n - 5(n_1+n_2)n_3^2n^2\big\} \wh{v}_2(n_1)\wh{v}_2(n_2)\chi_k(n_3)\frac{1}{n_3}\wh{w}(n_3)\chi_k(n)\frac1n\wh{w}(n)\Big],
\end{aligned}\]
\[\begin{aligned}
\wt{E}_{3,2} &=\mbox{Re}\Big[c_1\wt{\beta}i \sum_{n,\overline{\N}_{3,n}}\big\{3n_1n_2(n_3+n) + 6n_2n_3n\big\}\\
&\hspace{9em}\times \wh{v}_2(n_1)\wh{v}_2(n_2)\chi_k(n_3)\frac{1}{n_3}\wh{w}(n_3)\chi_k(n)\frac1n\wh{w}(n)\Big]
\end{aligned}\]
and
\[\begin{aligned}
\wt{E}_{3,3} &=\mbox{Re}\Big[\wt{\beta} \sum_{n,\overline{\N}_{3,n}}\wh{N}(v_2)(n_1)\wh{v}_2(n_2)\chi_k(n_3)\frac{1}{n_3}\wh{w}(n_3)\chi_k(n)\frac1n\wh{w}(n)\\
&\hspace{5em}+\wh{v}_2(n_1)\wh{N}(v_2)(n_2)\chi_k(n_3)\frac{1}{n_3}\wh{w}(n_3)\chi_k(n)\frac1n\wh{w}(n)\\
&\hspace{5em}+\wh{v}_2(n_1)\wh{v}_2(n_2)\chi_k(n_3)\frac{1}{n_3}\wh{N}(v_1,v_2,w)(n_3)\chi_k(n)\frac1n\wh{w}(n)\\
&\hspace{5em}+\wh{v}_2(n_1)\wh{v}_2(n_2)\chi_k(n_3)\frac{1}{n_3}\wh{v}_2(n_3)\chi_k(n)\frac1n\wh{N}(v_1,v_2,w)(n)\Big].
\end{aligned}\]
Similarly as in the proof of Proposition \ref{prop:energy1-2}, we need to control
\begin{equation}\label{eq:energy1-3.3}
\left|\int_0^{t_k} \wt{E}_{1} + \wt{E}_{2} + \wt{E}_{3} \; dt \right|.
\end{equation}

We first control the terms 
\[-\mbox{Re}\left[\frac{20}{3}i \sum_{n,\overline{\N}_{3,n}}\chi_k(n)n\wh{v}_2(n_1)\wh{v}_2(n_2)n_3^2\wh{w}(n_3)\chi_k(n)\wh{w}(n)\right]\]
and
\[-\mbox{Re}\left[\frac{20}{3}i \sum_{n,\overline{\N}_{3,n}}\chi_k(n)n(n_1+n_2)\wh{v}_2(n_1)\wh{v}_2(n_2)n_3\wh{w}(n_3)\chi_k(n)\wh{w}(n)\right]\]
in $\wt{E}_{1}$,
\[\mbox{Re}\left[- 5\wt{\alpha}i\sum_{n,\overline{\N}_{3,n}}(n_1+n_2)n_3^2n^2\wh{v}_2(n_1)\wh{v}_2(n_2)\psi_k(n_3)\frac{1}{n_3}\wh{w}(n_3)\chi_k(n)\frac1n\wh{w}(n)\right]\]
and
\[\mbox{Re}\left[- 5\wt{\beta}i\sum_{n,\overline{\N}_{3,n}}(n_1+n_2)n_3^2n^2\wh{v}_2(n_1)\wh{v}_2(n_2)\chi_k(n_3)\frac{1}{n_3}\wh{w}(n_3)\chi_k(n)\frac1n\wh{w}(n)\right]\]
in $\wt{E}_{2}$ and $\wt{E}_{3}$, respectively. In order to use Lemma \ref{lem:commutator1}, we choose $\wt{\alpha} = -\frac{4}{3}$ and $\wt{\beta} = -\frac{2}{3}$ (In fact, we need to choose $\wt{\alpha}_{i,j} = -\frac{4}{3}$ and $\wt{\beta}_{i,j} = -\frac{2}{3}$, $1 \le i \le j \le 2$). Then it suffices to control the following terms:
\begin{equation}\label{eq:energy1-3.4}
\begin{aligned}
&\sum_{k_1,k_2 \le k - 10}\Big|\sum_{n,\overline{\N}_{3,n}}\int_0^{t_k} \chi_k(n)n[\chi_{k_1}(n_1)\wh{v}_2(n_1)\chi_{k_2}(n_2)\wh{v}_2(n_2)n_3^2\wh{w}(n_3)]\chi_k(n)\wh{w}(n) \;dt\\
&+ \frac12\sum_{n,\overline{\N}_{3,n}}\int_0^{t_k} (n_1+n_2)\chi_{k_1}(n_1)\wh{v}_2(n_1)\chi_{k_2}(n_2)\wh{v}_2(n_2)\chi_k(n_3)n_3\wh{w}(n_3)\chi_k(n)n\wh{w}(n) \;dt\\
&- \sum_{n,\overline{\N}_{3,n}}\int_0^{t_k} (n_1+n_2)\chi_{k_1}(n_1)\wh{v}_2(n_1)\chi_{k_2}(n_2)\wh{v}_2(n_2)\psi_k(n_3)n_3\wh{w}(n_3)\chi_k(n)n\wh{w}(n) \;dt \Big|,
\end{aligned}
\end{equation}
\begin{equation}\label{eq:energy1-3.5}
\begin{aligned}
&\sum_{0\le k_1,k_2 \le k - 10}\Big|\sum_{n,\overline{\N}_{3,n}}\int_0^{t_k} \chi_k(n)\Big[(n_1+n_2)\chi_{k_1}(n_1)\wh{v}_2(n_1)\chi_{k_2}(n_2)\wh{v}_2(n_2)n_3\wh{w}(n_3)\Big]\\
&\hspace{27em}\times\chi_k(n)n\wh{w}(n)\; dt\\
&- \sum_{n,\overline{\N}_{3,n}}\int_0^{t_k} (n_1+n_2)\chi_{k_1}(n_1)\wh{v}_2(n_1)\chi_{k_2}(n_2)\wh{v}_2(n_2)\chi_k(n_3)n_3\wh{w}(n_3)\chi_k(n)n\wh{w}(n)\;dt \Big|,
\end{aligned}
\end{equation}
\begin{equation}\label{eq:energy1-3.6}
\begin{aligned}
\sum_{\substack{\max(k_1,k_2) \ge k-9\\k_3 \ge 0}}\Big|\sum_{n,\overline{\N}_{3,n}}\int_0^{t_k}&\chi_{k_1}(n_1)\wh{v}_2(n_1)\chi_{k_2}(n_2)\wh{v}_2(n_2)\\
&\times\chi_{k_3}(n_3)n_3^2\wh{w}(n_3)\chi_k^2(n)n\wh{w}(n)\;dt\Big|,
\end{aligned}
\end{equation}
\begin{equation}\label{eq:energy1-3.7}
\begin{aligned}
\sum_{\substack{\max(k_1,k_2) \ge k-9\\k_3 \ge 0}}\Big|\sum_{n,\overline{\N}_{3,n}}\int_0^{t_k}&(n_1+n_2)\chi_{k_1}(n_1)\wh{v}_2(n_1)\chi_{k_2}(n_2)\wh{v}_2(n_2)\\
&\times\chi_{k_3}(n_3)n_3\wh{w}(n_3)\chi_k^2(n)n\wh{w}(n)\;dt\Big|,
\end{aligned}
\end{equation}
and
\begin{equation}\label{eq:energy1-3.9}
\begin{aligned}
\sum_{\max(k_1,k_2) \ge k-9}&\Big|\sum_{n,\overline{\N}_{3,n}}\int_0^{t_k}(n_1+n_2)n_3^2n^2\chi_{k_1}(n_1)\wh{v}_2(n_1)\chi_{k_2}(n_2)\wh{v}_2(n_2)\\
&\hspace{5em} \times (\psi_k(n_3)+\chi_k(n_3))\frac{1}{n_3}\wh{w}(n_3)\chi_k(n)\frac1n\wh{w}(n)\;dt\Big|.
\end{aligned}
\end{equation}
We apply \eqref{eq:commutator1-1} and \eqref{eq:commutator1-3} to \eqref{eq:energy1-3.4} and \eqref{eq:energy1-3.5}, respectively, to have
\[\begin{aligned}
\eqref{eq:energy1-3.4} + \eqref{eq:energy1-3.5} &\lesssim \sum_{k_1,k_2 \le k-10} 2^{\max(2k_1,2k_2)} \norm{P_{k_1}v_2}_{F_{k_1}(T)}\norm{P_{k_2}v_2}_{F_{k_2}(T)}\sum_{|k-k'|\le 5} \norm{P_{k'}w}_{F_{k'}(T)}^2\\
&\lesssim \norm{v_2}_{F^0(T)}\norm{v_2}_{F^{2+}(T)}\sum_{|k-k'|\le 5} \norm{P_{k'}w}_{F_{k'}(T)}^2,
\end{aligned}\]
which implies
\begin{equation}\label{eq:energy1-3.10}
\begin{aligned}
\sum_{k \ge 1} 2^{2sk} \sup_{t_k \in [0,T]} \left(\eqref{eq:energy1-3.4} + \eqref{eq:energy1-3.5} \right) \lesssim \norm{v_2}_{F^0(T)}\norm{v_2}_{F^{2+}(T)}\norm{w}_{F^s(T)}^2,
\end{aligned}
\end{equation}
whenever $s \ge 0$.

For \eqref{eq:energy1-3.6} and \eqref{eq:energy1-3.7}, similarly as the estimates of $A_3(k)$ and $A_4(k)$ in the proof of Proposition \ref{prop:energy1-2}, we divide the summation over $k_1,k_2,k_3$ into
\begin{equation}\label{eq:summation}
\sum_{\substack{k_1,k_3 \le k- 10 \\ |k_2 - k| \le 5}}+\sum_{\substack{k_1 \le k- 10 \\ k_2, k_3 \ge k - 9}}+\sum_{\substack{k_3 \le k- 10 \\ k_1, k_2 \ge k - 9}}+\sum_{k_1,k_2,k_3 \ge k - 9},
\end{equation}
assuming without loss of generality $k_1 \le k_2$. We restrict \eqref{eq:energy1-3.6} to the first summation, by \eqref{eq:energy1-1.3} and \eqref{eq:energy1-1.4}, we have
\[\begin{aligned}
\sum_{k_1 \le k_3-10} 2^{2k_3} &\norm{P_{k_1}v_2}_{F_{k_1}(T)}\norm{P_{k_3}w}_{F_{k_3}(T)}\sum_{|k-k_2|\le 5} \norm{P_{k_2}v_2}_{F_{k_2}(T)}\norm{P_{k}w}_{F_k(T)}\\
&\lesssim \norm{v_2}_{F^0(T)}\norm{w}_{F^{2+}(T)}\sum_{|k-k_2|\le 5}\norm{P_{k_2}v_2}_{F_{k_2}(T)}\norm{P_{k}w}_{F_k(T)}.
\end{aligned}\]
For the restriction to the second summation, by using \eqref{eq:energy1-1.2} and \eqref{eq:energy1-1.4}, we have 
\[\begin{aligned}
&\sum_{k_1 \le k- 10} 2^{k_1/2} \norm{P_{k_1}v_2}_{F_{k_1}(T)}\sum_{\substack{|k-k_2|\le 5 \\ |k-k'| \le 5}} 2^{2k}\norm{P_{k_2}v_2}_{F_{k_2}(T)}\norm{P_{k'}w}_{F_{k'}(T)}^2\\
&\hspace{1em} + \sum_{k_1 \le k - 10}\norm{P_{k_1}v_2}_{F_{k_1}(T)}2^k\norm{P_{k}w}_{F_k(T)}\sum_{\substack{k_2,k_3 \ge k + 9 \\ |k_2-k_3| \le 5}}2^{k_3}\norm{P_{k_2}v_2}_{F_{k_2}(T)}\norm{P_{k_3}w}_{F_{k_3}(T)}\\
&\lesssim \norm{v_2}_{F^{0}(T)}\norm{v_2}_{F^{\frac32+}(T)}\sum_{|k-k'|\le 5} \norm{P_{k'}w}_{F_{k'}(T)}^2\\
&\hspace{1em} + \norm{v_2}_{F^0(T)}\norm{v_2}_{F^{2+}(T)}\norm{w}_{F^s(T)}2^{-sk-\varepsilon k}\norm{P_{k}w}_{F_k(T)},
\end{aligned}\]
for $s \ge 0 $ and $0 < \varepsilon \ll 1$. Due to two derivatives in the low frequency mode, we can obtain better or same bounds from the third summation than the second summation. For the last restriction, by using \eqref{eq:energy1-1.1}, \eqref{eq:energy1-1.2}, \eqref{eq:energy1-1.3} and \eqref{eq:energy1-1.4}, we have
\[\begin{aligned}
&\sum_{|k-k'| \le 5}2^{7k/2}\norm{P_{k'}v_2}_{F_{k'}(T)}^2\norm{P_{k'}w}_{F_{k'}(T)}^2 \\
&\hspace{1em}+ \sum_{\substack{k_3 \ge k+9\\|k_3-k'| \le 5}}2^{k_3}\norm{P_{k'}v_2}_{F_{k'}(T)}^2\norm{P_{k'}w}_{F_{k'}(T)}2^{3k/2}\norm{P_{k}w}_{F_k(T)} \\
&\hspace{1em}+ \sum_{|k-k_1|\le 5}2^{3k/2}\norm{P_{k_1}v_2}_{F_{k_1}(T)}\norm{P_{k}w}_{F_k(T)}\sum_{\substack{k_3 \ge k+9 \\ |k_2-k_3|\le 5}}2^{k_3}\norm{P_{k_2}v_2}_{F_{k_2}(T)}\norm{P_{k_3}w}_{F_{k_3}(T)}\\
&\hspace{1em}+ \norm{P_kw}_{F_k(T)}\sum_{k_1 \ge k+9}\norm{P_{k_1}v_2}_{F_{k_1}(T)}\sum_{\substack{k_3 \ge k_1+9 \\ |k_2-k_3|\le 5}}2^{2k_3}\norm{P_{k_2}v_2}_{F_{k_2}(T)}\norm{P_{k_3}w}_{F_{k_3}(T)}\\
&\lesssim \norm{v_2}_{F^{\frac74}(T)}^2\sum_{|k-k'| \le 5}\norm{P_{k'}w}_{F_{k'}(T)}^2 + \norm{v_2}_{F^{\frac54}(T)}^2\norm{w}_{F^s(T)}2^{-sk-\varepsilon k}\norm{P_kw}_{F_k(T)}\\
&\hspace{1em}+ \norm{v_2}_{F^{2}(T)}\norm{w}_{F^s(T)}\sum_{|k-k_1|\le 5}2^{-sk}\norm{P_{k_1}v_2}_{F_{k_1}(T)}\norm{P_{k}w}_{F_k(T)}\\ 
&\hspace{1em}+ \norm{v_2}_{F^{0}(T)}\norm{v_2}_{F^{2+}(T)}\norm{w}_{F^s(T)}2^{-sk-\varepsilon k}\norm{P_{k}w}_{F_k(T)},
\end{aligned}\]
for $s \ge 0$ and $0<\varepsilon \ll 1$. 

For \eqref{eq:energy1-3.7}, by using \eqref{eq:energy1-1.3} and \eqref{eq:energy1-1.4}, \eqref{eq:energy1-3.7} restricted to the first summation in \eqref{eq:summation} is dominated by
\[\begin{aligned}
\sum_{k_1 \le k_3-10} 2^{k_3} &\norm{P_{k_1}v_2}_{F_{k_1}(T)}\norm{P_{k_3}w}_{F_{k_3}(T)}\sum_{|k-k_2|\le 5}2^{k_2} \norm{P_{k_2}v_2}_{F_{k_2}(T)}\norm{P_{k}w}_{F_k(T)}\\
&\lesssim \norm{v_2}_{F^0(T)}\norm{w}_{F^0(T)}\sum_{|k-k_2|\le 5}2^{2k_2+\varepsilon k_2}\norm{P_{k_2}v_2}_{F_{k_2}(T)}\norm{P_{k}w}_{F_k(T)}.
\end{aligned}\]
For the restriction to the other summations, we obtain the same result as the estimation of \eqref{eq:energy1-3.6}. Hence, we obtain
\begin{equation}\label{eq:energy1-3.12}
\begin{aligned}
\sum_{k \ge 1} 2^{2sk} &\sup_{t_k \in [0,T]} \Big(\eqref{eq:energy1-3.6} + \eqref{eq:energy1-3.7} \Big) \\
&\lesssim  \norm{v_2}_{F^{s}(T)}^2\norm{w}_{F^s(T)}^2 + \norm{v_2}_{F^{0}} \norm{v_2}_{F^{s+2}(T)}\norm{w}_{F^0(T)}\norm{w}_{F^s(T)},
\end{aligned}
\end{equation}
whenever $s > 2$ and 
\begin{equation}\label{eq:energy1-3.8}
\sum_{k \ge 1} \sup_{t_k \in [0,T]} \left(\eqref{eq:energy1-3.6} + \eqref{eq:energy1-3.7} \right) \lesssim  \norm{v_2}_{F^{0}(T)}\norm{v_2}_{F^{2+}(T)}\norm{w}_{F^0(T)}^2,
\end{equation}
at $L^2$-level.

For \eqref{eq:energy1-3.9}, similarly as \eqref{eq:energy1-2.7}, we obtain
\[\begin{aligned}
\eqref{eq:energy1-3.9} \lesssim&~{} \sum_{k_2 \ge k+10}\norm{P_{k_2}v_2}_{F_{k_2}(T)}^22^{5k/2}\norm{P_kw}_{F_k(T)}^2\\
&+ \sum_{|k-k'|\le 5}2^{7k/2}\norm{P_{k'}v_2}_{F_{k'}(T)}^2\norm{P_{k'}w}_{F_{k'}(T)}^2  \\ 
&+ \sum_{k_1 \le k-10}2^{k_1/2}\norm{P_{k_1}v_2}_{F_{k_1}(T)}\sum_{|k-k'|\le5}2^{2k}\norm{P_{k'}v_2}_{F_{k'}(T)}\norm{P_{k'}w}_{F_{k'}(T)}^2 \\
\lesssim&~{} \left(\norm{v_2}_{F^{\frac74}(T)}^2 + \norm{v_2}_{F^{\frac12+}(T)}\norm{v_2}_{F^{2}(T)}\right)\sum_{|k-k'|\le 5}\norm{P_{k'}w}_{F_{k'}(T)}^2, 
\end{aligned}\]
which implies
\begin{equation}\label{eq:energy1-3.13}
\sum_{k\ge 1}2^{2sk}\sup_{t_k \in [0,T]} \eqref{eq:energy1-3.9} \lesssim \left(\norm{v_2}_{F^{\frac74}(T)}^2 + \norm{v_2}_{F^{\frac12+}(T)}\norm{v_2}_{F^{2}(T)}\right)\norm{w}_{F^s(T)}^2,
\end{equation}
whenever $s \ge 0$.

Now, we focus on the rest terms in $\wt{E}_{1},\wt{E}_{2},\wt{E}_{3}$. First, we estimate the cubic terms in $\wt{E}_{1}$. Since 
\[\left|\sum_{n} \chi_k^2(n) n^3 \wh{v}_1(-n)\wh{v}_2(-n)\wh{w}(n)\wh{w}(n)\right| \lesssim \norm{v_1}_{H^{\frac32}}\norm{v_2}_{H^{\frac32}}\norm{P_kw}_{L^2}^2\]
and $F^s(T) \hookrightarrow C_TH^s$ \eqref{eq:small data1.1}, we can obtain for the first term in $\wt{E}_{1}$ that
\begin{equation}\label{eq:energy1-3.17}
\begin{aligned}
\sum_{k\ge 1}2^{2sk}\sup_{t_k \in [0,T]}\Big|\sum_{n}\int_0^{t_k} \chi_k^2(n) n^3 \wh{v}_1(-n)&\wh{v}_2(-n)\wh{w}(n)\wh{w}(n) \; dt\Big| \\
&\lesssim \norm{v_1}_{F^{\frac32}(T)}\norm{v_2}_{F^{\frac32}(T)}\norm{w}_{F^s(T)}^2
\end{aligned}
\end{equation}
for $s \ge 0$.

For the non-resonant interaction components in $\wt{E}_{1}$, it suffices from \eqref{eq:energy-nonlinear3-2} and \eqref{eq:energy-nonlinear4-2} to consider
\begin{equation}\label{eq:energy1-3.18}
\left|\int_0^{t_k}\sum_{n,\overline{\N}_{3,n}}\chi_{k_1}(n_1)\wh{v}_2(n_1)\chi_{k_2}(n_2)n_2^2\wh{v}_2(n_2)\chi_{k_3}(n_3)\wh{w}(n_3)\chi_k^2(n)n\wh{w}(n) \;dt \right|
\end{equation}
and
\begin{equation}\label{eq:energy1-3.19}
\left|\int_0^{t_k}\sum_{n,\overline{\N}_{3,n}}\chi_{k_1}(n_1)n_1\wh{v}_2(n_1)\chi_{k_2}(n_2)n_2\wh{v}_2(n_2)\chi_{k_3}(n_3)\wh{w}(n_3)\chi_k^2(n)n\wh{w}(n) \;dt \right|.
\end{equation}

For \eqref{eq:energy1-3.18}, since there are two derivatives on the $P_{k_2}v_2$ and one derivative on the $P_kw$, it is enough to consider the case when $k = \max(k_1,k_2,k_3,k)$ and $|k_2-k| \le 5$. Then, we use Lemma \ref{lem:energy1-1} to obtain that
\[\begin{aligned}
&\sum_{k_1,k_2,k_3 \ge 0} \left|\int_0^{t_k}\sum_{n,\overline{\N}_{3,n}}\chi_{k_1}(n_1)\wh{v}_2(n_1)\chi_{k_2}(n_2)n_2^2\wh{v}_2(n_2)\chi_{k_3}(n_3)\wh{w}(n_3)\chi_k^2(n)n\wh{w}(n) \;dt \right| \\
\lesssim&~{} \sum_{|k-k'|\le 5}2^{\frac72k}\norm{P_{k'}v_2}_{F_{k'}(T)}\norm{P_{k'}v_2}_{F_{k'}(T)}\norm{P_{k'}w}_{F_{k'}(T)}^2 \\
&+ \sum_{\substack{|k-k'|\le 5\\k_1 \le k-10}}2^{2k}2^{k_1/2}\norm{P_{k_1}v_2}_{F_{k_1}(T)}\norm{P_{k'}v_2}_{F_{k'}(T)}\norm{P_{k'}w}_{F_{k'}(T)}^2\\
&+ \sum_{\substack{|k-k'|\le 5\\k_3 \le k-10}}2^{2k}2^{k_3/2}\norm{P_{k_3}w}_{F_{k_3}(T)}\norm{P_{k'}v_2}_{F_{k'}(T)}\norm{P_{k'}v_2}_{F_{k'}(T)}\norm{P_{k'}w}_{F_{k'}(T)}\\
&+ \sum_{\substack{|k-k_2|\le 5\\k_3 \le k_2-10\\k_1 \le k_3-10}}2^{2k}\norm{P_{k_1}v_2}_{F_{k_1}(T)}\norm{P_{k_3}w}_{F_{k_3}(T)}\norm{P_{k_2}v_2}_{F_{k_2}(T)}\norm{P_{k}w}_{F_k(T)}\\
&+ \sum_{\substack{|k-k_2|\le 5\\k_1 \le k_2-10\\k_3 \le k_1-10}}2^{2k}\norm{P_{k_1}v_2}_{F_{k_1}(T)}\norm{P_{k_3}w}_{F_{k_3}(T)}\norm{P_{k_2}v_2}_{F_{k_2}(T)}\norm{P_{k}w}_{F_k(T)}\\
&+ \sum_{\substack{|k-k_2|\le 5\\k_1 \le k_2-10\\|k_1-k_3|\le 5}}2^{2k}2^{\frac12k_1}\norm{P_{k_1}v_2}_{F_{k_1}(T)}\norm{P_{k_3}w}_{F_{k_3}(T)}\norm{P_{k_2}v_2}_{F_{k_2}(T)}\norm{P_{k}w}_{F_k(T)}\\
\lesssim&~{} \left(\norm{v_2}_{F^{\frac74}(T)}^2+\norm{v_2}_{F^{\frac12+}(T)}\norm{v_2}_{F^{2}(T)}\right)\sum_{|k-k'|\le 5}\norm{P_{k'}w}_{F_{k'}(T)}^2 \\
&+ \left(\norm{v_2}_{F^{\frac54+}(T)}^2\norm{w}_{F^{s}(T)}+\norm{v_2}_{F^{\frac12}(T)}\norm{v_2}_{F^{s+2+}(T)}\norm{w}_{F^{0}(T)}\right)2^{-sk-\varepsilon k}\norm{P_kw}_{F_k(T)}.
\end{aligned}\]
This implies
\begin{equation}\label{eq:energy1-3.20}
\begin{aligned}
\sum_{k\ge 1}2^{2sk}\sup_{t_k \in [0,T]}\sum_{k_1,k_2,k_3 \ge 0}\eqref{eq:energy1-3.18} &\lesssim \left(\norm{v_2}_{F^{\frac74}(T)}^2+ \norm{v_2}_{F^{\frac12+}(T)}\norm{v_2}_{F^{2}(T)}\right)\norm{w}_{F^s(T)}^2\\
&+\norm{v_2}_{F^{\frac12}(T)}\norm{v_2}_{F^{2s}(T)}\norm{w}_{F^0}\norm{w}_{F^s(T)},
\end{aligned}
\end{equation}
for $s > 2$ and 
\begin{equation}\label{eq:energy1-3.21}
\begin{aligned}
\sum_{k\ge 1}\sup_{t_k \in [0,T]}\sum_{k_1,k_2,k_3 \ge 0}\eqref{eq:energy1-3.18} &\lesssim \left(\norm{v_2}_{F^{\frac74}(T)}^2+\norm{v_2}_{F^{\frac12+}(T)}\norm{v_2}_{F^{2}(T)}\right)\norm{w}_{F^0(T)}^2\\
&+\norm{v_2}_{F^{\frac12}(T)}\norm{v_2}_{F^{2+}(T)}\norm{w}_{F^0(T)}^2
\end{aligned}
\end{equation}
at $L^2$-level. 

For \eqref{eq:energy1-3.19}, since derivatives are distributed to several functions, this term is weaker than \eqref{eq:energy1-3.18} in some sense and hence we omit to estimate \eqref{eq:energy1-3.19}.

For the rest cubic terms in $\wt{E}_{2,1}, \wt{E}_{2,2}, \wt{E}_{3,1}$ and $\wt{E}_{3,2}$, it is enough to consider 
\begin{equation}\label{eq:energy1-3.22}
\sum_{k_1,k_2 \ge 0}\left|\int_0^{t_k}\sum_{n,\overline{\N}_{3,n}}\chi_{k_1}(n_1)n_1\wh{v}_2(n_1)\chi_{k_2}(n_2)n_2^3\wh{v}_2(n_2)\chi_{k}(n_3)\wh{w}(n_3)\chi_k(n)\frac1n\wh{w}(n) \;dt \right|
\end{equation}
and
\begin{equation}\label{eq:energy1-3.23}
\sum_{k_1,k_2 \ge 0}\left|\int_0^{t_k}\sum_{n,\overline{\N}_{3,n}}\chi_{k_1}(n_1)\wh{v}_2(n_1)\chi_{k_2}(n_2)n_2^3\wh{v}_2(n_2)\chi_{k}(n_3)\wh{w}(n_3)\chi_k(n)\wh{w}(n) \;dt \right|
\end{equation}
under the assumption that $k_1 \le k_2$.\footnote{In fact, since there are two more derivatives in $\wt{E}_{2,1}$ and $\wt{E}_{3,1}$ than $\wt{E}_{2,2}$ and $\wt{E}_{3,2}$, \eqref{eq:energy1-3.22} and \eqref{eq:energy1-3.23} dominate all terms in $\wt{E}_{2,2}$ and $\wt{E}_{3,2}$.} Moreover, we may consider \eqref{eq:energy1-3.22} and \eqref{eq:energy1-3.23} for $k \le k_1 - 10$ and $k_1 \le k -10$, respectively. If $k_2 \le k -10$, by using \eqref{eq:energy1-1.3} and \eqref{eq:energy1-1.4}, we obtain
\[\begin{aligned}
&\sum_{\substack{k_2 \le k - 10\\ k_1 \le k_2 - 10}}2^{3k_2}2^{-k}\norm{P_{k_1}v_2}_{F_{k_1}(T)}\norm{P_{k_2}v_2}_{F_{k_2}(T)}\norm{P_{k}w}_{F_k(T)}^2\\ 
&+ \sum_{\substack{k_2 \le k - 10\\ |k_1-k_2|\le 5}}2^{\frac72k_2}2^{-k}\norm{P_{k_1}v_2}_{F_{k_1}(T)}\norm{P_{k_2}v_2}_{F_{k_2}(T)}\norm{P_{k}w}_{F_k(T)}^2\\
\lesssim&~{} \norm{v_2}_{F^0(T)}\norm{v_2}_{F^{2+}(T)}\norm{P_kw}_{F_k(T)}^2 + \norm{v_2}_{F^{\frac54}(T)}^2\norm{P_kw}_{F_k(T)}^2.
\end{aligned}\]
If $k_1 \le k - 10$ and $|k-k_2| \le 5$, by using \eqref{eq:energy1-1.2}, we have
\[\begin{aligned}
\sum_{\substack{k_1 \le k - 10\\ |k_2 -k| \le 5}}2^{2k}2^{k_1/2}\norm{P_{k_1}v_2}_{F_{k_1}(T)}&\norm{P_{k_2}v_2}_{F_{k_2}(T)}\norm{P_{k}w}_{F_k(T)}^2 \\
&\lesssim \norm{v_2}_{F^{\frac12+}(T)}\norm{v_2}_{F^{2}(T)}\norm{P_kw}_{F_k(T)}^2.
\end{aligned}\]
Otherwise, we use \eqref{eq:energy1-1.1} and \eqref{eq:energy1-1.3} to obtain that
\[\begin{aligned}
&\sum_{|k_1 - k|\le 5}2^{\frac72k}\norm{P_{k_1}v_2}_{F_{k_1}(T)}\norm{P_{k_2}v_2}_{F_{k_2}(T)}\norm{P_{k}w}_{F_k(T)}^2\\ 
&+ \sum_{\substack{k \le k_1 - 10\\ |k_1-k_2|\le 5}}2^{-\frac12k}2^{3k_1}\norm{P_{k_1}v_2}_{F_{k_1}(T)}\norm{P_{k_2}v_2}_{F_{k_2}(T)}\norm{P_{k}w}_{F_k(T)}^2\\
\lesssim&~{} \norm{v_2}_{F^{\frac74}(T)}^2\norm{P_kw}_{F_k(T)}^2 + \norm{v_2}_{F^{\frac32}(T)}\norm{v_2}_{F^{\frac32}(T)}\norm{P_kw}_{F_k(T)}^2.
\end{aligned}\]
Above results imply that
\begin{equation}\label{eq:energy1-3.24}
\sum_{k\ge 1}2^{2sk}\sup_{t_k \in [0,T]} \left(\eqref{eq:energy1-3.22} + \eqref{eq:energy1-3.23} \right) \lesssim \left( \norm{v_2}_{F^{\frac74}(T)}^2+\norm{v_2}_{F^{\frac12+}(T)}\norm{v_2}_{F^{2+}(T)}\right)\norm{w}_{F^s(T)}^2,
\end{equation}
whenever $s \ge 0$.

Now, we concentrate on quintic and septic terms in \eqref{eq:energy1-3.3}. We first estimate quintic terms in $\wt{E}_{1}$. Since we can observe the symmetry of functions and one derivative is taken on $P_kw$, it suffices to consider
\begin{equation}\label{eq:energy1-3.25}
\left|\sum_{n,\overline{\N}_{5,n}}\prod_{j=1}^{4} \chi_{k_j}(n_j)\wh{v}_2(n_j)\chi_{k_5}(n_5)\wh{w}(n_5)\chi_k^2(n)n\wh{w}(n) \;dt\right|,
\end{equation}
under the assumption that $k_1 \le k_2 \le k_3 \le k_4$ and $k_4, k_5 \le k$. When $k_4 \le k_5 - 10$, similarly as in the estimation of $A_8(k)$, the one of the following cases should happen:
\[|n_4| \ll |n|^{4/5},\]
\[|n_4| \gtrsim |n|^{4/5} \mbox{ and } |n_3| \sim |n_4|\]
and
\[|n_4| \gtrsim |n|^{4/5} \mbox{ and } |n_3| \ll |n_4|.\]
We use $2^{-j_{max}/2} \lesssim 2^{-2k}$ in \eqref{eq:energy1-2.11}, $|n| \lesssim |n_3|^{5/8}|n_4|^{5/8}$ and $2^{-j_{max}/2} \lesssim 2^{-(2+\frac25)k}$ in \eqref{eq:energy1-2.11} for each case, respectively, to obtain
\[\sum_{\substack{|k-k_5|\le 5\\k_4 \le k_5 - 10\\0 \le k_1 \le k_2 \le k_3 \le k_4}} \eqref{eq:energy1-3.25} \lesssim \norm{v_2}_{F^{\frac12+}(T)}^2\norm{v_2}_{F^{\frac98}(T)}^2\sum_{|k-k'|\le 5}\norm{P_{k'}w}_{F_{k'}(T)}^2\]
by using \eqref{eq:energy1-2.11}. If $|k_5 -k_4| \le 5$ and $k_3 \le k_4 - 10$, one derivative on $P_kw$ can be moved to $P_{k_4}v_2$, and hence we get from \eqref{eq:energy1-2.11} that
\[\sum_{\substack{|k-k_5|\le 5\\|k_4-k_5| \le 5\\0 \le k_1 \le k_2 \le k_3 \le k_4 - 10}} \eqref{eq:energy1-3.25} \lesssim \norm{v_2}_{F^{\frac12}(T)}^3\norm{v_2}_{F^{\frac32+}(T)}\sum_{|k-k'|\le 5}\norm{P_{k'}w}_{F_{k'}(T)}^2.\]
Otherwise, similarly as the case when $k_4 \le k_5- 10$, we obtain
\[\begin{aligned}
\sum_{\substack{|k-k_5|\le 5\\k_4 \le k_5 - 10\\0 \le k_1 \le k_2 \le k_3 \le k_4}} \eqref{eq:energy1-3.25} \lesssim&~{} \norm{v_2}_{F^{\frac12+}(T)}^2\norm{v_2}_{F^{\frac98}(T)}\norm{w}_{F^{\frac98}(T)}\sum_{|k-k'|\le 5}\norm{P_{k'}v_2}_{F_{k'}(T)}\norm{P_kw}_{F_{k'}(T)}\\
&+\norm{v_2}_{F^{\frac12}(T)}^2\norm{v_2}_{F^{\frac32+}(T)}\norm{w}_{F^{\frac12}(T)}\sum_{|k-k'|\le 5}\norm{P_{k'}v_2}_{F_{k'}(T)}\norm{P_kw}_{F_{k'}(T)}
\end{aligned}\]
at worst. Hence, we conclude that
\begin{equation}\label{eq:energy1-3.26}
\begin{aligned}
\sum_{k\ge 1}2^{2sk}\sup_{t_k \in [0,T]}\sum_{k_1,k_2,k_3,k_4,k_5 \ge 0} \eqref{eq:energy1-3.25} &\lesssim \norm{v_2}_{F^{\frac12+}(T)}^2\norm{v_2}_{F^{\frac32+}(T)}^2\norm{w}_{F^s(T)}^2\\
&+\norm{v_2}_{F^{\frac12+}(T)}^2\norm{v_2}_{F^{\frac32+}(T)}\norm{v_2}_{F^{s}(T)}\norm{w}_{F^s(T)}^2,
\end{aligned}
\end{equation}
whenever $s \ge \frac98$, and
\begin{equation}\label{eq:energy1-3.27}
\sum_{k\ge 1}\sup_{t_k \in [0,T]}\sum_{k_1,k_2,k_3,k_4,k_5 \ge 0} \eqref{eq:energy1-3.25} \lesssim \norm{v_2}_{F^{\frac12+}(T)}^2\norm{v_2}_{F^{\frac32+}(T)}^2\norm{w}_{F^0(T)}^2
\end{equation}
at $L^2$-level.

\begin{remark}\label{rem:resonant2}
From Remark \ref{rem:modified energy} and \ref{rem:resonant1}, we have to check whether quintic resonant interaction components in $\wt{E}_{2}$ and $\wt{E}_{3}$ really vanish or not. When we use the full modified energy in \eqref{eq:new energy1-5} and the full nonlinear term in \eqref{eq:energy-nonlinear3-1}, there should be 9-resonant interaction components in $\wt{E}_2$ as the worst term as follows:
\begin{equation}\label{eq:quintic resonant1}
\sum_{n,\overline{n}\in \Gamma_4(\Z)}\wh{v}_i(n_1)\wh{v}_i(n_2)\chi_k(n-n_{3,1}-n_{3,2})\wh{v}_i(n_{3,1})\wh{v}_i(n_{3,2})\chi_k(n)n|\wh{w}(n)|^2 \hspace{2em} i=1,2,
\end{equation}
\begin{equation}\label{eq:quintic resonant2}
\sum_{n,\overline{n}\in \Gamma_4(\Z)}\wh{v}_1(n_1)\wh{v}_2(n_2)\chi_k(n-n_{3,1}-n_{3,2})\wh{v}_1(n_{3,1})\wh{v}_2(n_{3,2})\chi_k(n)n|\wh{w}(n)|^2, \hspace{2em}
\end{equation}
\begin{equation}\label{eq:quintic resonant3}
\sum_{n,\overline{n}\in \Gamma_4(\Z)}\wh{v}_1(n_1)\wh{v}_2(n_2)\chi_k(n-n_{3,1}-n_{3,2})\wh{v}_i(n_{3,1})\wh{v}_i(n_{3,2})\chi_k(n)n|\wh{w}(n)|^2 \hspace{2em} i=1,2,
\end{equation}
\begin{equation}\label{eq:quintic resonant4}
\sum_{n,\overline{n}\in \Gamma_4(\Z)}\wh{v}_i(n_1)\wh{v}_i(n_2)\chi_k(n-n_{3,1}-n_{3,2})\wh{v}_1(n_{3,1})\wh{v}_2(n_{3,2})\chi_k(n)n|\wh{w}(n)|^2 \hspace{2em} i=1,2
\end{equation}
and for $i=1,j=2$ or $i=2,j=1$,
\begin{equation}\label{eq:quintic resonant5}
\sum_{n,\overline{n}\in \Gamma_4(\Z)}\wh{v}_i(n_1)\wh{v}_i(n_2)\chi_k(n-n_{3,1}-n_{3,2})\wh{v}_j(n_{3,1})\wh{v}_j(n_{3,2})\chi_k(n)n|\wh{w}(n)|^2, \hspace{2em}
\end{equation}
where $\overline{n} = (n_1,n_2,n_{3,1},n_{3,2})$. In view of \eqref{eq:quintic resonant}, we can easily know that \eqref{eq:quintic resonant1} and \eqref{eq:quintic resonant2} vanish, since those are purely real numbers. Furthermore, since we chose $\wt{\alpha}_{i,j}$ as the same number for all $1 \le i \le j \le 2$, by combining some of terms in \eqref{eq:quintic resonant3}, \eqref{eq:quintic resonant4} and \eqref{eq:quintic resonant5}, we can make all terms vanish. For example, we have from the same argument as in \eqref{eq:quintic resonant} that
\[\begin{aligned}
&\overline{\sum_{\overline{n}\in \Gamma_4(\Z)}\wh{v}_1(n_1)\wh{v}_2(n_2)\chi_k(n-n_{3,1}-n_{3,2})\wh{v}_i(n_{3,1})\wh{v}_i(n_{3,2})\chi_k(n)n|\wh{w}(n)|^2}\\
+&\overline{\sum_{\overline{n}\in \Gamma_4(\Z)}\wh{v}_i(n_1)\wh{v}_i(n_2)\chi_k(n-n_{3,1}-n_{3,2})\wh{v}_1(n_{3,1})\wh{v}_2(n_{3,2})\chi_k(n)n|\wh{w}(n)|^2}\\
=&\sum_{\overline{n}\in \Gamma_4(\Z)}\wh{v}_i(n_1)\wh{v}_i(n_2)\chi_k(n-n_{3,1}-n_{3,2})\wh{v}_1(n_{3,1})\wh{v}_2(n_{3,2})\chi_k(n)n|\wh{w}(n)|^2\\
+&\sum_{\overline{n}\in \Gamma_4(\Z)}\wh{v}_1(n_1)\wh{v}_2(n_2)\chi_k(n-n_{3,1}-n_{3,2})\wh{v}_i(n_{3,1})\wh{v}_i(n_{3,2})\chi_k(n)n|\wh{w}(n)|^2,
\end{aligned}\]
for $i=1,2$. Of course, \eqref{eq:quintic resonant5} can vanish by summing $i=1,j=2$ term and $i=2,j=1$ term. 

For the other quintic resonant interaction components, it suffices to consider 
\begin{equation}\label{eq:quintic resonant6}
\left|\int_0^{t_k}\sum_{n,\overline{n}\in \Gamma_4(\Z)}\wh{v}_2(n_1)\wh{v}_2(n_2)\chi_k(n-n_{3,1}-n_{3,2})\wh{v}_2(n_{3,1})\wh{w}(n_{3,2})\chi_k(n)n\wh{v}_2(-n)\wh{w}(n)\; dt\right|.
\end{equation}
Once we use the similar way as in \eqref{eq:energy1-1.5} with \eqref{eq:tri-block estimate-a1}, we can obtain
\[\begin{aligned}
\eqref{eq:quintic resonant6} &\lesssim \sum_{k_1,k_2,k_{3,1},k_{3,2} \ge 0}2^{(k_{min}+k_{thd})/2}\norm{P_{k_1}v_2}_{F_{k_1}(T)}\norm{P_{k_2}v_2}_{F_{k_2}(T)}\norm{P_{k_{3,1}}v_2}_{F_{1,k_{3,1}}(T)}\\
&\hspace{9em} \times\norm{P_{k_{3,2}}w}_{F_{1,k_{3,2}}(T)}\sum_{|k-k'|\le 5}2^{k}\norm{P_{k'}v_2}_{L_T^{\infty}L_x^2}\norm{P_{k'}w}_{L_T^{\infty}L_x^2}\\
&\lesssim \norm{v_2}_{F^{\frac12+}(T)}^3\norm{w}_{F^0(T)}\sum_{|k-k'|\le 5}2^{k}\norm{P_{k'}v_2}_{L_T^{\infty}L_x^2}\norm{P_{k'}w}_{L_T^{\infty}L_x^2}.
\end{aligned}\]
By using the embedding $F^s(T) \hookrightarrow C_TH^s$ for $s \ge 0$, we conclude that
\begin{equation}\label{eq:quintic resonant7}
\sum_{k\ge 1}2^{2sk}\sup_{t_k \in [0,T]}\eqref{eq:quintic resonant6} \lesssim \norm{v_2}_{F^{\frac12+}(T)}^3\norm{w}_{F^0(T)}\norm{v_2}_{F^{s}(T)}\norm{w}_{F^{s+1}(T)},
\end{equation}
whenever $s \ge 0$, and
\begin{equation}\label{eq:quintic resonant7}
\sum_{k\ge 1}\sup_{t_k \in [0,T]}\eqref{eq:quintic resonant6} \lesssim \norm{v_2}_{F^{\frac12+}(T)}^3\norm{v_2}_{F^{1}(T)}\norm{w}_{F^0(T)}^2
\end{equation}
at $L^2$-level.

The same argument also holds for the quintic resonant terms in $\wt{E}_3$.
\end{remark}

For quintic terms in $\wt{E}_{2}$ and $\wt{E}_{3}$, by the symmetries of $n_1,n_2$ and $n_3,n$ variables, respectively, it is enough to consider
\begin{equation}\label{eq:E4.3}
\sum_{k_1,k_2 \ge 0}\left|\int_0^{t_k}\sum_{n,\overline{\N}_{3,n}}\chi_{k_1}(n_1)\wh{v}_2(n_1)\chi_{k_2}(n_2)\wh{N}(v_2)(n_2)\psi_k(n_3)\frac{1}{n_3}\wh{w}(n_3)\chi_k(n)\frac1n\wh{w}(n) \; dt\right|
\end{equation}
and
\begin{equation}\label{eq:E4.4}
\sum_{k_1,k_2 \ge 0}\left|\int_0^{t_k}\sum_{n,\overline{\N}_{3,n}}\chi_{k_1}(n_1)\wh{v}_2(n_1)\chi_{k_2}(n_2)\wh{v}_2(n_2)\psi_k(n_3)\frac{1}{n_3}\wh{v}_2(n_3)\chi_k(n)\frac1n\wh{N}(v_1,v_2,w)(n)\; dt\right|.
\end{equation}
For $N_{1,1}(v_2,w)$ in \eqref{eq:energy-nonlinear1-2}, we use the same way as \eqref{eq:E4-1} and \eqref{eq:E4-2} to estimate \eqref{eq:E4.3} and \eqref{eq:E4.4}, and then we obtain 
\begin{equation}\label{eq:quintic bound2}
\sum_{k\ge 1}2^{2sk}\sup_{t_k \in [0,T]} \left(\eqref{eq:E4.3} + \eqref{eq:E4.4}\right) \lesssim \norm{v_2}_{F^{\frac12}(T)}^4\norm{w}_{F^s(T)}^2,
\end{equation}
for $s \ge 0$.

For the rest quintic terms, it suffices to consider
\begin{equation}\label{eq:energy1-3.28}
\begin{aligned}
&\Big|\int_0^{t_k}\sum_{n,\overline{\N}_{3,n}\N_{3,n_2}}\chi_{k_1}(n_1)\wh{v}_2(n_1)\chi_{k_{2,1}}(n_{2,1})\wh{v}_2(n_{2,1})\chi_{k_{2,2}}(n_{2,2})\wh{v}_2(n_{2,2})\\
&\hspace{5em}\times\chi_{k_{2,3}}(n_{2,3})n_{2,3}^3\wh{v}_2(n_{2,3})\chi_k(n_3)\frac{1}{n_3}\wh{w}(n_3)\chi_k(n)\frac1n\wh{w}(n) \;dt\Big|
\end{aligned}
\end{equation}
under the assumption that $k_{2,1} \le k_{2,2} \le k_{2,3}$, and
\begin{equation}\label{eq:energy1-3.29}
\begin{aligned}
&\Big|\int_0^{t_k}\sum_{n,\overline{\N}_{3,n}\N_{3,n_3}}\chi_{k_1}(n_1)\wh{v}_2(n_1)\chi_{k_2}(n_2)\wh{v}_2(n_2)\chi_{k}(n_3)\chi_{k_{3,1}}(n_{3,1})\wh{v}_2(n_{3,1})\\
&\hspace{5em}\times\chi_{k_{3,2}}(n_{3,2})\wh{v}_2(n_{3,2})\chi_{k_{3,3}}(n_{3,3})n_{3,3}^2\wh{w}(n_{3,3})\chi_k(n)\frac1n\wh{w}(n) \;dt\Big|
\end{aligned}
\end{equation}
under the assumption that $k_{3,1} \le k_{3,2} \le k_{3,3}$. But, both \eqref{eq:energy1-3.28} and \eqref{eq:energy1-3.29} can be controlled by the exact same argument as in the estimation of the quintic term in $\wt{E}_{1}$, and hence we conclude that
\begin{equation}\label{eq:energy1-3.30}
\begin{aligned}
\sum_{k\ge 1}2^{2sk}\sup_{t_k \in [0,T]}&\left(\sum_{k_1,k_{2,1},k_{2,2},k_{2,3}\ge 0} \eqref{eq:energy1-3.28}+\sum_{k_1,k_2,k_{3,1},k_{3,2},k_{3,3} \ge 0}\eqref{eq:energy1-3.29}\right)\\
&\hspace{5em}\lesssim \norm{v_2}_{F^{\frac12+}(T)}^2\norm{v_2}_{F^{\frac32+}(T)}^2\norm{w}_{F^s(T)}^2\\
&\hspace{6em}+\norm{v_2}_{F^{\frac12+}(T)}^2\norm{v_2}_{F^{\frac32+}(T)}\norm{v_2}_{F^{s}(T)}\norm{w}_{F^s(T)}^2,
\end{aligned}
\end{equation}
whenever $s \ge \frac98$, and
\begin{equation}\label{eq:energy1-3.31}
\begin{aligned}
\sum_{k\ge 1}\sup_{t_k \in [0,T]} &\left(\sum_{k_1,k_{2,1},k_{2,2},k_{2,3}\ge 0} \eqref{eq:energy1-3.28}+\sum_{k_1,k_2,k_{3,1},k_{3,2},k_{3,3} \ge 0}\eqref{eq:energy1-3.29}\right)\\
&\hspace{9em} \lesssim \norm{v_2}_{F^{\frac12+}(T)}^2\norm{v_2}_{F^{\frac32+}(T)}^2\norm{w}_{F^0(T)}^2
\end{aligned}
\end{equation}
at $L^2$-level.

Finally, we estimate septic terms in $\wt{E}_{2}$ and $\wt{E}_{3}$. From \eqref{eq:new energy1-3} and the symmetry of functions, it is enough to consider 
\begin{equation}\label{eq:energy1-3.32}
\left|\int_0^{t_k}\sum_{\overline{n} \in \Gamma_8(\Z)}n_6\prod_{j=1}^{6}\chi_{k_j}(n_j)\wh{v}_2(n_j)\chi_{k}(n_7)\frac{1}{n_7}\wh{w}(n_7)\chi_k(n)\frac1n\wh{w}(n) \;dt\right|
\end{equation}
and
\begin{equation}\label{eq:energy1-3.33}
\left|\int_0^{t_k}\sum_{\overline{n} \in \Gamma_8(\Z)}\prod_{j=1}^{6}\chi_{k_j}(n_j)\wh{v}_2(n_j)\chi_{k_7}(n_7)\wh{w}(n_7)\chi_k(n)\frac1n\wh{w}(n) \;dt\right|
\end{equation}
under the assumption that $k_1 \le k_2 \le k_3 \le k_4 \le k_5 \le k_6$. Similarly as in the proof of Proposition \ref{prop:energy1-2}, we obtain by using \eqref{eq:energy1-2.15} that
\begin{equation}\label{eq:energy1-3.34}
\begin{aligned}
&\sum_{k\ge 1}2^{2sk}\sup_{t_k \in [0,T]}\left(\sum_{k_1,k_2,k_3,k_4,k_5,k_6,k_7 \ge 0} \eqref{eq:energy1-3.32}+\sum_{k_1,k_2,k_3,k_4,k_5,k_6\ge0}\eqref{eq:energy1-3.33}\right)\\
&\hspace{15em}\lesssim \norm{v_2}_{F^{\frac12+}(T)}^5\left(\norm{v_2}_{F^{\frac12+}(T)}+\norm{v_2}_{F^{s}(T)}\right)\norm{w}_{F^s(T)}^2\,
\end{aligned}
\end{equation}
whenever $s \ge 0$.

Therefore, by gathering \eqref{eq:energy1-3.10}, \eqref{eq:energy1-3.12}, \eqref{eq:energy1-3.13}, \eqref{eq:energy1-3.17}, \eqref{eq:energy1-3.20}, \eqref{eq:energy1-3.24}, \eqref{eq:energy1-3.26}, \eqref{eq:quintic resonant7}, \eqref{eq:quintic bound2}, \eqref{eq:energy1-3.30} and \eqref{eq:energy1-3.34}, and by recalling the definition of the modified energy \eqref{eq:new energy1-5}, we obtain \eqref{eq:energy1-3.2}. Also, by gathering \eqref{eq:energy1-3.10}, \eqref{eq:energy1-3.8}, \eqref{eq:energy1-3.13}, \eqref{eq:energy1-3.17}, \eqref{eq:energy1-3.21}, \eqref{eq:energy1-3.24}, \eqref{eq:energy1-3.27}, \eqref{eq:quintic resonant7}, \eqref{eq:quintic bound2}, \eqref{eq:energy1-3.31} and \eqref{eq:energy1-3.34}, we get \eqref{eq:energy1-3.1}. 
\end{proof}

As a corollary to Lemma \ref{lem:comparable energy1-2} and Proposition \ref{prop:energy1-3}, we obtain an \emph{a priori} bound of $\norm{w}_{E^s(T)}$ for the difference of two solutions.
\begin{corollary}\label{cor:energy1-3}
Let $s > 2$ and $T \in (0,1]$. Then, there exists $0 < \delta \ll 1$ such that
\begin{equation}\label{eq:energy1-3.1.1}
\begin{aligned}
\norm{w}_{E^0(T)}^2 \lesssim&~{} (1+ \norm{v_{1,0}}_{H^{\frac12+}}^2+\norm{v_{1,0}}_{H^{\frac12+}}\norm{v_{2,0}}_{H^{\frac12+}}+\norm{v_{2,0}}_{H^{\frac12+}}^2)\norm{w_0}_{L_x^2}^2\\
&+\left(\norm{v_1}_{F^{2+}(T)}^2+\norm{v_1}_{F^{2+}(T)}\norm{v_2}_{F^{2+}(T)}+\norm{v_2}_{F^{2+}(T)}^2\right)\norm{w}_{F^0(T)}^2\\
&+\Big(\sum_{j=0}^{4}\norm{v_1}_{F^{\frac32+}(T)}^{4-j}\norm{v_2}_{F^{\frac32+}(T)}^j\Big)\norm{w}_{F^0(T)}^2\\
&+\Big(\sum_{j=0}^{6}\norm{v_1}_{F^{\frac12+}(T)}^{6-j}\norm{v_2}_{F^{\frac12+}(T)}^j\Big)\norm{w}_{F^0(T)}^2.
\end{aligned} 
\end{equation}
and
\begin{equation}\label{eq:energy1-3.2.1}
\begin{aligned}
\norm{w}_{E^s(T)}^2 \lesssim&~{} (1+ \norm{v_{1,0}}_{H^{\frac12+}}^2+\norm{v_{1,0}}_{H^{\frac12+}}\norm{v_{2,0}}_{H^{\frac12+}}+\norm{v_{2,0}}_{H^{\frac12+}}^2)\norm{w_0}_{H^s}^2\\
&+\left(\norm{v_1}_{F^s(T)}^2+\norm{v_1}_{F^s(T)}\norm{v_2}_{F^s(T)}+\norm{v_2}_{F^s(T)}^2\right)\norm{w}_{F^s(T)}^2\\
&+\left(\sum_{i,j=1,2}\norm{v_i}_{F^{\frac12}(T)}\norm{v_j}_{F^{2s}(T)}\right)\norm{w}_{F^0(T)}\norm{w}_{F^s(T)}\\
&+\Big(\sum_{j=0}^{4}\norm{v_1}_{F^{s}(T)}^{4-j}\norm{v_2}_{F^{s}(T)}^j\Big)\norm{w}_{F^s(T)}^2\\
&+\Big(\sum_{j=0}^{3}\norm{v_1}_{F^{s}(T)}^{3-j}\norm{v_2}_{F^{s}(T)}^j\Big)(\norm{v_1}_{F^{2s}(T)}+\norm{v_2}_{F^{2s}(T)})\norm{w}_{F^0(T)}\norm{w}_{F^s(T)}\\
&+\Big(\sum_{j=0}^{6}\norm{v_1}_{F^{s}(T)}^{6-j}\norm{v_2}_{F^{s}(T)}^j\Big)\norm{w}_{F^s(T)}^2,
\end{aligned}
\end{equation} 
for solutions $w \in C([-T,T];H^{\infty}(\T))$ to \eqref{eq:5mkdv8} and $v_1, v_2 \in C([-T,T];H^{\infty}(\T))$ to \eqref{eq:5mkdv4} satisfying $\norm{v_1}_{L_T^{\infty}H_x^{\frac12+}} < \delta$ and $\norm{v_2}_{L_T^{\infty}H_x^{\frac12+}} < \delta$.
\end{corollary}

\section{Proof of Theorem \ref{thm:main}}\label{sec:main}
In this section, we prove the Theorem \ref{thm:main}. The main ingredients are the multilinear estimates and energy estimates which are shown in Sections \ref{sec:nonlinear} and \ref{sec:energy} respectively. The method is the compactness argument which follows basically the idea of Ionescu, Kenig and Tataru \cite{IKT2008}. 
\begin{proposition}\label{prop:small data1-1}
Let $s \ge 0$, $T \in (0,1]$ and $v \in F^s(T)$. Then
\begin{equation}\label{eq:small data1.1}
\sup_{t \in [-T,T]} \norm{v(t)}_{H^s(\T)} \lesssim \norm{v}_{F^s(T)}. 
\end{equation} 
\end{proposition}

\begin{proposition}\label{prop:small data1-2}
Let $T \in (0,1]$ and $v,w \in C([-T,T];H^{\infty})$ satisfying 
\[\pt \wh{v}(n) + i\mu(n) \wh{v}(n) = \wh{w}(n) \mbox{  on  } (-T,T) \times \Z.\]
Then, for any $s \ge 0$, we have
\begin{equation}\label{eq:small data1.2}
\norm{v}_{F^s(T)} \lesssim \norm{v}_{E^s(T)} + \norm{w}_{N^s(T)}.
\end{equation} 
\end{proposition}
\begin{proof}
Even though this problem is under the periodic condition, the proofs of Propositions \ref{prop:small data1-1} and \ref{prop:small data1-2} are exactly same as in \cite{GKK2013}. See Appendix A in \cite{GKK2013} for the proof in terms of the fifth-order KdV flow. We also refer to \cite{IKT2008, Guo2011, KP2015} for the proof.
\end{proof}
\subsection{Small data local well-posedness.}
We first state the local well-posedness result for \eqref{eq:5mkdv} with the small initial data.
\begin{proposition}\label{prop:small data1-3}
Let $s > 2$ and $T \in (0,1]$. Let $u_0 \in H^s(\T)$ be an initial data satisfying
\begin{equation}\label{eq:specified1}
\int_{\T} (u_0(x))^2 \; dx = \gamma_1, \hspace{2em} \int_{\T} (\px u_0(x))^2 + (u_0(x))^4 \; dx = \gamma_2,
\end{equation}
for some $\gamma_1, \gamma_2 \ge 0$, and $\norm{u_0}_{H^s(\T)} \le \delta_0 \ll 1$, for some sufficiently small $\delta_0 > 0$\footnote{$\delta_0$ should satisfy at least $ \delta_0< \delta$ for $\delta > 0$ in Lemma \ref{lem:comparable energy1-1}. Moreover, it should facilitate the continuity mechanism to obtain the \emph{a priori} bound for a smooth solution in the proof of Proposition \ref{prop:small data1-3}.}. Then, \eqref{eq:5mkdv} has a unique solution $u(t)$ with the initial data $u_0$ on $[-T,T]$ satisfying
\begin{align}
&u(t,x) \in C([-T,T];H^s(\T)),\nonumber\\
&\eta(t)\sum_{n \in \Z} e^{i(nx - 20n\int_0^t \norm{u(s)}_{L^4}^4 \; ds)}\wh{u}(t,n) \in C([-T,T];H^s(\T)) \cap F^s(T), \nonumber
\end{align}
where $\eta$ is any cut-off function in $C^{\infty}(\R)$ with $\mathrm{supp}\eta \subset [-T,T]$.

Moreover, the flow map $S_T : H^s \to C([-T,T];H^s(\T))$ is continuous.
\end{proposition}

\begin{remark}
The scaling property \eqref{eq:scaling1}, in addition to a smooth cut-off function, in Subsection \ref{subsec:small data} guarantees that the Cauchy problem for arbitrary data on $[-T,T]$, for some $T > 0$ depending on the initial data is equivalent to the Cauchy problem for small data on $[-1,1]$. Hence the smallness assumption in Proposition \ref{prop:small data1-3} does not depend on the time $T$, since $T \le 1$, if we choose $\delta_0 >0$ sufficiently small. Moreover, \eqref{eq:scaling1} in addition to the following observation
\[\norm{u_{0,\lambda}}_{H^s(\T_{\lambda})} \sim \lambda^{-\frac12}\norm{u_0}_{H^s(\T)}, \qquad \lambda \gg 1,\]
reveals the existence of the time $T = O(\norm{u_0}_{H^s}^{-10})$.
\end{remark}

\begin{proof}
In the following proof, we fix $s > 2$. From the theory of the complete integrability (or inverse spectral method), we know that there is a smooth solution $u$ to \eqref{eq:5mkdv} with $u_0 \in H^{\infty}(\T)$. For $v$ to \eqref{eq:5mkdv3}, since $\norm{u_0}_{H^s(\T)} = \norm{v_0}_{H^s(\T)}$, we use Proposition \ref{prop:small data1-2}, \ref{prop:nonlinear1} (a) and Corollary \ref{cor:energy1-2}\footnote{When we choose $\delta_0 < \delta$, where $\delta$ is given by Lemma \ref{lem:comparable energy1-1}, we can use directly Corollary \ref{cor:energy1-2} instead of Proposition \ref{prop:energy1-2}.} in order to obtain that
\begin{eqnarray*}
\left \{
\begin{array}{l}
\norm{v}_{F^{s}(T')}\lesssim \norm{v}_{E^s(T')} + \sum_{j=1}^4\norm{N_{j}(v)}_{N^s(T')};\\
\sum_{j=1}^4\norm{N_{j}(v)}_{N^s(T')} \lesssim (1+\norm{v}_{F^s(T')}^2)\norm{v}_{F^s(T')}^3;\\
\norm{v}_{E^s(T')}^2 \lesssim (1+ \norm{v_0}_{H^s}^2)\norm{v_0}_{H^s}^2 + (1 + \norm{v}_{F^{s}(T)}^2 + \norm{v}_{F^{s}(T)}^4)\norm{v}_{F^{s}(T)}^4,
\end{array}
\right.
\end{eqnarray*}
for any $T' \in [0,T]$. Let $X(T') = \norm{v}_{E^s(T')} + \sum_{j=1}^4\norm{N_{j}(v)}_{N^s(T')}$. Then we can know that $X(T')$ is non-decreasing and continuous on $[0,T]$ (see \cite{KP2015} and \cite{GO2015} for non-periodic and periodic problem, respectively). If $\delta_0$ is small enough, then by using the bootstrap argument (see \cite{Tao2006}), we can obtain $X(T') \lesssim \norm{v_0}_{H^s}$, and hence
\begin{equation}\label{eq:a priori 0}
\norm{v}_{F^s(T')} \lesssim \norm{v_0}_{H^s(\T)},
\end{equation}
for all $T' \in [0,T]$. This implies
\[\sup_{t \in [-T,T]}\norm{v}_{H^s(\T)} \lesssim \norm{v_0}_{H^s(\T)},\]
by Proposition \ref{prop:small data1-1}. 

We fix $u_0 \in H^s(\T)$ with $\norm{u_0}_{H^s(\T)} \le \delta_0 \ll 1$. Then we can choose a sequence of functions $\set{u_{0,j}}_{j=1}^{\infty} \subset H^{\infty}(\T)$ such that $u_{0,j}$ satisfies \eqref{eq:specified1} and $u_{0,j} \to u_0$ in $H^s(\T)$ as $j \to \infty$. Let $u_j(t) \in H^{\infty}(\T)$ be a solution to \eqref{eq:5mkdv} with the initial data $u_{0,j}$. Then, we first show the sequence $\set{v_j}_{j=1}^{\infty}$ is a Cauchy sequence in $C([-T,T];H^s(\T))$. Let $\epsilon > 0$ be given. For $K \in \Z_+$, let $v_{0,j}^K = P_{\le K}v_{0,j}$. Then $v_j^K(t,x) = P_{\le K}v_j$ satisfies the following frequency localized equation:
\begin{equation}\label{eq:localized equation}
\begin{split}
\pt\chi_{\le K}(n)\Big(\wh{v}_j(n) - i\mu(n)\wh{v}_j(n)\Big)=&-20i\chi_{\le K}(n)n^3|\wh{v}_j(n)|^2\wh{v}_n(n)\\
&+6i \chi_{\le K}(n)n\sum_{\N_{5,n}} \wh{v}_j(n_1)\wh{v}_j(n_2)\wh{v}_j(n_3)\wh{v}_j(n_4)\wh{v}_j(n_5) \\
&+10i\chi_{\le K}(n)n \sum_{\N_{3,n}} \wh{v}_j(n_1)\wh{v}_j(n_2)n_3^2\wh{v}_j(n_3) \\
&+5i\chi_{\le K}(n)n \sum_{\N_{3,n}} (n_1+n_2)\wh{v}_j(n_1)\wh{v}_j(n_2)n_3\wh{v}_j(n_3),
\end{split}
\end{equation}
with the initial data $v_j^K(0) = v_{0,j}^K$. Then, by the triangle inequality, we have 
\[\begin{aligned}
\sup_{t \in [-T,T]}\norm{v_j - v_l}_{H^s(\T)} \le&~{}\sup_{t \in [-T,T]}\norm{v_j - v_j^K}_{H^s(\T)}+\sup_{t \in [-T,T]}\norm{v_j^K - v_l^K}_{H^s(\T)}\\
&+\sup_{t \in [-T,T]}\norm{v_l^K - v_l}_{H^s(\T)},
\end{aligned}\]
and it suffices to show that
\begin{equation}\label{eq:small data1.4}
\sup_{t \in [-T,T]}\norm{v_j^K - v_j}_{H^s(\T)} < \frac{\epsilon}{3}
\end{equation} 
and
\begin{equation}\label{eq:small data1.5}
\sup_{t \in [-T,T]}\norm{v_j^K - v_l^K}_{H^s(\T)} < \frac{\epsilon}{3}.
\end{equation} 
For \eqref{eq:small data1.4}, we use \eqref{eq:nonlinear1} and \eqref{eq:energy1-2.1.1} with \eqref{eq:a priori 0} so that
\begin{equation}\label{eq:C1}
\begin{aligned}
\sup_{t \in [-T,T]}\norm{v_j^K - v_j}_{H^s(\T)} &\lesssim \norm{(I-P_{\le K})v_j}_{F^s(T)}\\
&\le C_1 \norm{v_{0,j}-v_{0,j}^K}_{H^s(T)},
\end{aligned}
\end{equation}
for any $j, K$ and $C_1 \ge 1$. 

In order to deal with \eqref{eq:small data1.5}, from \eqref{eq:small data1.2}, \eqref{eq:nonlinear2} and \eqref{eq:energy1-3.1.1} with \eqref{eq:a priori 0}, we have 
\[\norm{v_1-v_2}_{F^0(T)} \lesssim \norm{v_{1,0} - v_{2,0}}_{L_x^2(\T)},\]
and with this, we obtain from \eqref{eq:small data1.2}, \eqref{eq:nonlinear1} and \eqref{eq:energy1-3.2.1} with \eqref{eq:a priori 0} that
\[\norm{v_1-v_2}_{F^s(T)} \lesssim \norm{v_{1,0} - v_{2,0}}_{H^s(\T)} + (\norm{v_{1,0}}_{H^{2s}(\T)}+\norm{v_{2,0}}_{H^{2s}(\T)})\norm{v_{1,0}-v_{2,0}}_{L_x^2(\T)}.\]
Hence, we conclude that
\begin{equation}\label{eq:C2}
\begin{aligned}
\sup_{t \in [-T,T]}\norm{v_j^K - v_l^K}_{H^s(\T)} &\lesssim \norm{v_j^K - v_l^K}_{F^s(T)}\\
&\lesssim \norm{v_{0,j}^K - v_{0,l}^K}_{H^s(\T)} + (\norm{v_{0,j}^K}_{H^{2s}(\T)}+\norm{v_{0,l}^K}_{H^{2s}(\T)})\norm{v_{0,j}^K-v_{0,l}^K}_{L_x^2(\T)}\\
&\lesssim C_2\norm{v_{0,j}^K - v_{0,l}^K}_{H^s(\T)} + C_3K^s\norm{v_{0,j}^K-v_{0,l}^K}_{L^2(\T)},
\end{aligned}
\end{equation}
for any $j,l,K$ and $C_2, C_3 \ge 1$. 

We note in \eqref{eq:C1} and \eqref{eq:C2} that we can apply the nonlinear estimates and the energy estimates to \eqref{eq:localized equation} (also high frequency localized equation) in order to obtain \eqref{eq:a priori 0} for high frequency localized solution (for \eqref{eq:C1}) and Corollary \ref{cor:energy1-3} (for \eqref{eq:C2}) by following the similar argument as in Section \ref{sec:nonlinear} and as the proof of Proposition \ref{prop:energy1-3}. Precisely, thanks to the support property, we can always pick out the $\norm{(I-P_{\le K})v_j}_{F^s(T)}$ and $\norm{v_j^K-v_l^K}_{F^0}$ (hence we get $\norm{(I-P_{\le K})v_{0,j}}_{H^s}$ and $\norm{v_{1,0}-v_{2,0}}_{L_x^2(\T)}$, respectively), even if we consider the energy estimate for \eqref{eq:localized equation} (also high frequency localized equation) instead of \eqref{eq:5mkdv4}.

For constants $C_1, C_2, C_3 \ge 1$ in \eqref{eq:C1} and \eqref{eq:C2}, let $C = \max(C_1,C_2,C_3)$. Since $v_0 \in H^s(\T)$. Then there exists $L \in \Z_+$ such that $K \ge L \ge 1$ implies
\begin{equation}\label{eq:limit1-1}
\norm{(I-P_K)v_0}_{H^s(\T)} < \frac{\epsilon}{9C}.
\end{equation}

Moreover, for fixed $K \ge L \ge 1$, since $v_{0,j} \to v_0$ in $H^s(\T)$, as $j \to \infty$, there is $N = N(K)$ such that $j,l \ge N$ implies
\begin{equation}\label{eq:limit1-2}
\norm{v_{0,j} - v_0}_{H^s(\T)}, \norm{v_{0,j} - v_{0,l}}_{H^s(\T)} < \frac{\epsilon}{9C}
\end{equation}
and
\begin{equation}\label{eq:limit1-3}
\norm{v_{0,j}^K - v_{0,l}^K}_{L^2(\T)} < \frac{\epsilon}{9CK^s}.
\end{equation}

Then, by choosing suitable $K \ge L$ and $j,l \ge N(K)$, we have from \eqref{eq:limit1-1}, \eqref{eq:limit1-2} and \eqref{eq:limit1-3} that

Therefore, when we choose suitable $j,l \ge N$ and $K \ge L$, by using \eqref{eq:limit1-1} and \eqref{eq:limit1-2}, we have
\begin{equation}\label{eq:limit argument1}
\begin{aligned}
\sup_{t \in [-T,T]}\norm{v_j^K - v_j}_{H^s(\T)}&\le C_1 \norm{v_{0,j}-v_{0,j}^K}_{H^s(T)}\\
&\le C_1 \Big(\norm{v_{0,j}-v_{0}}_{H^s(T)} +\norm{v_0 - v_0^K}_{H^s(T)} + \norm{v_0^K - v_{0,j}^K}_{H^s(T)}\Big)\\
&< \frac{\epsilon}{3}
\end{aligned}
\end{equation}
and
\begin{equation}\label{eq:limit argument2}
\begin{aligned}
\sup_{t \in [-T,T]}\norm{v_j^K - v_l^K}_{H^s(\T)} &\le C_2\norm{v_{0,j}^K - v_{0,l}^K}_{H^s(\T)} + C_3K^s\norm{v_{0,j}^K-v_{0,l}^K}_{L_x^2(\T)}\\
&< C_2\norm{v_{0,j}^K - v_0^K}_{H^s(\T)} + C_2\norm{v_{0,l}^K - v_0^K}_{H^s(\T)} + \frac{\epsilon}{9}\\
&< \frac{\epsilon}{3}.
\end{aligned}
\end{equation}

We apply those result in addition to \eqref{eq:limit argument1} to \eqref{eq:limit argument2} to complete the limit argument. Hence we obtain the solution as the limit. The uniqueness of the solution and the continuity of the flow map come from a similar argument, so we omit detail.

Now, it remains to show that the local well-posedness of \eqref{eq:5mkdv3} implies that of \eqref{eq:5mkdv}. In view of the definition of the nonlinear transformation \eqref{eq:modified solution}, it suffices to show the bi-continuity property of the nonlinear transformation in $C([-T,T];H^s(\T))$. 
\begin{lemma}\label{lem:bi-continuity}
Let $s \ge \frac14$ and $0 < T < \infty$. Then, $\NT(u)$ defined as in \eqref{eq:modified solution} is bi-continuous from a ball in $C([-T,T];H^s(\T))$ to itself. 
\end{lemma}
\begin{proof}
We only show the continuity of $\NT^{-1}$, since the proof of the continuity of $\NT$ is similar as and easier than that of $\NT^{-1}$. Precisely, when $v_k \in C_TH^s$ converges to $v$ in $C_TH^s$ as $k \to \infty$, we need to show
\[u_k = \NT^{-1}(v_k) \to \NT^{-1}(v) = u\hspace{1em}\mbox{in}\hspace{1em} C_TH^s, \hspace{1em} \mbox{as}\hspace{1em} k \to \infty.\]
Fix $0 < T < \infty$. We assume that $\norm{v_k}_{L_T^{\infty}H^s}, \norm{v}_{L_T^{\infty}H^s} \le K$, for some $K > 0$. Observe that 
\begin{align*}
\wh{u}_k(n) -\wh{u}(n) =&~{} e^{ic_3n\int_0^t \norm{u_k(s)}_{L^4}^4 \; ds}\wh{v}_k(n) - e^{ic_3n\int_0^t \norm{u(s)}_{L^4}^4 \; ds}\wh{v}(n) \\
=&~{}\left[e^{ic_3n\int_0^t \norm{u_k(s)}_{L^4}^4 \; ds} - e^{ic_3n\int_0^t \norm{u(s)}_{L^4}^4 \; ds} \right]\wh{v}_k(n) \\
&+e^{ic_3n\int_0^t \norm{u(s)}_{L^4}^4 \; ds}(\wh{v}_k(n) - \wh{v}(n))\\
=&~{}e^{ic_3n\int_0^t \norm{u(s)}_{L^4}^4 \; ds}\left[e^{ic_3n\int_0^t \norm{u_k(s)}_{L^4}^4 - \norm{u(s)}_{L^4}^4 \; ds} - 1 \right]\wh{v}_k(n) \\
&+e^{ic_3n\int_0^t \norm{u(s)}_{L^4}^4 \; ds}(\wh{v}_k(n) - \wh{v}(n))
\end{align*}
for $n \neq 0$, and
\[\wh{u}_k(0) -\wh{u}(0) = \wh{v}_k(0) - \wh{v}(0).\] 
Then, for fixed $s \ge 1/4$ and $t \in [-T,T]$, we have
\begin{align}
\norm{u_k(t) - u(t)}_{H^s}^2 \le&~{} |\wh{v}_k(0) - \wh{v}(0)|^2 \label{eq:bi-conti1}\\
&+2^{s+1}\sum_{|n| \ge 1}\left|e^{ic_3n\int_0^t \norm{u_k(s)}_{L^4}^4 - \norm{u(s)}_{L^4}^4 \; ds} - 1\right|^2|n|^{2s}|\wh{v}_k(n)|^2 \label{eq:bi-conti2}\\
&+2^{s+1}\sum_{|n| \ge 1}|n|^{2s}|\wh{v}_k(n) - \wh{v}(n)|^2 \label{eq:bi-conti3}.
\end{align}
Let $\varepsilon > 0$ be given. Since $e^{i\theta}$ is continuous at $\theta = 0$, there exists  $\delta>0$ such that 
\begin{equation}\label{eq:limit1}
|\theta| < \delta \hspace{1em}\Rightarrow\hspace{1em} \left|e^{i\theta}-1\right| <\frac{\varepsilon}{2^{\frac{s+1}{2}}\cdot \sqrt{6K}},
\end{equation}
and $\norm{v}_{L_T^{\infty}H^s} \le K$ implies that there exists $M>0$ such that
\begin{equation}\label{eq:limit2}
\sum_{|n| > M} |n|^{2s}|\wh{v}(n)|^2 < \frac{\varepsilon^2}{2^{s+1}\cdot 24}.
\end{equation}
Moreover, since $v_k \to v$ in $C_TH^s$ as $k \to \infty$, there exist $N_0,N_1>0$ such that
\begin{equation}\label{eq:limit3}
k \ge N_0 \hspace{1em}\Rightarrow\hspace{1em} \norm{v_k - v}_{L_T^{\infty}H^s} < \frac{\varepsilon}{2^{\frac{s+1}{2}}\cdot 3}
\end{equation}
and
\begin{equation}\label{eq:limit4}
k \ge N_1 \hspace{1em}\Rightarrow\hspace{1em} \norm{v_k - v}_{L_T^{\infty}H^s} < \frac{\delta}{2c_3MTK^3}.
\end{equation}
Let $N := \max(N_0,N_1)$. If $k \ge N$, from \eqref{eq:limit3}, we can control \eqref{eq:bi-conti1} and \eqref{eq:bi-conti3} as
\begin{equation}\label{eq:bi-conti4}
|\wh{v}_k(0) - \wh{v}(0)|^2 < \frac{\varepsilon^2}{3}
\end{equation}
and
\begin{equation}\label{eq:bi-conti5}
2^{s+1}\sum_{|n| \ge 1}|n|^{2s}|\wh{v}_k(n) - \wh{v}(n)|^2 < \frac{\varepsilon^2}{3}
\end{equation}
Now, we consider \eqref{eq:bi-conti2}. Observe from the Plancherel's theorem that\footnote{This observation is essential to obtain the bi-continuity property of the nonlinear transformation $\NT(u)$.}
\begin{equation}\label{eq:bi-conti6}
\begin{aligned}
\norm{u}_{L^4}^2 &= \norm{u^2}_{L^2}=\norm{\wh{u} \ast \wh{u}}_{l^2} \\
&=\left(\sum_{n \in \Z} \left|\sum_{n_1 \in\Z}\wh{u}(n_1)\wh{u}((n-n_1) \right|^2 \right)^{\frac12} \\ 
&=\left(\sum_{n \in \Z} \left|e^{ic_3n\int_0^t \norm{u(s)}_{L^4}^4 \; ds}\sum_{n_1 \in\Z}\wh{v}(n_1)\wh{v}(n-n_1) \right|^2\right)^{\frac12} \\
&=\left(\sum_{n \in \Z} \left|\sum_{n_1 \in\Z}\wh{v}(n_1)\wh{v}((n-n_1) \right|^2\right)^{\frac12}\\
&=\norm{v^2}_{L^2}=\norm{v}_{L^4}^2.
\end{aligned}
\end{equation}
Then, by using the triangle inequality and the Sobolev embedding, we have
\begin{equation}\label{eq:bi-conti7}
\begin{aligned}
\Big|\int_0^t&\norm{v_k(s)}_{L^4}^4 - \norm{v(s)}_{L^4}^4 \; dt \Big| \\
&= \left|\int_0^T(\norm{v_k(s)}_{L^4} - \norm{v(s)}_{L^4})(\norm{v_k(t)}_{L^4}+\norm{v(t)}_{L^4})(\norm{v_k(t)}_{L^4}^2+\norm{v(t)}_{L^4}^2) \; dt \right|\\
&\le\int_0^T\norm{v_k(t) - v(t)}_{L^4}(\norm{v_k(t)}_{L^4}+\norm{v(t)}_{L^4})(\norm{v_k(t)}_{L^4}^2+\norm{v(t)}_{L^4}^2) \; dt\\
&\lesssim T(\sup_{t \in [-T,T]}\norm{v_k}_{H^s(\T)}^3+\sup_{t \in [-T,T]}\norm{v}_{H^s(\T)}^3)\sup_{t \in [-T,T]}\norm{v_k-v}_{H^s(\T)}\\
&\lesssim 2TK^3\sup_{t \in [-T,T]}\norm{v_k-v}_{H^s(\T)},
\end{aligned}
\end{equation}
for $s \ge \frac14$.

We divide the summation in \eqref{eq:bi-conti2} into
\[\sum_{1 \le |n| \le M} + \sum_{|n| > M}.\]
Then, for $1 \le |n| \le M$, if $k \ge N$, from \eqref{eq:bi-conti6}, \eqref{eq:bi-conti7} and \eqref{eq:limit4}, we have
\[\left|c_3n\int_0^t \norm{u_k(s)}_{L^4}^4 - \norm{u(s)}_{L^4}^4 \; ds\right| < \delta\]
which implies
\[\left|e^{ic_3n\int_0^t \norm{u_k(s)}_{L^4}^4 - \norm{u(s)}_{L^4}^4 \; ds} - 1\right|^2 < \frac{\varepsilon^2}{2^{s+1}\cdot6K^2},\]
by using \eqref{eq:limit1}.

For $|n| > M$, since
\[\left|e^{ic_3n\int_0^t \norm{u_k(s)}_{L^4}^4 - \norm{u(s)}_{L^4}^4 \; ds} - 1\right|^2 \le 4,\]
by using \eqref{eq:limit2}, we have
\[2^{s+1}\sum_{|n| > M}|n|^{2s} \left|e^{ic_3n\int_0^t \norm{u_k(s)}_{L^4}^4 - \norm{u(s)}_{L^4}^4 \; ds} - 1\right|^2|n|^{2s}|\wh{v}_k(n)|^2 < \frac{\varepsilon^2}{6}.\]
Hence, we have
\begin{equation}\label{eq:bi-conti8}
2^{s+1}\sum_{|n| \ge 1}|n|^{2s} \left|e^{ic_3n\int_0^t \norm{u_k(s)}_{L^4}^4 - \norm{u(s)}_{L^4}^4 \; ds} - 1\right|^2|n|^{2s}|\wh{v}_k(n)|^2 < \frac{\varepsilon^2}{3}.
\end{equation}
Together with \eqref{eq:bi-conti4}, \eqref{eq:bi-conti5} and \eqref{eq:bi-conti8}, we obtain
\[\norm{u_k - u}_{H^s} < \varepsilon,\]
which completes the proof of Lemma \ref{lem:bi-continuity}.
\end{proof}

\begin{remark}\label{rem:full nonlinear transform}
One can define the fully nonlinear transform as
\[\NT(u)(t,x) = v(t,x) := \frac{1}{\sqrt{2\pi}}\sum_{n \in \Z} e^{i(nx - \int_0^t b_1n^3\norm{u(s)}_{L_x^2}^2 + b_2n\norm{u(s)}_{\dot{H}^1}^2 + b_3n\norm{u(s)}_{L_x^4}^4 \; ds)}\wh{u}(t,n),\]
where $b_1 = a_2$, $b_2 = a_1 + 3a_3$ and $b_3 = -a_4$. It enables to transform \eqref{eq:5mkdv5} into the following
\[\begin{split}
\pt\wh{v}(n) - in^5\wh{v}(n) =&~{}  -i(a_1-3a_2+a_3)|\wh{v}(n)|^2\wh{v}(n) +ia_1 \sum_{\N_{3,n}} \wh{v}(n_1)n_2\wh{v}(n_2)n_3^2\wh{v}(n_3) \\
&+ia_2 \sum_{\N_{3,n}} \wh{v}(n_1)\wh{v}(n_2)n_3^3\wh{v}(n_3) +ia_3 \sum_{\N_{3,n}} n_1\wh{v}(n_1)n_2\wh{v}(n_2)n_3\wh{v}(n_3)\\
&+ia_4 \sum_{\N_{5,n}} \wh{v}(n_1)\wh{v}(n_2)\wh{v}(n_3)\wh{v}(n_4)n_5\wh{v}(n_5).
\end{split}\]
Similarly as above, we can have the local-in-time solution to above equation. Moreover, one can show that the nonlinear transform is bi-continuous from a ball in $C([-T,T];H^s(\T))$ to itself for $s \ge 1$. Indeed, for the $b_3n\norm{u(s)}_{L_x^4}^4$ term, we use the exact same argument in the proof of Lemma \ref{lem:bi-continuity} to show the bi-continuity property for $s \ge 1/4$. We now focus on the other terms ($b_1n^3\norm{u(s)}_{L_x^2}^2$ and $b_2n\norm{u(s)}_{\dot{H}^1}^2$). We point out that the key point in the proof of Lemma \ref{lem:bi-continuity} is to show \eqref{eq:bi-conti6}. From the definition of the nonlinear transform, since we can know that the nonlinear transform preserves $L^2$-based norm ($\norm{u}_{L^2}=\norm{v}_{L^2}$ and $\norm{u}_{\dot{H}^1}=\norm{v}_{\dot{H}^1}$), we also prove the bi-continuity property for $s \ge 1$\footnote{The regularity threshold $s \ge 1$ comes from the $\norm{v}_{\dot{H}^1}$ factor in the \eqref{eq:bi-conti7} type estimate.} by following the similar way in the proof of Lemma \ref{lem:bi-continuity}. 

It is remarkable that this observation enables to remove all non-trivial resonances in the nonlinearity, and hence the fully non-integrable equation can be studied without any restriction. However, we still emphasize that the integrable structure is useful to capture the natural cancellation property of some of non-trivial resonances.

We also remark that this observation does not guarantee the cancellation of $n^k\norm{u}_{L^p}^p$ type term, for $p \neq 2$ and $k \ge 2$, since we do not use both the invariance of $L^2$-based norm and the property of convolution operator as in \eqref{eq:bi-conti6}. Thus we guess that the study of higher-order equations in the hierarchy \eqref{eq:hamiltonian} may depend on its integrable structure.
\end{remark}

From Lemma \ref{lem:bi-continuity}, we can complete the proof of Proposition \ref{prop:small data1-3}.
\end{proof}

\subsection{Local well-posedness with arbitrary initial data.}\label{subsec:small data}
Now, let us complete the proof of Theorem \ref{thm:main}. In order to extend the result of Proposition \ref{prop:small data1-3} to the local well-posedness for the arbitrary initial data, we can use the scaling argument, since this problem is scaling sub-critical. More precisely, by the scaling symmetry, we know for $\lambda \ge 1$ that
\begin{equation}\label{eq:scaling1}
u_{\lambda}(t,x) = \lambda^{-1}u(\lambda^{-5}t,\lambda^{-1}x)
\end{equation}
is also the solution to \eqref{eq:5mkdv} if $u$ is the solution to \eqref{eq:5mkdv}. Since $u_{\lambda}$ is the $2\pi\lambda$-periodic function, we need to modify slightly all estimates obtained in previous sections for the small data problem. But, since proofs follow the arguments in Sections \ref{sec:preliminary}, \ref{sec:L2 block estimate}, \ref{sec:nonlinear} and \ref{sec:energy}, so let us point out only different things. We start with introducing some notations adapted to the $2\pi\lambda$-periodic setting.

We put $\T_{\lambda} = [0,2\pi\lambda]$ and $\Z_{\lambda} := \set{n/\lambda : n \in \Z}$. For a function $f$ on $\T_{\lambda}$, we define
\[\int_{\T_{\lambda}} f(x) \; dx := \int_0^{2\pi\lambda} f(x) \; dx.\]
For a function $f$ on $\Z_{\lambda}$, we define normalized counting measure $dn$:
\begin{equation}\label{eq:counting measure}
\int_{\Z_{\lambda}} f(n) \; dn := \frac{1}{\lambda}\sum_{n\in\Z_{\lambda}} f(n)
\end{equation}
and $\ell_n^2(\lambda)$ norm:
\[\norm{f}_{\ell_n^2(\lambda)}^2 := \int_{\Z_{\lambda}}|f(n)|^2 \; dn.\]
We define the Fourier transform of $f$ with respect to the spatial variable by
\[\wh{f}(n) := \frac{1}{\sqrt{2\pi}}\int_0^{2\pi\lambda} e^{-ixn}f(x)\; dx, \hspace{3em} n \in \Z_{\lambda},\]
and we have the Fourier inversion formula
\[f(x) := \frac{1}{\sqrt{2\pi}}\int_{\Z_{\lambda}} e^{ixn}\wh{f}(n)\; dn, \hspace{3em} x \in \T_{\lambda}.\]
Of course, we can naturally define the space-time Fourier transform similarly.

Then the usual properties of the Fourier transform hold:
\begin{equation}\label{eq:Plancherel}
\norm{f}_{L_x^2(\T_{\lambda})} = \norm{\wh{f}}_{\ell_n^2(\lambda)},
\end{equation}
\[\int_0^{2\pi\lambda}f(x)\overline{g}(x) \; dx = \int_{\Z_{\lambda}}\wh{f}(n)\overline{\wh{g}}(n) \; dn,\]
\[\wh{fg}(n) = (\wh{h} \ast \wh{g})(n) = \int_{\Z_{\lambda}}\wh{f}(n-n_1)\wh{g}(n_1) \; dn_1\]
and for $m \in \Z_+$,
\begin{equation}\label{eq:derivatives}
\px^mf(x) = \int_{\Z_{\lambda}}e^{ixn}(in)^{m}\wh{f}(n) \; dn.
\end{equation}
Together with \eqref{eq:Plancherel} and \eqref{eq:derivatives}, we can define the Sobolev space $H^s(\T_{\lambda})$ with the norm
\[\norm{f}_{H^s(\T_{\lambda})} = \norm{\bra{n}^s\wh{f}(n)}_{\ell_n^2(\lambda)}.\]

Under those observations, we consider \eqref{eq:5mkdv} for a $2\pi\lambda$-periodic solution $u_{\lambda}$ defined as in \eqref{eq:scaling1}, but we will denote $u_{\lambda}$ by $u$ (also $v$ and $w$), for simplicity, in this section. Recall the fifth-order modified KdV equation \eqref{eq:5mkdv}. 
\begin{equation}\label{eq:5mkdv9}
\pt u - \px^5 u + 40u\px u \px^2u + 10 u^2\px^3u + 10(\px u)^3 - 30u^4 \px u = 0.
\end{equation}
Similarly as in Section \ref{sec:preliminary}, we take the Fourier coefficient in the spatial variable of \eqref{eq:5mkdv9}\footnote{Before doing that, we first change the nonlinear term into the divergence form similarly as in \eqref{eq:5mkdv2}.} to obtain
\begin{equation}\label{eq:5mkdv11}
\begin{split}
\pt\wh{u}(n) - in^5\wh{u}(n) =&~{} 10in\int_{\Z_{\lambda}^2}\wh{u}(n_1)\wh{u}(n_2)(n-n_1-n_2)^2\wh{u}(n-n_1-n_2) \; dn_1dn_2\\
&+10in\int_{\Z_{\lambda}^2}\wh{u}(n_1)n_2\wh{u}(n_2)(n-n_1-n_2)\wh{u}(n-n_1-n_2) \; dn_1dn_2\\
&+6in\int_{\Z_{\lambda}^5}\wh{u}(n_1)\wh{u}(n_2)\wh{u}(n_3)\wh{u}(n_4)\wh{u}(n-n_1-n_2-n_3-n_4) \; dS,
\end{split}
\end{equation}
where $dS = dn_1dn_2dn_3dn_4$. Since $2\pi\lambda$-periodic solution still satisfies all conservation laws, by considering the cubic and the quintic resonant interactions, we can reduce \eqref{eq:5mkdv11} to 
\[\begin{aligned}
\pt\wh{v}(n) - i&(n^5 + c_{1,\lambda}n^3 + c_{2,\lambda}n)\wh{v}(n) \\
=&~{} -\frac{20i}{\lambda^2}n^3|\wh{v}(n)|^2\wh{v}(n)\\
&+10in\int_{\Z_{\lambda}^2,\N_{3,n,\lambda}}\wh{v}(n_1)\wh{v}(n_2)(n-n_1-n_2)^2\wh{v}(n-n_1-n_2) \; dn_1dn_2\\
&+10in\int_{\Z_{\lambda}^2,\N_{3,n,\lambda}}\wh{v}(n_1)n_2\wh{v}(n_2)(n-n_1-n_2)\wh{v}(n-n_1-n_2) \; dn_1dn_2\\
&+6in\int_{\Z_{\lambda}^5,\N_{5,n,\lambda}}\wh{v}(n_1)\wh{v}(n_2)\wh{v}(n_3)\wh{v}(n_4)\wh{v}(n-n_1-n_2-n_3-n_4) \; dn_1dn_2dn_3dn_4,
\end{aligned}\]
where
\[c_{1,\lambda} = \frac{10}{\lambda}\norm{u_0}_{L^2(\T_{\lambda})}^2, \hspace{2em} c_{2,\lambda}=\frac{10}{\lambda}(\norm{u_0}_{\dot{H}^1(\T_{\lambda})}^2+\norm{u_0}_{L^4(\T_{\lambda})}^4),\]
and $\N_{3,n,\lambda}$ and $\N_{5,n,\lambda}$ are defined similarly as in \eqref{eq:3nonres} and \eqref{eq:5nonres}, respectively, for $\Z_{\lambda}$-variables. Moreover, $v$ is also defined similarly as in \eqref{eq:modified solution} by
\[v(t,x) := \frac{1}{\sqrt{2\pi}}\int_{\Z_{\lambda}} e^{i(nx - c_{3,\lambda}n \int_0^t \norm{u(s)}_{L^4}^4 \; ds)}\wh{u}(t,n) \; dn,\]
where $c_{3,\lambda}=\frac{20}{\lambda}$. Let 
\[\mu_{\lambda}(n) = n^5 + c_{1,\lambda}n^3 + c_{2,\lambda}n.\]
From those observations, we change function spaces $X_k$, $F_{k}$, $N_k$ and $E_{k}$ by $X_{k,\lambda}$\footnote{All properties of $X_k$-norm still hold for $X_{k,\lambda}$.}, $F_{k,\lambda}$, $N_{k,\lambda}$ and $E_{\lambda}^s(T)$ with norms
\[\norm{f}_{X_{k,\lambda}} = \sum_{j \ge 0}2^{j/2}\norm{\eta_j(\tau - \mu_{\lambda}(n))\cdot f(\tau,n)}_{L_{\tau}^2\ell_n^2(\lambda)},\]
\[\norm{f}_{F_{k,\lambda}} = \sup_{t_k\in\R}\norm{\ft[\eta_0(2^{2k}(t-t_k))\cdot f]}_{X_{k,\lambda}},\]
\[\norm{f}_{N_{k,\lambda}} = \sup_{t_k\in\R}\norm{(\tau-\mu_{1,\lambda}(n) + i2^{2k})^{-1}\ft[\eta_0(2^{2k}(t-t_k))\cdot f]}_{X_{k,\lambda}}\]
and
\[\norm{u}_{E_{\lambda}^s(T)}^2 = \norm{P_{\le0}u(0)}_{L^2(\T_{\lambda})}^2 + \sum_{k \ge 1}\sup_{t_k\in[-T,T]}2^{2sk}\norm{P_{k}u(t_k)}_{L^2(\T_{\lambda})}^2.\]
Now, we check the nonlinear estimate
\begin{equation}\label{eq:lambda nonlinear}
\begin{aligned}
&\sum_{i=1,3,4} \norm{N_{i}(u,v,w)}_{N_{\lambda}^s(T)} + \norm{N_2(v_1,v_2,v_3,v_4,v_5)}_{N_{\lambda}^s(T)} \\
&\hspace{7em}\lesssim \lambda^{\frac12}\norm{u}_{F_{\lambda}^s(T)}\norm{v}_{F_{\lambda}^s(T)}\norm{w}_{F_{\lambda}^s(T)} + \lambda^{\frac12}\prod_{i=1}^{5}\norm{v_i}_{F_{\lambda}^s(T)}
\end{aligned}
\end{equation}
and the energy estimate
\begin{equation}\label{eq:lambda energy}
\begin{aligned}
\norm{v}_{E_{1,\lambda}^s(T)}^2 \lesssim&~{} (1+ \norm{v_0}_{H^s(\T_{\lambda})}^2)\norm{v_0}_{H^s(\T_{\lambda})}^2 \\
&+ \lambda^{\frac12}(1 +\norm{v}_{F_{\lambda}^{\frac12+}(T)}^2 + \norm{v}_{F_{\lambda}^{\frac12+}(T)}^4)\norm{v}_{F_{\lambda}^{2+}(T)}^2\norm{v}_{F_{\lambda}^s(T)}^2,
\end{aligned}
\end{equation}
in Sections \ref{sec:nonlinear} and \ref{sec:energy}, respectively.

First, consider the $L^2$-block estimates in Section \ref{sec:L2 block estimate}. Define the functional for the trilinear estimate by
\[\begin{aligned}
J_{\lambda}&(f_{k_1,j_1},f_{k_2,j_2},f_{k_3,j_3},f_{k_4,j_4}) =\\
& \int_{\substack{\Z_{\lambda}^3\\ \N_{3,n_4,\lambda}}}\int_{\overline{\zeta}\in\Gamma_4(\R)}f_{k_1,j_1}(\zeta_1,n_1)f_{k_2,j_2}(\zeta_2,n_2)f_{k_3,j_3}(\zeta_3,n_3)f_{k_4,j_4}(\zeta_4 + G(n_1,n_2,n_3),n_4).
\end{aligned}\]
Then, in view of the proof of Lemma \ref{lem:tri-L2}, since we use the normalized counting measure \eqref{eq:counting measure}, we obtain the exact same result as in Lemma \ref{lem:tri-L2} even for $2\pi\lambda$-periodic functions, while the threshold of  restriction $2^{j_{sub}}2^{-4k_{max}}=1$ is replaced by $\lambda2^{j_{sub}}2^{-4k_{max}}=1$. In fact, since the $L^2$-block estimates still hold independent on $\lambda$, we can use the similar way to obtain nonlinear estimates. The only different thing is to use the fact that
\begin{equation}\label{eq:modulation}
2^{j_{\max}} \gtrsim |(n_1+n_2)(n_1+n_3)(n_2+n_3)|(n_1^2+n_2^2+n_3^2+n^2), \hspace{1em} n_1,n_2,n_3,n \in \Z_{\lambda}.
\end{equation}
If $|k_{min}-k_{max}| \le 5$, we have
\begin{equation}\label{eq:lambda bad}
2^{j_{max}} \gtrsim \lambda^{-2}2^{3k_{max}},
\end{equation}
and if $k_{thd} \le k_{max} - 10$ and $|k_{min}-k_{thd}| \le 5$, we obtain
\[2^{j_{max}} \gtrsim \lambda^{-1}2^{4k_{max}}.\]
To get \eqref{eq:lambda nonlinear}, we follows almost same argument as in the nonlinear estimate, while we use the short time advantage $j_{max} \ge 2k_{max}$ instead of \eqref{eq:lambda bad} in the proof of Lemma \ref{lem:nonres1}. 

For \eqref{eq:lambda energy}, we define the modified energy similarly as in \eqref{eq:new energy1-1} and \eqref{eq:new energy1-2} by
\[\begin{aligned}
E_{\lambda,k}(v)(t) =&~{} \norm{P_kv(t)}_{L_x^2(\T_{\lambda})}^2 \\
&+ \mbox{Re}\left[\alpha_{\lambda} \int_{\Z_{\lambda}^3,\overline{\N}_{3,n}}\wh{v}(n_1)\wh{v}(n_2)\psi_k(n_3)\frac{1}{n_3}\wh{v}(n_3)\chi_k(n)\frac1n\wh{v}(n) \; dn_1d_2dn\right]\\
&+ \mbox{Re}\left[\beta_{\lambda} \int_{\Z_{\lambda}^3,\overline{\N}_{3,n}}\wh{v}(n_1)\wh{v}(n_2)\chi_k(n_3)\frac{1}{n_3}\wh{v}(n_3)\chi_k(n)\frac1n\wh{v}(n)\; dn_1d_2dn\right]
\end{aligned}\]
and
\[E_{\lambda,T}^s(v) = \norm{P_0v(t)}_{L_x^2(\T_{\lambda})}^2 + \sum_{k \ge 1}2^{2sk} \sup_{t_k \in [-T,T]} E_{\lambda,k}(v)(t_k).\]
From \eqref{eq:modulation}, we need to change Lemma \ref{lem:energy1-1} (a) and (c) as follows:\footnote{Similarly as the nonlinear estimate, we use the short time advantage instead of maximum modulation effect.}
\[\left| \sum_{n_4,\overline{\N}_{3,n_4}} \int_0^T  \wh{v}_1(n_1)\wh{v}_2(n_2)\wh{v}_3(n_3)\wh{v}_4(n_4) \; dt\right| \lesssim 2^{k_4}\prod_{i=1}^{4}\norm{v_i}_{F_{k_i}(T)}\]
and
\[\left| \sum_{n_4,\overline{\N}_{3,n_4}} \int_0^T  \wh{v}_1(n_1)\wh{v}_2(n_2)\wh{v}_3(n_3)\wh{v}_4(n_4) \; dt\right| \lesssim \lambda^{\frac12}2^{-k_4}\prod_{i=1}^{4}\norm{v_i}_{F_{k_i}(T)},\]
respectively.

For the difference of two solutions, we use the similar argument as above, and then by following the small data well-posedness argument in Section \ref{subsec:small data} in addition to the standard scaling-rescaling argument, we can complete the proof of Theorem \ref{thm:main}.

\appendix

%\section{}\label{app:smooth sol}
\section{High regularity well-posedness}\label{app:smooth sol}
In this Appendix, we show the high regularity well-posedness for the non-integrable fifth-order modified KdV equation. We first consider the $\varepsilon$-parabolic equation, and for the smooth solution to the parabolic equation, we show an \emph{a priori} bound for the solution $u$. Afterward, in addition to an \emph{a priori} bound and bootstrap argument, we use the approximation method to show that the solution of $\varepsilon$-parabolic equation converges to the solution of the fifth-order modified KdV equation. Finally, we use the Bona-Smith \cite{BS1978} argument to obtain the high regularity well-posedness result for the fifth-order modified KdV equation. The main difficulty is to obtain the energy estimate for both the parabolic and the fifth-order modified KdV equations. However, by using the modified energy \eqref{eq:kwon energy}, which is introduced by Kwon \cite{Kwon2008-1} for the fifth-order KdV equation on $\R$, in addition to the Kato-Ponce type commutator estimate and the Sobolev embedding, we can obtain the energy bound of the solution $u$.  

Furthermore, as another purpose of the Appendix section, we emphasize that the generalized fifth-order modified KdV equation is \emph{unconditionally} locally well-posed for $s > 7/2$. As mentioned, we only use the modified energy and Sobolev embedding to solve the local well-posedness problem. In the sense of the unconditional well-posedness, it is necessary and crucial to construct the modified energy.   

We use symbols $D^s$ and $J^s$ as Fourier multiplier operators defined as
\[\ft_x[D^sf](k) = |n|^s\wh{f}(k), \hspace{1em} \mbox{and} \hspace{1em} \ft_x[J^sf](k) = (1+|n|^2)^{s/2}\wh{f}(k).\]

We consider the non-integrable fifth-order modified KdV equation\footnote{As we know, the ordinary fifth-order modified KdV equation has the quintic non-linear term, but, since one can easily control the this term for $s > \frac32$ by using the Leibnitz rule for fractional derivative, Kato-Ponce commutator estimate and Sobolev embedding, we only consider \eqref{eq:5mkdv6} without quintic term for avoiding complicated calculations.}:
\begin{equation}\label{eq:5mkdv6}
\begin{cases}
\pt u - \px^5 u + c_1u\px u \px^2u + c_2 u^2\px^3u + c_3(\px u)^3 = 0, \hspace{1em} (t,x) \in \R \times \T, \\
u(0,x) = u_0(x) \in H^s(\T),
\end{cases}
\end{equation}
where $c_i$'s are real constants. The following is the main Proposition in Appendix \ref{app:smooth sol}:
\begin{proposition}\label{prop:lwp}
Let $s > \frac72$ and $u_0 \in H^s(\T)$. Then, there is the time $T=T(\norm{u_0}_{H^s})>0$ such that \eqref{eq:5mkdv6} is (unconditionally) locally well-posed in $C([0,T];H^s)$.
\end{proposition}
The argument of proof basically follows that in \cite{Ponce1993} and \cite{Kwon2008-1} associated to the fifth-order KdV equations on $\R$. We prove Proposition \ref{prop:lwp} by following the several steps:

\textbf{Step I.} This step shows the existence of a smooth solution for the perturbed equation. We consider the following parabolic problem:
\begin{equation}\label{eq:5mkdv7}
\pt u - \px^5 u + c_1u\px u \px^2u + c_2 u^2\px^3u + c_3(\px u)^3 = \varepsilon \px^6u,
\end{equation}
where $\varepsilon > 0$. Then, we have
\begin{lemma}\label{lem:A1}
Let $\varepsilon > 0$ be given and $u_0$ be in Schwartz class. Then there is $T_{\varepsilon} > 0$ and a unique solution to \eqref{eq:5mkdv7} in the class
\[\Sch((0,T_{\varepsilon}) \times \T) \cap C([0,T_{\varepsilon}];H^{\infty}).\] 
\end{lemma}
\begin{proof}
The proof follows the argument of R. Temam. See \cite{Temam1969}.
\end{proof}

\textbf{Step II.} This step is to show that there is the time $T$ independent on $\varepsilon$ such that the solution $u^{\varepsilon}$ of \eqref{eq:5mkdv7} provided by Lemma \ref{lem:A1} is in the class $C([0,T];H^{\infty})$ 
by obtaining an \emph{a priori} bound of $u^{\varepsilon}$ in the $C([0,T];H^s)$ norm for $s > \frac72$. From the appropriate \emph{energy estimate} and the standard bootstrap argument, we have the following Lemma:

\begin{lemma}\label{lem:A2}
Let $s > \frac72$. Then, there exists $T = T(\norm{u_0}_{H^s})$ such that for any $\varepsilon>0$, the solution $u^{\varepsilon}$ to \eqref{eq:5mkdv7} provided Lemma \ref{lem:A1} satisfies
\[u^{\varepsilon} \in C([0,T];H^{\infty})\]
and 
\[\sup_{t \in [0,T]}\norm{u^{\varepsilon}}_{H^s} \lesssim \norm{u_0}_{H^s}.\]
\end{lemma}

\textbf{Step III.} This step gives the local well-posedness result for \eqref{eq:5mkdv6} for $s > \frac72$. The following is the conclusion in this step, which exactly implies Proposition \ref{prop:lwp}. 
\begin{lemma}\label{lem:A3}
Let $s > \frac72$. Let $u^{\varepsilon}$ is the smooth solution to \eqref{eq:5mkdv7} provided by Lemma \ref{lem:A1} and \ref{lem:A2}. Then $u^{\varepsilon}$ converges to $u$ in the class $C([0,T];H^s)$ and hence $u$ is the unique solution to \eqref{eq:5mkdv6} in the same class. 
\end{lemma}

Both Ponce \cite{Ponce1993} and Kwon \cite{Kwon2008-1} used the idea of Bona and Smith \cite{BS1978} with energy estimates. The main difficulty in those works is to estimate the energy of solution $u$. Hence we omit the detailed arguments (such as the Bona-Smith argument and bootstrap argument) and finish this section by showing the following energy estimates:
\begin{lemma}\label{lem:A4}
Let $s > \frac72$ and $u(t,x)$\footnote{For the convenience, we use $u$ instead of $u^{\varepsilon}$ as the smooth solution to $\varepsilon$-parabolic equation.} be a Schwartz solution to \eqref{eq:5mkdv7} with sufficiently small $H^s$-norm\footnote{Since the scaling argument still works on the periodic problem, we may assume the smallness of $\norm{u}_{H^s}$.}. Then, there are constants $C_1$ and $C_2$, we have
\begin{equation}\label{eq:A1}
\sup_{t\in[0,T]}\norm{J^su(t)}_{L^2} \le C_1e^{C_2\int_0^t\norm{u(t')}_{H^s}^2 \; dt'}\norm{J^su(0)}_{H^s}.
\end{equation}
\end{lemma}

\begin{remark}
To complete this section, we, in fact, show the energy estimate for the solution $u$ to \eqref{eq:5mkdv6}. But, this exactly follows the proof of Lemma \ref{lem:A4} if one eliminates the $\varepsilon$-terms in the proof of Lemma \ref{lem:A4}, below. 
\end{remark}

To obtain \eqref{eq:A1}, one needs to control the time increment of $\norm{D^su(t)}_{L^2}$ using itself and other norms with the same size. But, since the nonlinear term of \eqref{eq:5mkdv7} has multi-derivatives, the standard energy method gives \footnote{In fact, we have
\[\frac{d}{dt}\norm{D^su}_{L^2}^2 +2\varepsilon \norm{D^{s+3}u}_{L^2}^2\lesssim \norm{\px^3u}_{L^{\infty}}^2\norm{D^su}_{L^2}^2 + \left|\int uu_x D^su_xD^su_x\right|,\]
but this implies \eqref{eq:standard energy} for the smooth solution $u$.}

\begin{equation}\label{eq:standard energy}
\frac{d}{dt}\norm{D^su}_{L^2}^2 \lesssim \norm{\px^3u}_{L^{\infty}}^2\norm{D^su}_{L^2}^2 + \left|\int uu_x D^su_xD^su_x\right|
\end{equation}
and the last term of the right-hand side of \eqref{eq:standard energy} is not favorable. Hence we use the modified energy introduced in its form by Kwon \cite{Kwon2008-1} as the following:
\begin{equation}\label{eq:kwon energy}
E_s(t) := \norm{D^su(t)}_{L^2}^2 + \norm{u(t)}_{L^2}^2 + a_s\int u(t)^2D^{s-2} \px u(t)D^{s-2}\px u(t),
\end{equation}
where the constant $a_s$, which eliminates the bad term in the right-hand side of \eqref{eq:standard energy}, will be chosen later.
  
To prove Lemma \ref{lem:A4}, we only need to show that 
\begin{equation}\label{eq:comparable}
c \norm{u(t)}_{H^s}^2 \le E_s(t) \le C \norm{u(t)}_{H^s}^2,
\end{equation}
for some $c,C>0$, and
\begin{equation}\label{eq:Gronwall}
\frac{d}{dt}E_s(t) \lesssim_s \norm{u(t)}_{H^s}^2E_s(t).
\end{equation}

We begin with introducing \emph{Kato-Ponce commutator estimates}, which are useful tools to prove \eqref{eq:Gronwall}. The followings are the commutator estimates for functions defined on $\R$.
\begin{lemma}[Commutator estimate \cite{KP1988}]\label{lem:kato ponce}
Let $s \ge 1$. Then, we have
\[\norm{[D^s;f]g}_{L^2} \lesssim_s \norm{f_x}_{L^{\infty}}\norm{D^{s-1}g}_{L^2}+\norm{D^sf}_{L^2}\norm{g}_{L^{\infty}},\]
where $[\;\cdot\;;\;\cdot\;]$ is the standard commutator defined as
\[[A;B]C = A(BC) - B(AC).\]
\end{lemma} 

\begin{lemma}[Kwon, Lemma 2.2 in \cite{Kwon2008-1}]\label{lem:general kato ponce0}
Let $s > 0$. Then, we have
\begin{align*}
\Big\|D^s(u \px^3v) - uD^s(\px^3v) - s\px u D^s(\px^2v) - &\frac{s(s-1)}{2}\px^2 uD^s(\px v) \Big\|_{L^2}\\
&\lesssim_s \norm{\px^3u}_{L^{\infty}}\norm{D^sv}_{L^2}+\norm{D^su}_{L^2}\norm{\px^3 v}_{L^{\infty}},
\end{align*}
and
\begin{align*}
\Big\|D^s(\px u \px^2v) - \px uD^s(\px^2v) &- s\px^2 u D^s(\px v) \Big\|_{L^2}\\
&\lesssim_s \norm{\px^3u}_{L^{\infty}}\norm{D^sv}_{L^2}+\norm{D^su}_{L^2}\norm{\px^3 v}_{L^{\infty}}.
\end{align*}
\end{lemma} 
For the proofs of Lemmas \ref{lem:kato ponce} and \ref{lem:general kato ponce0}, we refer \cite{KP1988} and \cite{Kwon2008-1}, respectively.

For the functions defined on $\T$, the corresponding estimates of Lemmas \ref{lem:kato ponce} and \ref{lem:general kato ponce0} can be seen in Lemma 9.A.1 of \cite{IK2007} (see also \cite{TV2014, KP2016}). We, in particular, refer to Lemma 2.5 in \cite{KP2016} for the corresponding version of Lemma \ref{lem:general kato ponce} on $\T$. 

\begin{lemma}[Kenig-Pilod, Lemma 2.5 in \cite{KP2016}]\label{lem:general kato ponce}
Let $s > 0$. Then, we have
\begin{align*}
\Big\|D^s(u \px^3v) - uD^s(\px^3v) - &s\px u D^s(\px^2v) - \frac{s(s-1)}{2}\px^2 uD^s(\px v) \Big\|_{L^2}\\
&\lesssim_s \sum_{j=0}^{3}\norm{\px^ju}_{L^{\infty}}\norm{D^sv}_{L^2}+\sum_{j=0}^{3}\norm{\px^j v}_{L^{\infty}}\norm{D^su}_{L^2},
\end{align*}
and
\begin{align*}
\Big\|D^s(\px u \px^2v) - \px uD^s(\px^2v) &- s\px^2 u D^s(\px v) \Big\|_{L^2}\\
&\lesssim_s \sum_{j=0}^{3}\norm{\px^ju}_{L^{\infty}}\norm{D^sv}_{L^2}+\sum_{j=0}^{3}\norm{\px^j v}_{L^{\infty}}\norm{D^su}_{L^2}.
\end{align*}
\end{lemma} 
Note that not only $\norm{\px^3u}_{L^{\infty}}$ and $\norm{\px^3v}_{L^{\infty}}$ terms but also $\norm{\px^ju}_{L^{\infty}}$ and $\norm{\px^jv}_{L^{\infty}}$ terms, $j=0,1,2$, appears in the right-hand side of estimates in Lemma \ref{lem:general kato ponce} compared with Lemma \ref{lem:general kato ponce0}, but latter terms are negligible in some sense thanks to the Sobolev embedding theorem.

\begin{proof}[Proof of \eqref{eq:comparable}.]
We use the H\"older inequality and the Sobolev embedding to the third term in $E_s(t)$ in order that $E_s(t)$ is bounded by $|a_s|C\norm{u}_{H^s}^4$. Then, we have \eqref{eq:comparable}, when $\norm{u}_{H^s}^2 \le \frac{1}{2C|a_s|}$.
\end{proof}

\begin{proof}[Proof of \eqref{eq:Gronwall}.]
The standard energy method to \eqref{eq:5mkdv6} yields
\begin{equation}\label{eq:A2}
\begin{split}
\frac12\frac{d}{dt}\norm{D^su}_{L^2}^2 &= -\varepsilon\norm{D^{s+3}u}_{L^2}^2 -c_1\int D^s(u \px u \px^2 u)D^su \\
&-c_2 \int D^s(u^2 \px^3 u)D^su  -c_3 \int D^s(\px u\px u\px u)D^su.
\end{split}
\end{equation}
We note that
\begin{equation}\label{eq:note1}
\px u\px u\px u = \px (u \px u\px u) - 2(u\px u\px^2 u),
\end{equation}
\begin{equation}\label{eq:note2}
u \px u \px^2 u = \frac12 \px(u^2)\px^2 u.
\end{equation}
Then, \eqref{eq:A2} can be rewritten from \eqref{eq:note1} and \eqref{eq:note2} that
\begin{align*}
\frac12\frac{d}{dt}\norm{D^su}_{L^2}^2 &= -\varepsilon\norm{D^{s+3}u}_{L^2}^2 +\frac{2c_3-c_1}{2}\int D^s(\px(u^2) \px^2 u)D^su \\
&-c_2 \int D^s(u^2 \px^3 u)D^su  -c_3 \int D^s\px(u\px u\px u)D^su.
\end{align*}
In order to use Lemma \ref{lem:general kato ponce}, we add and subtract some terms, and then
\begin{align*}
\frac12\frac{d}{dt}\norm{D^su}_{L^2}^2 =&~{} -\varepsilon\norm{D^{s+3}u}_{L^2}^2\\
&+ \frac{2c_3-c_1}{2}\int \left[D^s(\px(u^2) \px^2 u) - \px(u^2)D^s(\px^2 u) - s\px^2(u^2)D^s(\px u)\right]D^su \\
&-c_2 \int \left[D^s(u^2 \px^3 u)-u^2D^s(\px^3 u) -s\px(u^2)D^s(\px^2 u) - \frac{s(s-1)}{2}\px^2(u^2)D^s(\px u)\right]D^su \\
&-\int \left[c_3D^s\px(u\px u\px u) + c_4\px(u^2)D^s(\px^2 u) + c_5\px^2(u^2)D^s(\px u)\right]D^su\\
=&~{}: I+II+III+IV.
\end{align*}
We use Lemma \ref{lem:general kato ponce}, product rule for fractional derivative and Sobolev embedding to estimate terms of $II$ and $III$ to have
\[|II+III| \lesssim \norm{u}_{H^s}^2\norm{D^su}_{L^2}^2.\]
We also perform the integration by parts to $VI$ to obtain
\[VI = d_1\int u \px^3 u D^s u D^s u +d_2 \int \px u \px^2 u D^s u D^s u +d_3 \int \px(u^2) D^s \px u D^s \px u.\]

On the other hand, taking the time derivative to the third term in $E_s(t)$ yields
\begin{align*}
\frac{d}{dt}& \int u^2 D^{s-2}\px uD^{s-2}\px u\\
=&~{} 2\int uu_t D^{s-2}\px uD^{s-2}\px u + 2\int (u^2) D^{s-2}\px u_tD^{s-2}\px u \\
=&-2\int u\px^5u D^{s-2}\px uD^{s-2}\px u - 2\int (u^2) D^{s-2}\px^6 uD^{s-2}\px u \;(=:A)\\
&+2\varepsilon\int u\px^6u D^{s-2}\px uD^{s-2}\px u + 2\varepsilon\int (u^2) D^{s-2}\px^7 uD^{s-2}\px u\;(=:B)\\
&-2c_1\int u(u\px u\px^2 u) D^{s-2}\px uD^{s-2}\px u + (u^2) D^{s-2}\px (u\px u\px^2 u) D^{s-2}\px u\\
&-2c_2\int u(u^2\px^3 u) D^{s-2}\px uD^{s-2}\px u + (u^2) D^{s-2}\px (u^2\px^2 u) D^{s-2}\px u\\
&-2c_3\int u(\px u\px u\px u) D^{s-2}\px uD^{s-2}\px u + (u^2) D^{s-2}\px (\px u\px u\px u) D^{s-2}\px u\\
=&:A+B+C.
\end{align*}
From the following observation
\begin{align*}
\px^5(fg^2) =&~{} f_{xxxxx}g^2 + 10f_{xxxx}gg_x + 20 f_{xxx}g_xg_x + 20f_{xxx}gg_{xx} \\
&+ 60f_{xx}g_xg_{xx}+20f_{xx}gg_{xxx}+30f_xg_{xx}g_{xx} + 40f_xg_xg_{xxx} \\
&+ 10 f_xgg_{xxxx} + 20fg_{xx}g_{xxx} +10fg_xg_{xxxx} + 2fgg_{xxxxx},
\end{align*}
putting $f = u$ and $g = D^{s-2}\px u$ with performing the integration by parts several times yields
\begin{align*}
A &=\alpha_1\int u\px^3uD^{s-2}\px uD^{s-2}\px^3 u +\alpha_2\int u\px^2uD^{s-2}\px uD^{s-2}\px^4 u\\
&+\alpha_3\int u\px uD^{s-2}\px^3 uD^{s-2}\px^3 u +\alpha_4\int \px u\px^2uD^{s-2}\px^2 uD^{s-2}\px^2 u\\
&+\alpha_5\int u\px^3uD^{s-2}\px^2 uD^{s-2}\px^2 u +\alpha_6\int \px u\px^3uD^{s-2}\px^2 uD^{s-2}\px u\\
&+\alpha_7\int \px^2 u\px^2uD^{s-2}\px^2 uD^{s-2}\px u.
\end{align*}
For the first two bad terms, by direct calculation, we can easily know $\alpha_1 = \alpha_2$. Hence, we perform the integration by parts again to  
\[\alpha_1\int u\px^3uD^{s-2}\px uD^{s-2}\px^3 u,\]
the one of distributed terms exactly cancels out the term
\[\alpha_2\int u\px^2uD^{s-2}\px uD^{s-2}\px^4 u.\]
The rest of distributed terms be absorbed to the other four terms, thus we finally obtain
\begin{align*}
A &=\beta_1\int u\px uD^{s-2}\px^3 uD^{s-2}\px^3 u +\beta_2\int \px u\px^2uD^{s-2}\px^2 uD^{s-2}\px^2 u\\
&+\beta_3\int u\px^3uD^{s-2}\px^2 uD^{s-2}\px^2 u +\beta_4\int \px u\px^3uD^{s-2}\px^2 uD^{s-2}\px u\\
&+\beta_5\int \px^2 u\px^2uD^{s-2}\px^2 uD^{s-2}\px u.
\end{align*}
Once we choose $a_s$ such that $a_s \cdot \beta_1 + d_3 = 0$, we get
\[|VI + A| \lesssim \norm{u}_{H^s}^2\norm{D^su}_{L^2}^2.\]

Now we concentrate on the term $C$. By using the integration by parts, the H\"older inequality, Lemma \ref{lem:kato ponce} and the Sobolev embedding to estimate term $C$, we can easily obtain
\[|C| \lesssim \norm{u}_{H^s}^2\norm{D^su}_{L^2}^2,\]
when $\norm{u}_{H^s} \le 1$, for instance
\begin{align*}
\Big|\int u^2 &D^{s-2}(u \px u \px^3 u) D^{s-2} \px u \Big|\\
\lesssim& \left|\int u^2 [[D^{s-2};u \px u] \px^3 u] D^{s-2} \px u \right| +\left|\int u^2 u\px uD^{s-2}\px^3 u D^{s-2} \px u \right|\\
\lesssim& \norm{u^2}_{L^{\infty}}\norm{[D^{s-2};u \px u] \px^3 u}_{L^2}\norm{D^{s-2} \px u}_{L^2} \\
&+\left|\int \px(u^2 u\px u)D^{s-2}\px^2 u D^{s-2} \px u \right|+\left|\int u^2 u\px uD^{s-2}\px^2 u D^{s-2} \px^2 u \right|\\
\lesssim& \norm{u}_{H^s}^4\norm{D^su}_{L^2}^2.
\end{align*}
We finally consider the term $B$. From the integration by parts, we have the following observations:
\begin{align*}
+2\varepsilon\int u\px^6u D^{s-2}\px uD^{s-2}\px u =&-2\varepsilon\int \px^3u\px^3uD^{s-1}uD^{s-1}u -12\varepsilon\int \px^2u\px^3uD^{s}uD^{s-1}u\\
&-8\varepsilon\int \px u\px^3uD^{s}uD^{s}u -12\varepsilon\int \px u\px^3uD^{s-1}uD^{s+1}u\\
&-4\varepsilon\int u\px^3uD^{s}uD^{s+1}u -4\varepsilon\int u\px^3uD^{s-1}uD^{s+2}u
\end{align*}
and
\begin{align*}
2\varepsilon\int (u^2) D^{s-2}\px^7 uD^{s-2}\px u =&2\varepsilon\int D^{s+3}u\px^2[u^2D^{s-1}u] +4\varepsilon\int \px u \px uD^{s+3}uD^{s-1}u\\
&+4\varepsilon\int u \px^2 uD^{s+3}uD^{s-1}u +8\varepsilon\int u \px uD^{s+3}uD^s u\\
&+2\varepsilon\int u^2D^{s+3}uD^{s+1}u.
\end{align*}
From the Sobolev embedding, the last three terms in the first observation and all terms in the second observations are dominated by
\[38\varepsilon K\norm{u}_{H^s}^2\norm{D^{s+3}u}_{L^2}^2,\]
where the constant $K$ only appears from the Sobolev embedding. Furthermore, the other terms in the first observation can be easily treated by the Sobolev embedding.

Gathering all things yields
\[\varepsilon\norm{D^{s+3}u}_{L^2}^2 + \frac{d}{dt}E_s(t) \lesssim_s \norm{u}_{H^s}^2E_s(t),\]
when $\displaystyle\norm{u}_{H^s}^2 \le \frac{1}{38K|a_s|}$, and hence we conclude from the Gronwall's inequality that
\[E_s(t) \lesssim_s e^{\int_0^t\norm{u(t')}_{H^s}^2 \; dt'}E_s(0),\]
which complete the proof.
\end{proof}


\begin{thebibliography}{alpha}
\bibitem{AS1981} M. Ablowitz, H. Segur, \emph{Solitons and inverse scattering transform}, in "SIAM Studies in Applied Mathematics." Vol. 4, SIAM, Philadelphia, 1981. 

\bibitem{BS1978} J. L. Bona, R. S. Smith, \emph{The initial value problem for the Korteweg-de Vries equation}, Proc. Roy. Soc. London Ser. A 278 (1978) 555--601.

\bibitem{Bourgain1993} J. Bourgain, \emph{Fourier transform restriction phenomena for certain lattice subsets and applications to nonlinear evolution equations. Parts I, II}, Geom. Funct. Anal. 3 (1993) 107--156, 209--262.

\bibitem{CCT2008} M. Christ, J. Colliander, T. Tao, \emph{A priori bounds and weak solutions for the nonlinear Schr\"odinger equation in Sobolev space of negative order}, Journal of Functional Analysis 254 (2008), 368--395.

\bibitem{CKSTT2004} J. Colliander, M. Keel, G. Staffilani, H. Takaoka, T. Tao, \emph{Multilinear estimates for periodic KdV equations, and applications}, J. Funct. Anal. 211 (2004) 173--218.

\bibitem{CTD2009} A. Choudhuri, B. Talukdar, U. Das, \emph{The Modified Korteweg-de Vries Hierarchy: Lax Pair Representation and Bi-Hamiltonian Structure}, Zeitschrift f\"ur Naturforschung A 64 (2009) 171--179.

\bibitem{GGKM1967} C. Gradner, J. Greene, M. Kruskal, R. Miura, \emph{A method for solving the Korteweg-de Vries equation}, Phys. Rev. Lett. 19 (1967) 1095--1097.

\bibitem{Guo2011} Z. Guo, \emph{Local well-posedness and a priori bounds for the modified Benjamin-Ono equation}, Advances in Differential Equations, 16/11-12 (2011), 1087--1137.

\bibitem{Guo2012} Z. Guo, \emph{Local well-posedness for dispersion generalized Benjamin-Ono equations in Sobolev spaces}, J. Differential Equations 252 (2012) 2053--2084. 

\bibitem{GPWW2011} Z. Guo, L. Peng, B. Wang, Y. Wang, \emph{Uniform well-posedness and inviscid limit for the Benjamin-Ono-Burgers equation}, Advances in Mathematics 228 (2011) 647--677.

\bibitem{GKK2013} Z. Guo, C. Kwak, S. Kwon, \emph{Rough solutions of the fifth-order KdV equations}, J. Funct. Anal. 265 (2013) 2791--2829.

\bibitem{GO2015} 	Z. Guo, T. Oh, \emph{Non-existence of solutions for the periodic cubic NLS below $L^2$}, to appear in Internat. Math. Res. Not.

\bibitem{IK2007}  A. Ionescu, C. Kenig, \emph{Local and global well-posedness of periodic KP-I equations}, “Mathematical Aspects of Nonlinear Dispersive Equations”, Ann. Math. Stud., 163, Princeton University Press, (2007) 181--212. 

\bibitem{IKT2008} A. Ionescu, C. Kenig, D. Tataru, \emph{Global well-posedness of the KP-I initial-value problem in the energy space}, Invent. Math. 173 (2) (2008) 265--304.

\bibitem{KP2003} T. Kappeler, J. P\"oschel, \emph{On the Korteweg-de Vries equation and KAM theory}, Geometric analysis and nonlinear partial differential equations, Springer, Berlin (2003) 397--416.

\bibitem{KT2005} T. Kappeler, P. Topalov, \emph{Global wellposedness of mKdV in $L^2(\T,\R)$}, Comm. Partial Differential Equations 30 (2005) 435--449. 

\bibitem{KT2006} T. Kappeler, P. Topalov, \emph{Global wellposedness of KdV in $H^{-1}(\T,\R)$}, Duke Math. J. 135 (2006) 327--360. 

\bibitem{KP1988} T. Kato, G. Ponce, \emph{Commutator estimates and the Euler and Navier–Stokes equations}, Comm. Pure Appl. Math. 41 (7) (1988) 891--907.

\bibitem{KPV1991} C. Kenig, G. Ponce, L. Vega, \emph{Oscillatory integrals and regularity of dispersive equations}, Indiana U. Math. J 40 (1991) 33--69.

\bibitem{KPV1996} C. Kenig, G. Ponce, L. Vega, \emph{A bilinear estimate with applications to the KdV equation}, J. Amer. Math. Soc. 9 (1996) 573--603.

\bibitem{KP2015} C. Kenig, D. Pilod, \emph{Well-posedness for the fifth-order KdV equation in the energy space}, Trans. Amer. Math. Soc. 367 (2015) 2551-2612.

\bibitem{KP2016} C. Kenig, D. Pilod, \emph{Local well-posedness for the KdV hierarchy at high regularity}, Adv. Diff. Eq., 21 (2016), 801--836.

\bibitem{KT2007} H. Koch, D. Tataru, \emph{A priori bounds for the 1D cubic NLS in negative Sobolev spaces}, Int. Math. Res. Not. IMRN 16 (2007), Art. ID rnm053, 36, DOI 10.1093/imrn/rnm053. MR2353092 (2010d:35307)

\bibitem{Kwak2016} C. Kwak \emph{Local well-posedness for the fifth-order KdV equations on $\T$}, J. Differential Equations 260 (2016) 7683--7737. http://dx.doi.org/10.1016/j.jde.2016.02.001.

\bibitem{Kwon2008-1} S. Kwon, \emph{On the fifth order KdV equation: Local well-posedness and lack of uniform continuity of the solution map} J. Differential Equations, 245 (2008) 2627--2659.
    
\bibitem{Kwon2008-2} S. Kwon, \emph{Well-posedness and ill-posedness of the fifth-order modified KdV equation}, Electron. J. Differential Equations, 2008 (2008) 1--15.

\bibitem{Lax1968} P. Lax, \emph{Integrals of nonlinear equations of evolution and solitary waves}, Comm. Pure Appl. Math 21 (1968) 467--490.

\bibitem{Linares1995} F. Linares, \emph{A higher order modified Korteweg-de Vries equation}, Comp. Appl. Math. 14 (3) 35--49.

\bibitem{Magri1978} F. Magri, \emph{A simple model of the integrable Hamiltonian equation}, J. Math. Phys. 19 (1978) 1156--1162.

\bibitem{Miura1968} R. Miura, \emph{Korteweg-de Vries equation and generalizations, I. A remarkable explicit nonlinear transformation}, J. Math. Phys. 9 (1968) 1202--1204.

\bibitem{Ponce1993} G. Ponce, \emph{Lax pairs and higher order models for water waves}, J. Differential Equations 102 (2) (1993) 360--381.

\bibitem{Staffilani1997} G. Staffilani, \emph{On solutions for periodic generalized KdV equations}, IMRN 18 (1997) 899--917.

\bibitem{Tao2001} T. Tao, \emph{Multilinear weighted convolution of $L^2$ functions and applications to nonlinear dispersive equations}, Amer. J. Math. 123 (5) (2001) 839--908.

\bibitem{Tao2006} T. Tao, \emph{Nonlinear Dispersive Equations : Local and Global Analysis}, CBMS Reg. Conf. Ser. Math. vol.106 (2006).

\bibitem{Temam1969} R. Temam, \emph{sur un probl$\grave{e}$m non lin$\acute{e}$aire}, J. Math. Pures Appl. 48 (1969) 159--172. 

\bibitem{TV2014}  N. Tzvetkov, N. Visciglia, \emph{Invariant measures and long time behavior for the Benjamin-Ono equation}, Int. Math. Res. Not., 17 (2014), 4679--4714.
\end{thebibliography}
\end{document}